\documentclass[12pt]{amsart}
\usepackage{graphicx}
\usepackage{amsmath, enumerate}
\usepackage{amsfonts, amssymb, amsthm, xcolor, stmaryrd}
\usepackage{tikz, tikz-cd, caption, tabu, longtable, mathrsfs, xtab, centernot, mathtools}
\usepackage{leftidx}
 \usepackage{dsfont, cite, todonotes}
\usepackage[toc,page]{appendix}
\usepackage{hyperref}

\numberwithin{equation}{subsection}
\theoremstyle{plain}

\newtheorem{thm}{Theorem}[section]
\newtheorem{lemma}[thm]{Lemma}
\newtheorem{prop}[thm]{Proposition}
\newtheorem{cor}[thm]{Corollary}
\newtheorem{thma}{Theorem}

\theoremstyle{definition}

\newtheorem{conj}{Conjecture}
\newtheorem{exmp}[thm]{Example}

\theoremstyle{remark}
\newtheorem{rmk}[thm]{Remark}

\usepackage{leftidx, mathtools}

\newcommand{\Hb}{\mathbb{H}}
\newcommand{\SL}{{\mathrm{SL}}}
\newcommand{\Zb}{\mathbb{Z}}
\newcommand{\Cb}{\mathbb{C}}
\newcommand{\Qb}{\mathbb{Q}}
\newcommand{\Nb}{\mathbb{N}}
\newcommand{\lp}{\left (}
\newcommand{\rp}{\right )}
\newcommand{\pf}{\mathfrak{p}}
\newcommand{\Gal}{{\mathrm{Gal}}}
\newcommand{\varep}{\varepsilon}
\newcommand{\Mp}{{\mathrm{Mp}}}
\newcommand{\Rb}{\mathbb{R}}
\newcommand{\smat}[4]{\left(\begin{smallmatrix}
                 #1 & #2\\
                 #3 & #4
\end{smallmatrix}\right)}
\newcommand{\pmat}[4]{\begin{pmatrix}
                 #1 & #2\\
                 #3 & #4
\end{pmatrix}}
\newcommand{\GL}{{\mathrm{GL}}}
\newcommand{\half}{{\tfrac{1}{2}}}
\newcommand{\Ac}{{\mathcal{A}}}
\newcommand{\ef}{\mathfrak{e}}
\newcommand{\ebf}{{\mathbf{e}}}
\newcommand{\Hc}{{\mathcal{H}}}
\newcommand{\Ss}{\mathscr{S}}
\newcommand{\Dc}{{\mathcal{D}}}
\newcommand{\SO}{{\mathrm{SO}}}
\newcommand{\Aut}{{\mathrm{Aut}}}
\newcommand{\zbar}{{\overline{z}}}

\newcommand{\ua}{{\underline{a}}}
\newcommand{\ub}{{\underline{b}}}
\newcommand{\ur}{{\underline{r}}}
\newcommand{\us}{{\underline{s}}}

\newcommand{\Aco}{A}

\newcommand{\Br}{\mathrm{B}}
\newcommand{\wtM}{\widetilde{M}}
\newcommand{\thetab}{\boldsymbol{\theta}}

\newcommand{\transpose}[1]{{\prescript{t}{}{#1}}}
\newcommand{\domega}{d\omega}

\newcommand{\tp}{\tilde{p}}
\newcommand{\hp}{\check{p}}

\newcommand{\sr}{\mathrm{s}}
\newcommand{\Cr}{\mathrm{C}}
\newcommand{\crm}{\mathrm{c}}
\newcommand{\rt}{\mathrm{sqrt}}
\newcommand{\Fc}{{\mathcal{F}}}
\newcommand{\Oc}{\mathcal{O}}
\newcommand{\df}{\mathfrak{d}}
\newcommand{\Cl}{{\mathrm{Cl}}}
\newcommand{\af}{\mathfrak{a}}
\newcommand{\Nm}{{\mathrm{Nm}}}
\newcommand{\ord}{{\mathrm{ord}}}
\newcommand{\bfrak}{\mathfrak{b}}
\newcommand{\sgn}{\mathrm{sgn}}
\newcommand{\Bf}{\mathfrak{B}}
\newcommand{\Af}{\mathfrak{A}}
\newcommand{\kro}[2]{\left( \tfrac{#1}{#2} \right)}
\newcommand{\lf}{\mathfrak{l}}
\newcommand{\Ec}{{\mathcal{E}}}
\newcommand{\tth}{\textsuperscript{th }}
\newcommand{\tr}{\mathrm{tr}}
\newcommand{\slf}{{\mathfrak{sl}}}
\newcommand{\spf}{{\mathfrak{sp}}}
\newcommand{\Sp}{\mathrm{Sp}}
\newcommand{\Wb}{\mathbb{W}}
\newcommand{\zf}{\mathfrak{z}}
\newcommand{\wf}{\mathfrak{w}}

\AtBeginDocument{%
   \def\MR#1{}
}
\begin{document}
\title{Average CM-values of Higher Green's Function and Factorization}
\author[Yingkun Li]{Yingkun Li}
\address{Fachbereich Mathematik,
Technische Universit\"at Darmstadt, Schlossgartenstrasse 7, D--64289
Darmstadt, Germany}
\email{li@mathematik.tu-darmstadt.de}
\thanks{The author is partially supported by the DFG grant BR-2163/4-2, an NSF postdoctoral fellowship and the LOEWE research unit USAG.}

\date{\today}
\maketitle

\begin{abstract}
In this paper, we prove an averaged version of an algebraicity conjecture in \cite{GKZ87} concerning the values of higher Green's function at CM points. Furthermore, we give the factorization of the ideal generated by such algebraic value in the spirit of the famous work of Gross and Zagier on singular moduli.
\end{abstract}
\tableofcontents
\section{Introduction}

Let $j(z)$ be the modular $j$-invariant on the modular curve $X_\Gamma := \Gamma \backslash \Hb$ with $\Gamma := \SL_2(\Zb)$ and $\Hb$ the upper-half complex plane. 
Its value at a CM point $z \in \Hb$, i.e.\ an algebraic number in an imaginary quadratic field $K \subset \Cb$, is called a singular moduli. 
The theory of complex multiplication tells us such a singular moduli is algebraic and generates an abelian extension of $K$. This is the extension of the well-known theorem of Kronecker-Weber from $\Qb$ to $K$, and part of Kronecker's Jugendtraum. 

On the other hand, the modular $j$-invariant also plays an important role in arithmetic intersection theory as the function $G_1(z_1, z_2) := 2\log|j(z_1) - j(z_2)|$ is the automorphic Green's function on $X_\Gamma^2$ with logarithm singularity along the diagonal $T_1 \subset X_\Gamma^2$. Values of $G_1$ at a divisor $Z$ on $X_\Gamma^2 \backslash T_1$ then give the archimedean part of the arithmetic intersection between $T_1$ and $Z$. When $Z = Z_\chi$ is the average of all pairs of CM points with discriminants $d_1$ and $d_2$ (see \eqref{eq:Zchi}), this arithmetic intersection is related to the Fourier coefficients of an incoherent Eisenstein series associated to $Z_\chi$.
This is the essential ingredient behind the analytic proof of the factorization of the norm of difference of singular moduli by Gross and Zagier in \cite{GZ85}.

The Green's function $G_1$ is a member of a family of \textit{higher Green's functions} $G_k$ on $X_\Gamma^2 \backslash T_1$ for $k \in \Nb$ (see \eqref{eq:Gs}), which are eigenfunctions under the hyperbolic Laplacians of both $z_1$ and $z_2$ with the same eigenvalue $k(1-k)$. 
They also play an important role in the calculation of arithmetic intersections by giving the archimedean height paring of CM-cycles in Kuga-Sato varieties \cite{Zhang97}.
Even though complex multiplication does not apply to special values of higher Green's functions, one still expects them to be of algebraic nature.
In \cite{GZ86}, Gross and Zagier made the following conjecture.
\begin{conj}
  \label{conj:gkz}
Let $G_{k, f}$ be the higher Green's function associated to a weakly holomorphic modular form $f \in M^!_{2-2k}$  for $k \ge 2$ (see \eqref{eq:Gkf}).
Suppose $f$ has rational Fourier coefficients.
For a CM point $Z = (z_1, z_2) \in X_\Gamma^2$ not on the singularity $T_f$ of $G_{k, f}$ (see \eqref{eq:Tf}), there exists $\alpha = \alpha_{\chi, f} \in \overline{\Qb} \subset \Cb$ such that  
$$
 G_{k, f}(Z) = (d_1d_2)^{\tfrac{1-k}{2}} \log |\alpha|,
$$
where $d_j$ is the discriminant of $z_j$.
\end{conj}

The first instance of this conjecture was proved by Gross, Kohnen and Zagier \cite{GKZ87}, where they consider the value of $G_{k, f}$ at $Z_\chi$ for $k$ \textit{odd} and $\sqrt{\Delta } \not \in \Qb$ with $\Delta := d_1 d_2$. In this case, $\alpha$ can be taken to be a positive rational number, and has an explicit factorization generalizing \cite{GZ85}. 
To state this result, let $F$ be the real quadratic field $\Qb(\sqrt{\Delta}) \subset \Rb$ with ring of integers $\Oc_F$. Denote $\chi$ the genus character corresponding to the CM extension $K := \Qb(\sqrt{d_1}, \sqrt{d_2})$ over $F$. 
For each integral ideal $\af \subset \Oc_F$, we can attach an integral ideal $I^+_\chi(\af) \subset \Oc_F$ given by
\begin{equation}
  \label{eq:I+}
  I^+_\chi(\af) := \prod_{\bfrak \mid \af} (\bfrak \cdot \bfrak')^{\chi(\bfrak)},
\end{equation}
where $'$ denotes the real conjugation in $\Gal(F/\Qb)$. Notice that $I^+_\chi(\af)$ is $\Gal(F/\Qb)$-invariant, hence principal, and has a unique positive integer as its generator. If $\af = (\lambda)$ is principal, we simply write $I^+_\chi(\lambda)$ for $I^+_\chi(\af)$.

For each positive integer $m$, we have a finite set $\mathrm{S}_m$ of elements in $\Oc_F$ defined by
\begin{equation}
  \label{eq:Sm}
  \mathrm{S}_m := \{\lambda \in \Oc_F: \lambda \sqrt{\Delta} \gg 0,~ \lambda \in \tfrac{1}{2} \Zb + \tfrac{m}{2} \sqrt{\Delta} \}.
\end{equation}
For each $\lambda \in \Oc_F$ and integers $m \in \Zb, k \ge 1$, denote 
\begin{equation}
  \label{eq:Ck}
  C_k(\lambda, m) := -(-m \sqrt{\Delta})^{k-1} P_{k-1} \lp \frac{\tr(\lambda)}{m \sqrt{\Delta}} \rp \in \Zb,
\end{equation}
where $P_{k-1}(t)$ is the $(k-1)$\tth Legendre polynomial (see \eqref{eq:Taylor}). 
The result of Gross-Zagier and Gross-Kohnen-Zagier can be phrased as follows.

\begin{thma}[Theorem 1.3 in \cite{GZ85} and Theorem 2 in Section V of \cite{GKZ87}]
\label{thm:GKZ}
Suppose $d_1, d_2 < 0$ are co-prime, fundamental discriminants. For odd $k \ge 1$ and $f \in M^!_{2-2k}$ with integral Fourier coefficients $c_f(n)$, let $\alpha^+ = \alpha^+_{\chi, f} \in \Qb$ be the unique positive generator of the ideal 
$$
\prod_{m \ge 1} \prod_{\lambda \in \mathrm{S}_m} I^+_\chi(\lambda)^{c_f(-m) C_k(\lambda, m)}.
$$
Then $ G_{k, f}(Z_\chi) = \Delta^{(1-k)/2} \log (\alpha^+_{})$. 
\end{thma}

\begin{rmk}
  When $k = 1$, the result above is the factorization of norm of difference of singular moduli by Gross and Zagier.
\end{rmk}

Since the work of Gross-Kohnen-Zagier, there has been a lot of progress proving conjecture \ref{conj:gkz}. 
When the discriminants $\Delta = d_1d_2$ is a perfect square, this was proved by Zhang \cite{Zhang97} conditional on the positivity of the height parings of Heegner cycles on Kuga-Sato varieties. 
Later in \cite{Viazovska11}, Viazovska gave an unconditional proof using the machinery of theta liftings. 
When $\Delta$ is not a perfect square, i.e.\ $F$ is a real quadratic field, Mellit proved the conjecture for $k = 2$, $z_1 = i$, $z_2$ arbitrary in his thesis \cite{Mellit08}. 
In a recent work \cite{BEY}, the conjecture is proved when one averages over all CM points $(z_1, z_2)$ with $z_1$ having discriminant $d_1$. 
Most recently, conjecture \ref{conj:gkz} is proved in \cite{Li21} for any CM point $Z = (z_1, z_2)$ and \textit{odd} $k \ge 3$. 
In these cases, the algebraic number $\alpha$ is in an abelian extension of an imaginary quadratic field. In general, one expects $\alpha$ to be in an abelian extension of the CM field $K$.

With the (sometimes conjectural) knowledge of the algebraicity of $\alpha$, it is very natural to ask for the factorization of the ideal generated by it.
For \textit{even} $k \ge 2$, this is not known even when one takes the average $Z = Z_\chi$. 
The goal of this work is to provide a result in this setting, which complements Theorem \ref{thm:GKZ} dealing with the cases that $k$ is odd. 
The subtlety created by the parity of $k$ is made clear in the following example by Mellit \cite{Mellit08}
$$
G_2 \lp i, \frac{-1 + \sqrt{-7}}{2} \rp = \frac{8}{\sqrt{4\cdot 7}} \log \frac{8 - 3\sqrt{7}}{8 + 3\sqrt{7}},
$$
where $d_1 = -4, d_2 = -7$. 
The algebraic number $\frac{8 - 3\sqrt{7}}{8 + 3\sqrt{7}}$ is not rational. Furthermore, it is a unit and its norm to $\Qb$ is 1.
It turns out that the right modification for even $k$ is to replace the integral ideal $I^+_\chi(\af)$ with the fractional ideal
\begin{equation}
  \label{eq:I-}
  I^-_\chi(\af) := \prod_{\bfrak \mid \af} \lp \frac{\bfrak}{ \bfrak' } \rp^{\chi(\bfrak)}
\end{equation}
for every integral ideal $\af \subset \Oc_F$. 
Since $I^-_\chi(\af)' = I^-_\chi(\af)^{-1}$, the ideal $I^-_\chi(\af)$ is not necessarily principal.
One can overcome this by raising it to a suitable power (for example the class number of $F$).
With this minor difference, we can now state the analog of Theorem \ref{thm:GKZ} for even $k$ as follows.
\begin{thma}[Theorem \ref{thm:main}]
  \label{thma:main1}
Suppose $d_1, d_2 < 0$ are co-prime, fundamental discriminants.
For even $k \ge 2$ and $f \in M^!_{2-2k}$ with integral Fourier coefficients $c_f(n)$, there exists $\kappa = \kappa_{f, F} \in \Nb$ and a positive generator $\alpha^- = \alpha^-_{\chi, f, \kappa} \in F$ of the fractional ideal
$$
\prod_{m \ge 1} \prod_{\lambda \in \mathrm{S}_m} I^-_\chi(\lambda)^{c_f(-m) C_k(\lambda, m)\kappa}
$$
such that $\kappa \cdot G_{k, f}(Z_\chi) = \Delta^{(1-k)/2} \log (\alpha^-_{})$. 
\end{thma}

\begin{rmk}
Even though the result is an existence statement, the algebraic number $\alpha^-$ is \textit{explicitly} constructed. 
  This theorem then also gives a constructive proof of Conjecture \ref{conj:gkz} when $Z_\chi$ consists of just one point. 
\end{rmk}

To give a numerical example, we take $k = 4$ and $f(\tau) = q^{-1} + O(1) \in M_{-6}$. Such $f$ exists (and is unique) since there is no non-trivial cusp form in $S_8$. If we take the discriminants $-7$ and $-23$, then the CM points are $(z_7, z_{23, 0}), (z_7, z_{23, 1})$ and $(z_7, z_{23, -1})$, where $z_d = z_{d, 0} := \frac{1 + \sqrt{-d}}{2}$ and $z_{23, \pm 1} = \frac{\pm 1 + \sqrt{-23}}{4}$. 
Theorem \ref{thma:main1} above gives us $\alpha^- \in F$ and $\kappa \in \Nb$ such that
$$
 \sum_{j = -1, 0, 1} G_{4, f}(z_7, z_{23, j}) = \frac{1}{{161}^{3/2} \kappa} \log (\alpha^-).
$$
The factorization tells us that $\alpha^-$ generates the fractional ideal 
$$
\lp 
\lp \frac{\pf_5}{\pf_5'} \rp^{-2878} 
\lp \frac{\pf_{17}}{\pf_{17}'}\rp^{-3580} 
\lp \frac{\pf_{19}}{\pf_{19}'}\rp^{2628}
\rp^\kappa,
$$
where $\pf_\ell = (\pi_\ell)$ are prime ideals in $F = \Qb(\sqrt{161})$ with 
$$\pi_5 = 38 + 3\sqrt{161}, \pi_{17} = 12 + \sqrt{161}, \pi_{19} = 25 + 2\sqrt{161}.$$ 
Numerically, we have computed the left hand side to be $-4.157888612785..$ and bounded $\kappa$ explicitly.
This helps us to see that we can take $\kappa = 1$ and 
$$
\alpha^- = 
\lp \frac{\pi_5}{\pi_5'} \rp^{-2878} 
\lp \frac{\pi_{17}}{\pi_{17}'}\rp^{-3580} 
\lp \frac{\pi_{19}}{\pi_{19}'}\rp^{2628} 
\varep_F^{1168},
$$
where $\varep_F = 11775 + 928\sqrt{161}$ is the fundamental unit of $F$.

In \cite{GKZ87}, the authors proved Theorem \ref{thm:GKZ} by expressing the special value of higher Green's function at CM points as linear combinations of the holomorphic part Fourier coefficients of certain real-analytic Eisenstein series, which are \textit{incoherent} in the sense of Kudla \cite{Kudla97}. These Fourier coefficients can be computed exactly and are rational multiples of logarithms of integers. 

Comparing the theorems above, it is then clear that one does not expect the value $G_{k, f}(Z_\chi)$ to be directly related to these Fourier coefficients when $k$ is even. In fact, the method in \cite{GKZ87} applies to all $k \ge 1$ with $G_{k, f}$ replaced by a function which is odd under the Atkin-Lehner involution when $k$ is even (see page 502 in \cite{GKZ87}). For level one, this returns the trivial equation. 

To prove Theorem \ref{thma:main1}, we take this analytic approach and replace the incoherent Eisenstein series with the $\mathrm{O}(2, 2)$ theta lift of a weight one real-analytic modular form on $\SL_2$ over $\Qb$. 
In the context of the Siegel-Weil formula \cite{KR88a, KR88b}, the incoherent Hilbert Eisenstein series used in \cite{GZ85} and \cite{BY06} lies on $\SL_2$ over the real quadratic field $F$ and is the integral of a theta function over the orthogonal group, i.e.~ sum of big CM points. 
Even though our lift goes in the opposite direction, the exceptional isomorphism between $\mathrm{O}(2, 2)$ over $\Qb$ and $\SL_2$ over $F$ makes it a viable idea.
The key step then is to find the right input and theta kernel.

The input is a weight one modular form, which was studied in \cite{CL20} and used to construct weight one harmonic Maass forms associated to real quadratic fields. 
The form is constructed by deforming a theta integral from $\mathrm{O}(1, 1)$ to $\SL_2$, and has no singularity. Therefore, it can be fed into the theta lifting machinery without regularization and the output will have no singularity either. Unlike the incoherent Eisenstein series in \cite{GKZ87}, our replacement will not behave nicely under the Laplacian. Fortunately, it still has a holomorphic part, which can be explicitly calculated. 
This, along with an error term arising from the non-holomorphic part, equals to $G_{k, f}(Z_\chi)$ (see Theorem \ref{thm:main1}). One then needs to carefully analyze the error term to see that it contributes a rational multiple of $\log \varep_F$ with $\varep_F$ the fundamental unit of $F$. This is the reason we need $\kappa$ in Theorem \ref{thma:main1}. 

The theta kernel we used are also quite interesting. 
For each $k \in 2\Nb$, we use a different Schwartz function formed from the Gaussian with a suitable polynomial $\tp_{k-1}$ in the sense of \cite{Borcherds98} (see section \ref{subsec:lift2}). They are easy to describe in the polynomial Fock space, and particular cases are related to the Schwartz forms studied by Kudla and Millson \cite{KM90}. 
To evaluate such a theta lift, we require a formula by Borcherds \cite[Theorem 7.1]{Borcherds98} with polynomials in both the plus and minus parts.
In addition, we also need the recent work \cite{AS18} on the Fourier expansion of Millson theta lifts.

It would be interesting to see how our formula fits into the context of arithmetic intersection theory, i.e.\ Gross-Zagier formula and various generalizations \cite{Zhang97, YZZ13, BY09, BKY12, BHY15, AGHMP18}.
In the Gross-Zagier type formulas, one has an equality between a special value of the derivative of a suitable $L$-function and an appropriate arithmetic intersection number, which has contributions from the archimedean and non-archimedean parts. The former are special values of Green's functions and the latter are local intersection numbers.
After explicit calculations, one can match the non-archimedean part with Fourier coefficients of incoherent Eisenstein series. 
When the $L$-function is identically zero, e.g.\ attached to the trivial cusp form, then one obtains an equality (up to sign) between the archimedean and non-archimedean contributions. 
This, for example, gives the factorization of singular moduli. Our result suggests that there should be a similar interpretation of Fourier coefficients of our replacement in this setting, with which one can then prove a Gross-Zagier type formula by replacing $f$ with a harmonic Maass form. 
It would then be interesting to generalize our proof to the cases when $f$ has higher level.

We organize the paper into three parts. 
The first part consists of sections \ref{sec:prelim} and \ref{sec:deform}, where we discuss some preliminary notions about theta lifts, real quadratic fields and the deformed theta integral studied in \cite{CL20}. The main results are the two counting propositions \ref{prop:count} and \ref{prop:rewrite2}. 
This can be considered as an analysis of the information contained in the non-archimedean component of the theta lift.
The archimedean component is contained in the last part of the paper. It consists of the last section and the appendix, where we give the relevant theta functions and study in details various theta lifts involved. 
Between these two technical parts is section \ref{sec:hGf}, which is the middle part of the paper. It consists of information about higher Green's function, the main result (Theorem \ref{thm:main}), and its proof. Assuming the first and last parts, the reader can read through the middle part to get the rough idea of the proof.

\vspace{.2in}

\noindent \textbf{Acknowledgement.} We thank Don Zagier for informing us about this conjecture and providing us with the unpublished note \cite{ZagierNote}. We also thank Jan Bruinier for helpful conversations and providing us with a preliminary draft of \cite{BEY}, Richard Borcherds for helpful communication concerning \cite{Borcherds98}, and Stephan Ehlen for helpful exchanges about numerical computations.

\section{Preliminary Facts}
\label{sec:prelim}
\subsection{Weil Representation and Modular Forms}
\label{subsec:lattice}
We first follow the convention of \cite{Borcherds98} to recall the Weil representation arising from finite quadratic modules.
Let ${\Mp}_2(\Rb)$ be the metaplectic two-fold cover of $\SL_2(\Rb)$ consisting of elements $(A,  \phi(\tau))$ with $ A = \smat{a}{b}{c}{d} \in \SL_2(\Rb)$ and $\phi: \Hb \to \Cb$ holomorphic function satisfying $\phi(\tau)^2 = c\tau + d$. 
Let $\Mp_2(\Zb) \subset {\Mp}_2(\Rb)$ be the inverse image of $\SL_2(\Zb) \subset \SL_2(\Rb)$. 
It is generated by $T := (\smat{1}{1}{0}{1}, 1)$ and $S = (\smat{0}{-1}{1}{0}, \sqrt{\tau})$, where $\sqrt{\cdot}$ denotes the principal branch of the holomorphic square root.

Let $\Gamma \subset \Mp_2(\Zb)$ be a subgroup of finite index and $\rho: \Gamma \to \GL(W)$ a representation on a finite dimensional $\Cb$-vector space $W$.
A real-analytic function $f: \Hb \to W$ is called \textit{modular} with respect to $\rho$ and weight $k \in \half \Zb$ if 
\begin{equation}
  \label{eq:modular}
(f \mid_k (A, \phi))(\tau) := \phi(\tau)^{-2k} f(A \cdot z) = \rho((A, \phi)) \cdot f(\tau)
\end{equation}
for all $(A, \phi) \in \Gamma$ and $\tau \in \Hb$.
We denote the space of all such functions and the usual subspace of weakly holomorphic, holomorphic and cuspidal forms by $\Ac_{k, \rho}(\Gamma), M^!_{k, \rho}(\Gamma), M_{k, \rho}(\Gamma), S_{k, \rho}(\Gamma)$ respectively. We sometimes omit $\rho$, resp.\ $\Gamma$, if it is trivial, resp.\ $\Mp_2(\Zb)$.
Given $f \in \Ac_{k, \rho}(\Gamma)$, we denote
\begin{equation}
  \label{eq:c}
  f^c(\tau) := v^{k} \overline{f(\tau)} \in \Ac_{-k, \overline{\rho}}(\Gamma)
\end{equation}
for convenience.

The representations that we are mainly interested in come from arithmetic.
Let $L$ be an even, integral lattice with quadratic form $Q$ of signature $(n^+, n^-)$, the dual lattice $L^\vee$ and the associated finite quadratic module 
\begin{equation}
  \label{eq:AL}
 \Aco_L := L^\vee/L,
\end{equation}
on which $Q$ becomes a quadratic form valued in $\Qb/\Zb$.
Let $\Br(\cdot, \cdot)$ be the associated bilinear form.
The vector space $\Cb[\Aco_L] := \bigoplus_{h \in \Aco_L} \Cb \ef_h$ is naturally an $\Mp_2(\Zb)$-module via the Weil representation $\rho_L$ defined by
\begin{equation}
  \label{eq:Weil_rep}
  \begin{split}
    \rho_L(T)(\ef_h) &:= \ebf(Q(h)) \ef_h, \\
\rho_L(S)(\ef_h) &:= \frac{\ebf(-(n^+ - n^-)/8)}{\sqrt{|A_L|}} \sum_{\mu \in A_L} \ebf(-\Br(h, \mu)) \ef_\mu.
  \end{split}
\end{equation}
Despite the notation, $\rho_L$ only depends on the finite quadratic module $A_L$. 
There is a natural hermitian pairing $\langle, \rangle$ on $\Cb[A_L]$ given by $\langle \ef_{h_1}, \ef_{h_2} \rangle := 1$ if $h_1 = h_2$ and zero otherwise. With respect to this pairing $\rho_L$ is a unitary representation.
When $n^+ + n^-$ is even, $\rho_L$ factors through $\Gamma := \SL_2(\Zb)$.
To understand $A_L$, it is sometimes useful to consider the $p$-component
\begin{equation}
  \label{eq:ALp}
  A_{L, p} := A_L \otimes_\Zb \Zb_p
\end{equation}
with $\Zb_p$ the $p$-adic integers for a prime $p \in \Nb$. 
It is a finite quadratic module with value in $\Zb[\tfrac{1}{p}]/\Zb$.
Through the natural map $A_L \to A_{L, p}$ given by $h \mapsto h \otimes 1$, we have
\begin{equation}
  \label{eq:ALdecomp}
  A_L = \bigoplus_{p\text{ prime}} A_{L, p}
\end{equation}
as finite quadratic modules.

Suppose $M \subset L$ is a sublattice of finite index and $s: L^\vee/M \to A_L$ the natural surjection. Then $L^\vee \subset M^\vee$ and there is a linear map $\psi: \Cb[A_M] \to \Cb[A_L]$ defined by 
\begin{equation}
  \label{eq:psi}
  \psi(\ef_h) :=
  \begin{cases}
    \ef_{s(h)}, & h \in L^\vee/M, \\
0, & \text{otherwise.}
  \end{cases}
\end{equation}
The adjoint of $\psi$ with respect to the hermitian pairing $\langle \cdot, \cdot \rangle$ is
\begin{equation}
  \label{eq:phi}
  \begin{split}
      \phi:  \Cb[A_L] &\to \Cb[A_M]\\
\ef_h &\mapsto \sum_{\mu \in L^\vee/M, s(\mu) = h} \ef_\mu.
  \end{split}
\end{equation}
It is straightforward to check that $\psi$ and $\phi$ are $\Gamma$-linear with respect to the Weil representations above and $\psi \circ \phi$ is $[L:M]$ times the identity map on $\Cb[A_L]$.

\subsection{Symmetric Space and Theta Function}
\label{subsec:theta}
We will now introduce the two-variable theta function associated to a suitable Schwartz function (see \eqref{eq:varphi1}). This can be done either adelically as in \cite{KR88a} or classically as in \cite{Borcherds98}, and the connection between the two are spelled out clearly in section 1 of \cite{Kudla03}. We take an approach closer to that of Borcherds, while keeping the adelic perspective visible in the setup as well. 

For $n \in \Nb$, let $\Rb^n$ be the Euclidean space with the metric 
$$
|\underline{x}|^2 := \sum_{1 \le j \le n} x_j^2 
$$
for $\underline{x} = (x_1, x_2, \dots, x_n) \in \Rb^n$.
With respect to the quadratic form $Q_n(\underline{x}) := \frac{1}{2} |\underline{x}|^2$, the space $\Rb^n$ becomes a real quadratic space.
Let $\Delta_\Rb:= \sum_{j} \partial_{x_j}^2$ be the Laplacian on $\Rb^{n}$ and define for $\alpha > 0$ the following operator
\begin{equation}
  \label{eq:Hop}
  \Hc_\alpha (f) := e^{- \Delta_\Rb/(8\pi \alpha)} f.
\end{equation}
on a smooth function $f: \Rb^{n} \to \Cb$. 
In particular, we consider $f = p$ a polynomial homogeneous of degree $m$.
It is easy to check that 
\begin{equation}
  \label{eq:Phom}
  \Hc_1(p)(\alpha \cdot \underline{x}) = \alpha^{m} \Hc_{\alpha^2} (p)(\underline{x}).
\end{equation}
for any $\alpha \in \Rb, \underline{x} \in \Rb^{n}$.
Using such a polynomial, we define the following Schwartz function on $\Rb^{n}$
\begin{equation}
  \label{eq:varphi}
  \varphi(\underline{x}; p) := \Hc_1(p)(\underline{x}) e^{-{\pi} |\underline{x}|^2},
\end{equation}
which is an eigenfunction of eigenvalue $i^{n}$ with respect to Fourier transform on $\Rb^n$.

Suppose $n^+, n^- \in \Nb$ with $n = n^+ + n^-$.
Let $\Rb^{n^+, n^-} := \Rb^{n^+} \oplus \Rb^{n^-}$ be the real quadratic space with the quadratic form $Q_0 := Q_{n^+} - Q_{n^-}$.
Suppose $p$ is a polynomial on $\Rb^n$ homogeneous of degrees $m^+$ and $m^-$ on $\underline{x}^+ \in \Rb^{n^+}$ and $\underline{x}^- \in \Rb^{n^-}$ respectively.
Then we can form the Schwartz function
\begin{equation}
  \label{eq:varphi1}
  \varphi(\underline{x}; p) := \Hc_1(p)(\underline{x}) e^{-2 \pi Q_0(\underline{x})} \in \Ss(\Rb^{n^+, n^-})
\end{equation}
for $\underline{x} \in \Rb^{n^+, n^-}$, which is also an eigenfunction under the Fourier transform on $\Rb^{n^+, n^-}$. Furthermore, it is in the subspace of $\Ss(\Rb^{n^+, n^-})$ called the polynomial Fock space \cite{KM90}.

For a real quadratic space $(V, Q)$ with signature $(n^+, n^-)$, denote $\Dc_V$ the oriented Grassmannian of $V$, i.e.\ the set of oriented, positive definite $n^+$-dimensional subspaces of $V$.
Then every $w \in \Dc_V$ induces an isometry 
\begin{equation}
  \label{eq:nuw}
  \begin{split}
     \nu_w : V_\Rb &\to \Rb^{n^+, n^-}
  \end{split}
\end{equation}
by projecting $\lambda \in V$ to $w$ and its orthogonal complement $w^\perp \subset V$. We denote these projections by $\lambda_w$ and $\lambda_{w^\perp}$.
For any function $f: \Rb^{n^+, n^-} \to \Cb$, we use $f_w: V_\Rb \to \Cb$ to denote the composition $f \circ \nu_w$. The quadratic form $Q_w := Q_0 \circ \nu_w$ is half the majorant $(\cdot, \cdot)_w$ defined by
\begin{equation}
  \label{eq:majorant}
  (\lambda, \lambda)_w := (\lambda, \lambda) - 2(\lambda_{w^\perp}, \lambda_{w^\perp})
\end{equation}
for $\lambda \in V$.
Given $\tau = u + iv \in \Hb$, we can now define a Schwartz function
\begin{equation}
  \label{eq:varphi2}
  \begin{split}
      \varphi(\lambda; \tau, w; p) &:= {v}^{n/4 - k/2} \varphi_w(\sqrt{v} \lambda; p) \ebf(Q(\lambda) u)\\
&= v^{n^-/2 + m^-} e^{-\Delta_\Rb/(8\pi v)}(p)(\nu_w(\lambda)) \ebf \lp Q(\lambda_w) \tau + Q(\lambda_{w^\perp}) \overline{\tau} \rp
  \end{split}
\end{equation}
for $\lambda \in V$ with $k := \frac{n^+ - n^-}{2} + m^+ - m^-$.

Let $L \subset V$ be an even, integral lattice of rank $n$ with dual lattice $L^\vee$ and finite quadratic module $A_L := L^\vee/L$. We can define the theta function $\Theta_L(\tau, w, p)$ valued in $\Cb[A_L]$ by
\begin{equation}
  \label{eq:ThetaL}
  \Theta_{L}(\tau, w; p) = 
 \sum_{h \in A_L} \ef_h \sum_{\lambda \in L + h}  \varphi(\lambda; \tau, w; p),
\end{equation}
and denote the $\ef_h$ component of $\Theta_L$ by $\Theta_{L, h}$.
We omit $p$ in the notation if it is the constant 1.
 Let $\Gamma_L$ be the kernel of the surjection $\SO(L) \to \Aut(A_L)$.
Then $\Theta_L(\tau, w; p)$ is clearly $\Gamma_L$ invariant in $w \in \Dc_V$. 
By Poisson summation, one can show that $\Theta_L(\tau, w; p) \in \Ac_{k, \rho_L}$ (see \cite[Theorem 4.1]{Borcherds98}).
For $m \in \Qb$ and $h \in A_L$, denote as usual
\begin{equation}
  \label{eq:Lmh}
  L_{m, h} := \{\lambda \in L + h: Q(\lambda) = m\},~
L(m, h) := \Gamma_L \backslash L_{m, h}.
\end{equation}
Then $\Theta_{L, h}$ has the Fourier expansion
$$
  \Theta_{L, h}(\tau, w; p) = 
v^{n/4-k/2} \sum_{m \in \Qb} \ebf(mu) \sum_{\lambda \in L_{m, h}} 
\varphi_w(\sqrt{v} \lambda; p) 
$$
For a finite index sublattice $M \subset L$, it is easy to check that 
\begin{equation}
  \label{eq:equivariant}
  \psi(\Theta_M(\tau, w; p)) = \Theta_L(\tau, w; p),
\end{equation}
where $\psi$ is the map in \eqref{eq:psi}.

\begin{exmp}
  \label{ex:P}
Suppose $V = \Rb^{1, 0}$ and $p(x) = p_b(x) = x^b$ for $b \in \Nb$. Then the polynomial in the theta function above can be expressed in terms of the Hermite polynomials.
In particular, we have
\begin{equation}
  \label{eq:hermite}
  \Hc_1(p_b)(x) e^{-\pi x^2} = (-4\pi)^{-b} (\partial_x - 2\pi x)^b e^{-\pi x^2}.
\end{equation}
For $N \in \Nb$, the lattice $\Zb$ is even integral with respect to the quadratic form $Q_N(x) := Nx^2$, and its dual lattice is given by $\frac{1}{2N} \Zb$.
We denote this lattice and the associated Weil representation by $P_N$ and $\rho_N$ respectively.
Since $P_N$ is definite, the Grassmannian consists of one point $o$, which induces
\begin{align*}
\nu_o:   P_N \otimes \Rb & \to \Rb^{1, 0}\\
r & \mapsto \sqrt{2 N}r.
\end{align*}
Using this lattice, we can form the theta series
\begin{equation}
  \label{eq:thetan}
  \theta_N(\tau; b) := \Theta_{P_N}(\tau, o; p_b) = \sum_{h \in \frac{1}{2N}\Zb/\Zb} \ef_h \sum_{r \in \Zb + h} 
\Hc_{v}(p_b)(\sqrt{2N}r)  \ebf (N r^2 \tau) \in \Ac_{b + 1/2, \rho_N}.
\end{equation}
For $b = 0, 1$, $\theta_N(\tau; b)$ is holomorphic in $M_{1/2, \rho_N}$ and $S_{3/2, \rho_N}$ respectively. If $P_N^-$ denotes the lattice $(\Zb, -Q_N)$, then $\Theta_{P^-_N}(\tau, o; p_b) = \theta^c_N({\tau}; b)$.
\end{exmp}

\begin{exmp}
  \label{ex:M2Z}
Let $(V, Q) = (M_2(\Rb), \det)$ be the real quadratic space of signature $(2, 2)$ with $M_2(\Rb)$ the vector space of 2 by 2 matrices with entries in $\Rb$. The Grassmannian $\Dc_V$ is identified with $\Hb^2$ via 
\begin{equation}
\label{eq:w}
\underline{z} = (z_1, z_2) \in \Hb^2 \mapsto w(\underline{z}) := \Rb \Re Z(\underline{z})
\oplus \Rb \Im Z(\underline{z}) \in \Dc_V,
\end{equation}
where we denote
\begin{equation}
  \label{eq:Z}
  Z(\underline{z}) = X(\underline{z}) + iY(\underline{z}) :=  \frac{1}{\sqrt{2y_1y_2} } \pmat{z_1 \overline{z_2}}{\overline{z_2}}{z_1}{1},  \;
Z^\perp(\underline{z}) = X^\perp(\underline{z}) + iY^\perp(\underline{z}) := \frac{1}{\sqrt{2 y_1y_2}} \pmat{z_1z_2}{z_2}{z_1}{1}.
\end{equation}
It will also be important later for us to consider the diagonally embedded $\Hb$ in $\Hb \times \Hb$. In that case, we use $z \in \Hb$ to denote its image $(z, z) \in \Hb \times \Hb$ and also the point $w((z, z)) \in \Dc_V$.
The vector $Y:= Y(z) = \tfrac{1}{\sqrt{2}}\smat{0}{-1}{1}{0}$ is independent of $z$ and the isometry $\nu_z: V \to \Rb^{2, 2}$ is then given by 
\begin{equation}
  \label{eq:nu}
\nu_z(\Lambda) :=  (\Lambda_z, \Br(\Lambda, Y), \Re(\Lambda_{z^\perp}), \Im(\Lambda_{z^\perp})) \in \Rb^{2, 2}, 
\end{equation}
where we denote
\begin{equation}
  \label{eq:Lambdaz}
  \Lambda_z := \Br(\Lambda, X(z)), \; \Lambda_{z^\perp} := \Br(\Lambda, Z^\perp(z)).
\end{equation}

The lattice $L = M_2(\Zb) \subset V$ is even and unimodular. The theta function $\Theta_{L}(\tau, w(\underline{z}))$ has the explicit form
\begin{equation}
\label{eq:Theta}
\Theta_{L}(\tau, w(\underline{z})) =  v \sum_{a, b, c, d \in \Zb} \ebf \lp 
\frac{|az_1 z_2 + bz_1 + c z_2 + d|^2 \tau - |az_1 \zbar_2 + bz_1 + c \zbar_2 + d|^2 \overline{\tau}}{4 y_1 y_2} 
\rp.
\end{equation}
From the result above, we know that $\Theta_L(\tau, w(\underline{z})) \in \Ac_0$ as a function of $\tau \in \Hb$ for any $\underline{z} \in \Hb^2$. 
\end{exmp}

Suppose $f(\tau) \in \Ac_{k, \rho_L}$ is regular on $\Gamma \backslash \Hb$ and has sufficient decay near the cusp, then one can integrate $\langle f(\tau), \Theta_L(\tau, w; p) \rangle v^{k} d\mu(\tau)$ on $\Gamma \backslash \Hb$ to obtain a function on $\Dc_V$. 
If $f$ is allowed to have linear exponential growths at the cusps and be nearly holomorphic, i.e. $L^n_\tau f = 0$ for $n \gg 0$, then Borcherds defined regularized theta integral
\begin{equation}
  \label{eq:Phireg}
  \begin{split}
      \Phi_L(w; p, f) &=   \int^\mathrm{reg}_{\Gamma \backslash \Hb} \langle f(\tau), \Theta_L(\tau, w; p) \rangle v^{k} d\mu(\tau)\\
&:= \mathrm{Const}_{s = 0}\lim_{T \to \infty} \int_{\Fc_T} \langle f(\tau), \Theta_L(\tau, w; p) \rangle v^{k - s} d\mu(\tau)
  \end{split}
\end{equation}
 following the regularization idea of Harvey and Moore \cite{HM96}. 
This is defined by integrating first over the truncated fundamental domain $\Fc_T := \{\tau = u+ iv \in \Hb: |\tau| > 1, |u| < 1/2, v < T\}$, before taking the limit $T \to \infty$ for $\Re(s) \gg 0$, analytically continuing this to $s \in \Cb$ and taking the constant term in the Laurent expansion in $s \in \Cb$.
The same integral can be defined for harmonic Maass forms as well \cite{Bruinier02, BF04}.

On the other hand, one could also integrate orthogonal modular forms in the variable $w$ to produce symplectic modular forms on $\Hb$. In \cite{Hecke26}, Hecke computed one of the first such theta integral and associated weight one cusp forms to real quadratic fields.
{
  This is a special case of the Siegel-Weil formula \cite{KR88a, KR88b}. 
  In the next section, we will review Hecke's construction along with its deformation.
}

\subsection{Real Quadratic Field}
\label{subsec:F}
From now on, let $F \subset \Rb$ be a real quadratic field with fundamental discriminant $\Delta > 1$. 
Its ring of integer and inverse different are given by $\Oc_F := \Zb + \Zb \omega_\Delta, \omega_\Delta := \frac{\Delta + \sqrt{\Delta}}{2}$ and $\df^{-1} := \frac{1}{\sqrt{\Delta}} \Oc_F$ respectively.
We also denote $\Delta_0$ the square-free part of $\Delta$, i.e.\ $\Delta = \Delta_0 \cdot c^2$ for some $c \in \Nb$.
Since $\Delta$ is fundamental, it is in one of the following forms
\begin{itemize}
\item $\Delta \equiv 1 \bmod{4}$ and $\Delta = \Delta_0$,
\item $\Delta \equiv 0 \bmod{4}$ and $\Delta/4 = \Delta_0 \equiv 2, 3 \bmod{4}$.
\end{itemize}
We denote the (narrow) class group of $F$ by $\Cl^{(+)}(F)$, which is the quotient of
\begin{equation}
  \label{eq:Idf}
  I_\df := \{ \Oc_F\text{-fractional ideal } \af \subset F : \af \text{ and } \df \text{ are relatively prime to each other.}\}
\end{equation}
by its subgroup of principal ideals (with totally positive generators).
By the Chebotarev density theorem, every class in $\Cl^+(F)$ has infinitely many prime representatives in $I_\df$.
For each prime $p \mid \Delta$, let $\df_p$ denote the unique (ramified) prime ideal in $\Oc_F$ above $p$. They are the only prime ideals dividing $\df$.
We use $'$ to denote the non-trivial Galois conjugation, with which we can define the norm map 
$
  \Nm(\alpha) := \alpha \alpha'
$
on $F$. Denote 
\begin{equation}
  \label{eq:NmFpm}
 \Nm(F)_\pm := \{m \in \Qb: \pm m > 0 \text{ and there exists } \alpha \in F \text{ with } \Nm(\alpha) = m \}. 
\end{equation}
If $F$ does not have an element of norm $-1$, then $\Nm(F)_+$ and $\Nm(F)_-$ are disjoint. This happens when $\Delta$ is divisible by a prime congruent to $3$ modulo $4$.

The group of units $\Oc_F^\times$ is generated by the fundamental unit $\varep_F > 1$. 
If $\Delta$ is divisible by a prime congruent to 3 modulo 4, then $\varep_F$ is totally positive, and the canonical map $\Cl^+(F) \to \Cl(F)$ is a 2-1 surjection.
Besides the usual norm map, we will also need the map on $\Oc_F$-fractional ideals
\begin{equation}
  \label{eq:Nm-}
  \Nm^-(\af) :=  \frac{\af}{\af'},
\end{equation}
whose image is exactly the kernel of $\Nm$.

Suppose $Q = \Nm$ is the quadratic form. Then $\Oc_F$ is an even integral lattice of signature $(1, 1)$, whose discriminant kernel is denoted by $\Gamma_\Delta$.
The Galois conjugation is an element in $\mathrm{O}(\Oc_F)$.
On the finite quadratic module
\begin{equation}
  \label{eq:ADelta}
 A_\Delta := \df^{-1}/\Oc_F,
\end{equation}
which is identified with $\Zb/(\Delta/ \gcd(\Delta, 2)) \Zb \times \Zb/(\gcd(\Delta, 2)) \Zb$ by the map 
\begin{equation}
  \label{eq:fqmmap}
\frac{a + b\omega_\Delta}{\sqrt{\Delta}} \mapsto  (a, b),
\end{equation}
the Galois conjugation acts as the map sending $h \in A_\Delta$ to $-h$.
The multiplication on $F$ also induces a multiplication $A_\Delta \times A_\Delta \to A_\Delta$.
For us later, it will be convenient to work with the $p$-components $A_{\Delta, p}$. If $p \nmid \Delta$, then $A_{\Delta, p}$ is trivial.
When $p \mid \Delta$ is an odd prime, it is easily seen that
\begin{equation}
  \label{eq:ADp}
A_{\Delta, p} \cong (\Zb/p\Zb, a \mapsto \alpha_{\Delta, p} \tfrac{a^2}{p}).   
\end{equation}
for some $\alpha_{\Delta, p} \in (\Zb/p\Zb)^\times$ depending on $\Delta$.
For $p = 2 \mid \Delta$, there are two possibilities
\begin{equation}
  \label{eq:AD2}
A_{\Delta, 2} \cong
\begin{cases}
   ((\Zb/2\Zb)^2 , (a, b) \mapsto   \tfrac{a^2 + b^2}{4}), & \ord_2(\Delta) = 2, \\
(\Zb/4\Zb \times \Zb/2\Zb, (a, b) \mapsto  \tfrac{\pm a^2 + 6b^2}{8}), & \ord_2(\Delta) = 3,
\end{cases}
\end{equation}
since $\Delta$ is a fundamental discriminant.
Note that the $\pm$'s must be chosen consistently and the choices depend on $\Delta$.

For every prime $p \mid \Delta$, we define an automorphism $\sigma_p$ on $A_{\Delta, p}$ by
\begin{equation}
  \label{eq:sigmap}
  \sigma_p(h) :=
  \begin{cases}
    (b, a), & \text{if } p = 2, \ord_2(\Delta) = 2, \text{ and }h = (a, b) \in (\Zb/2\Zb)^2 \cong A_{\Delta, 2} , \\
-h, &\text{ otherwise,}
  \end{cases}
\end{equation}
which has order 2. It is easy to check that for every prime $p\mid \Delta$ and distinct $h, \tilde{h} \in A_{\Delta, p}$
\begin{equation}
  \label{eq:equivalence}
Q_p(h) = Q_p(\tilde{h}) \Leftrightarrow \sigma_p(h) = \tilde{h}.  
\end{equation}
Since $A_\Delta = \oplus_{p \mid \Delta\text{ prime}} A_{\Delta, p}$ as a finite quadratic modules in \eqref{eq:ALdecomp}, we also use $\sigma_p$ to denote the isometry on $A_\Delta$ that acts as $\sigma_{p}$ on $A_{\Delta, p}$, and trivially on the other components.
These automorphisms all have order 2 and commute with each other. 
We denote the group they generate by $\mathrm{G}_\Delta$, which contains the subgroup
\begin{equation}
  \label{eq:GDeltah}
  \mathrm{G}_{\Delta, h}:= \{ \sigma \in \mathrm{G}_\Delta: \sigma(h) = h\}
\end{equation}
and has the following property.

\begin{lemma}
  \label{lemma:Qh}  
In the notations above, $h$ is in $\df_p \df^{-1}/\Oc_F \subset A_\Delta$ if and only if $\sigma_p$ fixes $h$.
Suppose $Q(h) = \frac{n}{\Delta}$ with $n \in \Zb/\Delta\Zb$, then $|\mathrm{G}_{\Delta, h}| = 2^{\omega(\gcd(n, \Delta))}$, where $\omega(n)$ is the number of distinct prime factors of $n \in \Zb$.
Furthermore, if there exists $\tilde{h} \in A_\Delta$ satisfying $Q(h) = Q(\tilde{h})$, then $\tilde{h} = \sigma(h)$ for some $\sigma \in \mathrm{G}_\Delta$.
\end{lemma}

\begin{rmk}
  \label{rmk:0}
For convenience, we define
\begin{equation}
  \label{eq:GDelta}
  G_\Delta := \{|d| : d \text{ fundamental discriminant dividing } \Delta \text{ and } \gcd(d, \Delta/d) = 1\},
\end{equation}
which has a group structure under the multiplication $d_1 \cdot d_2 := d_1d_2/\gcd(d_1, d_2)^2$.
We canonically identify it with $\mathrm{G}_\Delta$ via the map
\begin{equation}
  \label{eq:sigmad}
d \in G_\Delta \mapsto  \sigma_d := \prod_{p \mid d \text{ prime}} \sigma_p \in \mathrm{G}_\Delta.
\end{equation}
Consider the subgroups of $G_\Delta$ of the form
\begin{equation}
  \label{eq:GDeltad}
  G_{\Delta, d} := \left\{ \tilde{d} \in G_\Delta: \tilde{d} \mid d \right\}
\end{equation}
for some $d \in G_\Delta$.
For any $h \in A_\Delta$, $\sigma_d \in \mathrm{G}_{\Delta, h}$ if and only if $\sigma_p \in \mathrm{G}_{\Delta, h}$ for all primes $p \mid d$.
Therefore, we can define a function $\mathrm{d}: A_\Delta \to G_{\Delta}$ such that
\begin{equation}
  \label{eq:drm}
G_{\Delta, \mathrm{d}(h)} = \mathrm{G}_{\Delta, h}, 
\end{equation}
i.e.\ $\mathrm{d}(h)$ and $\gcd(\Delta, n)$ have the same prime factors. For any prime $p \mid \Delta$ and $h \in A_\Delta$, we have
\begin{equation}
  \label{eq:lcm}
  \mathrm{d}(p h) = \mathrm{lcm}(\mathrm{d}(h), p).
\end{equation}
\end{rmk}

\begin{rmk}
  \label{rmk:A*}
The following set 
\begin{equation}
  \label{eq:ADelta*}
  A_\Delta^* := \left\{h \in A_\Delta: \#\mathrm{G}_{\Delta, h} = 1\right\} = 
\left\{h \in A_\Delta: \mathrm{d}(h) = 1\right\} 
\end{equation}
will be useful for us later.
\end{rmk}

\begin{rmk}
\label{rmk:Delta0}
One can check that $\sigma_{\Delta_0}(h) = -h$, where $\Delta_0$ is the square-free part of $\Delta$. Therefore, $2h = 0 \in A_\Delta$ if and only if $\sigma_{\Delta_0} \in \mathrm{G}_{\Delta, h}$, i.e.\ $\Delta_0 \mid \mathrm{d}(h)$.
\end{rmk}

\begin{proof}
The first claim can be verified directly, though one needs to be slightly careful when $p = 2$ divides $\Delta$.
The second claim follows from the first since $h \in \df_p \df^{-1}/\Oc_F$ if and only if $p \mid m$.
For the last claim, notice that $Q(h) = Q(\tilde{h})$ is equivalent to $Q_p(h_p) = Q_p(\tilde{h}_p)$ for all primes $p \mid \Delta$, where $h_p$ and $\tilde{h}_p$ are the $p$-components of $h$ and $\tilde{h}$ respectively.
Therefore, $\tilde{h}_p$ is either $h_p$ or $\sigma_p(h_p)$ by \eqref{eq:equivalence}, and we are done after setting $\sigma := \prod_{p \mid \Delta \text{ prime,}~ \tilde{h}_p = \sigma_p(h_p)} \sigma_p \in \mathrm{G}_\Delta$.
\end{proof}

We can now replace $\Oc_F$ with any fractional $\Oc_F$-ideal $\af \subset F$ to obtain an even integral lattice $L_\af$ with $Q_\af = \Nm/\Nm(\af)$ and $L^\vee_\af = \af\df^{-1}$. 
We denote the finite quadratic module $(L_\af^\vee/L_\af, Q_\af)$ and the Weil representation associated to it by $A_\af$ and $\rho_\af$ respectively.
Their isomorphism classes only depend on the class of $\af$ in $\Cl^+(F)$.
The finite abelian groups $A_\af$ and $A_\Delta$ are always isomorphic. 
Furthermore, suppose that $\af, \bfrak \in I_\df$ with $\bfrak \subset\af $. Then one can use this inclusion to canonically identify the abelian groups $A_\af$ and $A_\bfrak$.
By taking $\bfrak = \af \cap \Oc_F$, we then have the canonical identification.
\begin{equation}
  \label{eq:identify}
A_\Delta =   A_{\af \cap \Oc_F} = A_{\af}.
\end{equation}
However, the quadratic form is given by $Q_\af = \Nm(\af)^{-1} Q$ under this group isomorphism.
Since the quadratic forms only differ by a factor invertible in $\Zb/\Delta\Zb$, one can define the automorphism $\sigma_p$ on $A_{\af, p}$ and $A_\af$ as in \eqref{eq:ADelta*} and \eqref{eq:sigmap} respectively.
Therefore the group $\mathrm{G}_\Delta =  G_\Delta$ also acts $A_\af$, and one can define the subgroup $\mathrm{G}_{\Delta, h}$ for $h \in A_\af$ and the subset $A_\af^* \subset A_\af$ as in \eqref{eq:GDeltah} and \eqref{eq:ADelta*} respectively.

We are in a particularly nice situation if $Q_\af = Q$, i.e.\ $\af\in \ker(\Nm) = \mathrm{im}(\Nm^-)$ 
, since \eqref{eq:identify} is an equality of finite quadratic modules, which is $\Gamma$-linear on $\Cb[A_{\af}]= \Cb[A_\Delta]$ with respect to $\rho_{\af}$ and $\rho_\Delta$ after extending by linearity. 
On the other hand, if $\af = (\alpha)$ for some $\alpha \in F$ coprime to $\df$, then there is another group isomorphism
\begin{equation}
  \label{eq:multiso}
  \begin{split}
 A_\Delta = \df^{-1}/\Oc_F &\cong A_\af = \af\df^{-1}/\af\\
h &\mapsto \alpha h.
  \end{split}
\end{equation}
When $A_\Delta$ is equipped with the quadratic form $Q_\alpha := \Nm/\Nm(\alpha)$, this is an isomorphism between finite quadratic modules.
If $\Nm(\alpha) = 1$, then $A_\af = A_\Delta$ and $(\alpha) = \Nm^-(\bfrak)$ with $\bfrak \in I_\df$ having order at dividing 2 in $\Cl^+(F)$.
When $[\bfrak]$ is trivial in $\Cl^+(F)$, i.e.\ $\bfrak = (\beta)$ with $\beta \in F$ totally positive, then $\Nm^-(\beta) = 1 \bmod{\df}$ and the composition $A_\Delta \cong A_\af = A_\Delta$ is the identity map. 
Therefore, the composition $A_\Delta \cong A_\af = A_\Delta$ only depends on the class of $\bfrak$ in $\Cl^+(F)$.
It is straightforward to verify that this is the automorphism $\sigma_p$ defined in \eqref{eq:sigmap} when $[\bfrak] = [\df_p]$ for any prime $p \mid \Delta$.

\subsection{Genus Theory}
\label{subsec:genus}
We make a quick digression to the genus theory of quadratic fields.
The map $\Nm^-$ defined in \eqref{eq:Nm-} induces the squaring map on $\Cl^+(F)$, whose kernel and image are denoted by
  \begin{equation}
    \label{eq:Cl2}
    \Cl^+_2(F) := \ker(\Nm^-) \subset \Cl^+(F), \quad     \Cl^{+,2}(F) := \mathrm{im}(\Nm^-) \subset \Cl^+(F)
  \end{equation}
respectively.
Then $\Cl^+_2(F)$ is exactly the group of 2-torsion elements in $\Cl^+(F)$ and is isomorphic to $(\Zb/2\Zb)^{\omega(\Delta) - 1}$ by genus theory. 
It is easy to see that $[\df_p] \in \Cl^+_2(F)$ for every prime $p \mid \Delta$. 
With the group $G_\Delta$ defined in \eqref{eq:GDelta} from the previous section, we can now describe $\Cl^+_2(F)$ precisely in the following lemma.
\begin{lemma}
  \label{lemma:genus}
There is a unique $d_0 \in G_\Delta \backslash \{1\}$ such that the class of
\begin{equation}
  \label{eq:df0}
  \df_0 := \prod_{p \mid d_0 \text{ prime}} \df_p
\end{equation}
in $\Cl^+(F)$ is trivial. 
Let $\varep_F^+ > 1$ be the generator of the subgroup of totally positive units in $\Oc_F^\times$.
Then multiplying by $\varep_F^+$ induces the automorphism $\sigma_{d_0}$ on $A_\af$.
The discriminant kernel $\Gamma_\Delta$ of $L_\af$ is generated by 
\begin{equation}
  \label{eq:varepDelta}
\varep_\Delta := (\varep_F^+)^2
\end{equation}
and the quotient group $\SO(L_\af)/\Gamma_\Delta$ is isomorphic to $(\Zb/2\Zb)^2$ with generators $\varep_F^+$ and $-1$. 
\end{lemma}

\begin{proof}
For every $d \in G_\Delta$, define
\begin{equation}
  \label{eq:dfd}
\df_{d} := \prod_{p \mid d~ \text{prime}} \df_p.  
\end{equation}
Then $d, \tilde{d} \in G_\Delta$ are equal if and only if $\df_d = \df_{\tilde{d}}$.
The unit $\varep_F^+$ is either $\varep_F$ or $\varep_F^2$ depending on whether $\varep_F$ has norm $1$ or $-1$.
By Hilbert's Theorem 90, there exists an $\alpha_0 \in \Oc_F$ such that $\varep^+_F = \alpha_0/\alpha'_0$. Since $\varep^+_F$ is a unit, we can choose $\alpha_0$ such that $(\alpha_0) = \df_{d_0}$ for some $d_0 \in G_\Delta \backslash\{1\}$. 
Suppose $d \in G_\Delta$ such that $\df_d = (\alpha_d)$ with $\alpha_d$ totally positive, then $\alpha_d'/\alpha_d = (\varep^+_F)^r = (\alpha_0/\alpha'_0)^r$ for some $r \in \Zb$, which implies $\alpha^r_0 \alpha_d \in \Zb$. If $r$ is even, then $\df_d$ has a generator in $\Zb$, meaning $d$ must be 1. If $r$ is odd, then $\df_{d_0} \df_d$ has a generator in $\Zb$, implying that $\df_{d}$ and $\df_{d_0}$ have the same prime factors, hence are the same. 
This proves the uniqueness of $d_0$. Note that $d_0 = \Delta$ if $\varep_F$ has norm $-1$.
We denote $\df_0 := \df_{d_0}$.
Using $\varep_F^+ = \alpha_0/\alpha'_0$ with $(\alpha_0) = \df_0$ and $\alpha_0 - \alpha'_0 \in \df$, it is straightforward to verify the congruence
\begin{equation}
  \label{eq:varepcong}
 \varep^+_F \equiv 1 \bmod{\df/\df_{0}},~\varep^+_F \equiv 
 \begin{cases}
   -1 \bmod{\df_{p}}, & p \mid d_0 \text{ odd,}\\
-1 \bmod{\df_2}, & 2 \mid d_0, \ord_2(\Delta) = 2,\\ 
-1 \bmod{\df^3_2}, & 2 \mid d_0, \ord_2(\Delta) = 3,
 \end{cases}
\end{equation}
which implies the second claim.
Note that $\ord_2(\Delta) = 2$ implies $\Delta_0 \equiv 3 \bmod{4}$ and $\Delta$ is divisible by a prime congruent to 3 modulo 4. In that case, the fundamental unit $\varep_F$ has norm $1$, i.e.\ $\varep_F^+ = \varep_F$.
Since $d_0 \neq 1$, $\df_0$ is non-trivial and $\varep^+_F \not\in \Gamma_\Delta$. 
As the group $\SO(L_\af)$ is the subgroup of units in $\Oc_F^\times$ with norm $1$, the quotient $\SO(L_\af)/\Gamma_\Delta$ contains $\pm 1$ and $\pm \varep^+_F$. This proves the last claim.
\end{proof}

\begin{cor}
  \label{cor:gen}
The map sending $d \in G_\Delta$ to $\df_d$ defined in \eqref{eq:dfd} is a surjective group homomorphism from $G_\Delta$ to $\Cl^+_2(F)$. Its kernel is generated by $d_0 \in G_\Delta$ satisfying $\df_{d_0} = \df_0$.
\end{cor}
Let $\overline{G_\Delta} = \overline{\mathrm{G}_\Delta}$ denote the quotient by the subgroup generated by $d_0$.
Since the group $G_{\Delta}$ acts on $A_\Delta$, we can consider define an equivalence relation $\sim_{d_0}$ such that $h_1 \sim_{d_0} h_2 $ if and only if $h_1 = \sigma_{d_0}(h_2)$ for $h_1, h_2 \in A_\Delta$. Then $\Cl^+_2(F)$ acts on the set
\begin{equation}
  \label{eq:ADeltabar}
  \overline{A_\Delta} := A_\Delta / \sim_{d_0}.
\end{equation}
For $h \in A_\Delta$, we also use $\overline{h}$ to denote its image in $\overline{A_\Delta}$, and $\mathrm{G}_{\Delta, \overline{h}}$ the stabilizer of $h$ or $\sigma_{d_0}(h)$ in $\mathrm{G}_\Delta$.

Using the embedding $F \subset \Rb$ fixed in the beginning of the section, we can make sense of the sign function on $\SO(L_\af)$, which factors through the discriminant kernel.
For $h \in A_\af$, we define
\begin{equation}
  \label{eq:srh}
  \sr_h := \sum_{s \in \SO(L_\af)/\Gamma_\Delta, s \cdot h = h} \sgn(s).
\end{equation}
Since the group $\SO(L_\af)/\Gamma_\Delta$ is isomorphic to $(\Zb/2\Zb)^2$, the only possible values of $\sr_h$ are $0, 1$ and $2$. 
Using Lemma \ref{lemma:genus}, we can explicitly determine them as follows.
\begin{lemma}
\label{lemma:shval}
Let $\mathrm{d}$ be the function defined in \eqref{eq:drm}, $\Delta_0$ the square-free part of $\Delta$ and $d_0 \in G_\Delta$ the fundamental discriminant in Lemma \ref{lemma:genus}.
Then we have
\begin{equation}
  \label{eq:shval}
  \sr_h = 
  \begin{cases}
    0 , & \Delta_0 \mid \mathrm{d}(h), \\
2, & d_0 \mid \mathrm{d}(h) \text{ and } \Delta_0 \nmid \mathrm{d}(h), \\
1, & \text{otherwise.}
  \end{cases}
\end{equation}
\end{lemma}

\begin{rmk}
  \label{rmk:4}
Since $d_0, \Delta_0 \neq 1$, we have $\sr_h = 1$ for $h \in A^*_\af$.
\end{rmk}
\begin{proof}
  It is clear $\sr_h = 0$ if and only if $2h = 0 \in A_\af$, which is equivalent to $\Delta_0 \mid \mathrm{d}(h)$ by Remark \ref{rmk:0}.
Now suppose $\Delta_0 \nmid \mathrm{d}(h)$, then $\sr_h = 2$ if and only if $\varep^+_F h = h$, which is equivalent to $\sigma_{d_0} \in \mathrm{G}_{\Delta, h}$ by Lemma \ref{lemma:genus}. This is then the same as $d_0 \mid \mathrm{d}(h)$ by \eqref{eq:drm}.
In all other cases, $\sr_h = 1$ and we are finished with the proof.
Note that it is possible that both $d_0$ and $\Delta_0$ divides $\mathrm{d}(h)$ for non-trivial $h \in A_\af$. For example, if $\Delta = 12 = d_0, \af = \Oc_F$ and $h = (1 + \sqrt{3})/2 \in A_\Delta$, then $\mathrm{d}(h) = \Delta$. 
\end{proof}

Now we can attach an element of the group ring $\Zb[\Cl^+(F)]$ to an ideal $\af \subset \Oc_F$ and $h \in A_\Delta$ as follows
\begin{equation}
  \label{eq:sqrt}
  \rt(\af, h) := 
\sum_{\begin{subarray}{c} 
\Bf = [\bfrak] \in \Cl^+(F) \\ \af = \Nm^-(\bfrak)(\mu) \text{ for a positive }\\ \frac{\mu}{\sqrt{\Delta}} \in \Oc_F + h \subset \df^{-1}
\end{subarray}} \Bf \in \Zb[\Cl^+(F)].
\end{equation}
Since $\frac{\alpha}{\alpha'} \equiv 1 \bmod{\df}$ whenever $\alpha \in F$ is relatively prime to $\df$, we know that the condition in the definition is independent of the choice of the representative $\bfrak \in I_\df$.
However, there is a choice between $\mu$ and $\varep_F^+ \mu$ for the positive generator. Since Lemma \ref{lemma:genus} tells us that the multiplication by $\varep_F^+$ induces $\sigma_{d_0}$ on $A_\af$, we have
\begin{equation}
  \label{eq:equal}
  \rt(\af, h) = \rt(\af, \sigma_{d_0}(h))
\end{equation}
for every $\af$ and $h$.
For other distinct $h_1, h_2 \in A_\Delta$, the supports of $\rt(\af, h_1)$ and $\rt(\af, {h_2})$ (when viewed as functions on $\Cl^+(F)$) are disjoint.
The action of an element $d \in G_\Delta = \mathrm{G}_{\Delta} \cong \Cl^+_2(F)$ on $\Cl^+(F)$ via multiplication with $[\df_d]$ induces an isomorphism on $\Zb[\Cl^+(F)]$, which sends $\rt(\af, h)$ to $\rt(\af, \sigma_d(h))$. 
This implies 
\begin{equation}
  \label{eq:permute}
\rt(\af, \sigma_{d}(h)) = [\df_d] \cdot \rt(\af, h). 
\end{equation}
When $\rt(\af, h)$ is non-zero, its support is a regular $G_{\Delta, \mathrm{d}(h)}$-set and has size $2^{\omega(\gcd(\Nm(\af), \Delta))}$.
In particular, it has exactly one element for $h \in A_\Delta^*$ if $\frac{\Nm(\af)}{\Delta} = Q(h)$ and $[\af] \in \Cl^{+, 2}(F)$.

For any positive $\alpha$, it is clear that 
\begin{equation}
  \label{eq:equiv1}
  \mathrm{supp}(\rt(\af, h)) \subset \mathrm{supp}(\rt((\alpha) \af, \alpha h)),
\end{equation}
which is an equality if $\alpha$ is relatively prime to $\df$.
Using the analysis above, we have the following precise result.
\begin{lemma}
  \label{lemma:countmul}
For any positive $\alpha \in \Oc_F$ and integral ideal $\af \subset \Oc_F$, let $d \in G_\Delta$ be the smallest element such that $\gcd(\Nm(\alpha), \Delta/d) = 1$. 
Then for any $h \in A_\Delta$, we have
\begin{equation}
  \label{eq:countmul}
    \rt((\alpha) \af, \alpha h) =
\frac{1}{\# (G_{\Delta, d} \cap \mathrm{G}_{\Delta, \overline{h}})}
\sum_{\sigma \in G_{\Delta, d}} \rt(\af, \sigma_{}(h)) .
\end{equation}
\end{lemma}
\begin{proof}
The choice of $d$ implies that $\alpha \sigma_{}(h) = \sigma_{}(\alpha h) = \alpha h$ for all $\sigma \in G_{\Delta, d}$. In fact, every $\tilde{h} \in A_\Delta$ with the property $\alpha \tilde{h} = \alpha h$ is of the form $\sigma_{\tilde{d}}(h)$ for some $\tilde{d} \mid d$ by \eqref{eq:equivalence}. 
On the other hand, $\rt(\af, h) = \rt(\af, \sigma(h))$ if and only if $\sigma(h) = h$ or $\sigma_{d_0}(h)$, i.e.\ $\sigma \in \mathrm{G}_{\Delta, \overline{h}}$. 
Therefore, we have an equality of both sides.
\end{proof}

Finally, recall that the genus field $F_\mathrm{gen}$ of $F$ is the maximal abelian, unramified extension of $F$ that is obtained from the composite of an abelian extension $K/\Qb$ and $F$. By class field theory, $\Gal(F_\mathrm{gen}/F)$ is identified as a quotient of $\Cl^+(F)$, and a genus character is a character of $\Cl^+(F)$ that factors through this quotient. We can also view it as a homomorphism from $\Zb[\Cl^+(F)]$ to $\Zb$.
If we factor $\Delta$ into $\Delta_1 \cdot \Delta_2$ with $\Delta_1, \Delta_2$ co-prime, fundamental discriminants, then there is a genus character $\chi_{\Delta_1, \Delta_2}$ on $\Cl^+(F)$ corresponding to it defined by
\begin{equation}
  \label{eq:chi}
  \chi_{\Delta_1, \Delta_2}(\Af) = \kro{\Delta_1}{A} = \kro{\Delta_2}{A}
\end{equation}
with $\Af = [\af]$ and $A = \Nm(\af)$ relatively prime to $\Delta$. 
 Note that for any $\chi_{\Delta_1, \Delta_2}$, we have
\begin{equation}
  \label{eq:minus}
  \chi_{\Delta_1, \Delta_2}([\df]) = \sgn(\Delta_1) = \sgn(\Delta_2).
\end{equation}
By genus theory for quadratic field, these are all the genus characters of $F$.
We call a genus character $\chi$ even, resp.\ odd, if $\chi([\df])$ is positive, resp.\ negative.
There exists an odd genus character if and only if $\Delta$ is divisible by a prime congruent to $3$ modulo $4$.

\section{Deformation of Theta Integral}
\label{sec:deform}
\subsection{Hecke's Cusp Form}
\label{subsec:Hecke}
We continue using the notation from the previous section and set $(L, Q) := (\af, \Nm/A)$ to be an even integral lattice of signature $(1, 1)$ for a fractional $\Oc_F$-ideal $\af \subset F$ with norm $A$.

By identifying $F \otimes \Rb$ with $\Rb^2$ through the embedding $F \subset \Rb$ and its Galois conjugate, we can view the Grassmannian $\Dc_F$ as the parabola, whose components are parametrized by $\Rb$ via $w \mapsto (e^{w}, e^{-w})$ and $w \mapsto (-e^{w}, -e^{-w})$ respectively. We identify $\Rb$ with the connected component of $\Dc_F$ in the first quadrant. 
 Each $w \in \Rb$ gives an isometry between $L \otimes \Rb = \Rb^2$ and $\Rb^{1, 1}$ by sending $(\lambda_1, \lambda_2)$ to $(\lambda_1 e^{-w} + \lambda_2 e^w, \lambda_1 e^{-w} - \lambda_2 e^w)/(\sqrt{2A})$.
The generator $\varep_\Delta$ of $\Gamma_\Delta$ acts on $\Rb$ by sending $w$ to $w + \log \varep_\Delta$, and $dw$ is the invariant measure.

In \cite{Hecke26}, Hecke constructed a holomorphic cusp form $\vartheta_\af \in S_{1, \rho_\af}$ by integrating the theta kernel
\begin{equation}
  \label{eq:Thetaaf}
  \Theta_\af(\tau, w) := \sqrt{v} \sum_{h \in A_\af} \ef_h \sum_{\lambda \in \af + h} \frac{\lambda e^{-w} + \lambda' e^w}{\sqrt{A}} \ebf \lp 
\frac{(\lambda e^{-w} + \lambda' e^w)^2 }{4A} \tau 
- \frac{(\lambda e^{-w} - \lambda' e^w)^2 }{4A} \overline{\tau} 
\rp
\end{equation}
against the constant function on $\Gamma_\Delta \backslash \Rb$.
Here, $\Theta_\af(\tau, w)$ equals $\Theta_L(\tau, w; p)$ with the polynomial $p:\Rb^{1, 1} \to \Rb$ given by $p(x_1, x_2) := \sqrt{2} x_1$.
He also explicitly calculated the Fourier expansion of $\vartheta_\af$, which is given as follows (see Satz 3 and Satz 4 in \cite{Hecke26}).

\begin{prop}
  \label{prop:vartheta}
Let $\af, L, Q, \Gamma_\Delta$ be as above. For $h \in A_\af$, the following function
\begin{equation}
  \label{eq:varthetah}
  \begin{split}
      \vartheta_{\af, h}(\tau) &:= \int_{\Gamma_\Delta \backslash \Rb}  \Theta_{\af, h} (\tau, w; p) dw = 
\sum_{m \in \Qb, m > 0} c_\af(m, h) q^m, \\
c_\af(m, h) &:= \sum_{\lambda \in L(m, h)} \sgn(\lambda) = \sum_{\begin{subarray}{c} \lambda \in \langle \varep_\Delta \rangle \backslash \af + h\\ \Nm(\lambda) = A m\end{subarray}} \sgn(\lambda), \text{ when } m > 0,
  \end{split}
\end{equation}
is the $\ef_h$-component of a cusp form $\vartheta_\af(\tau) \in S_{1, \rho_\af}$.
\end{prop}

\begin{rmk}
  \label{rmk:1}
Let $L^-$ denote the lattice $L$ with $-Q$ as the quadratic form and
\begin{equation}
  \label{eq:Theta-}
  \Theta^-_\af(\tau, w)  =
\sqrt{v} \sum_{h \in A_\af} \ef_h \sum_{\lambda \in \af + h} \frac{\lambda e^{-w} - \lambda' e^w}{\sqrt{A}} \ebf \lp 
\frac{(\lambda e^{-w} - \lambda' e^w)^2 }{4A} \tau 
- \frac{(\lambda e^{-w} + \lambda' e^w)^2 }{4A} \overline{\tau} 
\rp
\end{equation}
 the corresponding theta kernel.
The same construction above applies and produces a cusp form $\vartheta_\af^- \in S_{1, \overline{\rho_{\af}}}$. If we denote its Fourier coefficients by $c^-_{\af}(m ,h)$, then $c^-_\af(m, h) = c_{\af (\mu)}(m,  \mu h)$ for all $m \in \Qb, h \in \af\df^{-1}/\af$, and any relatively prime to $\df$ element $\mu \in F^\times$ satisfying $\mu > 0$ and $\mu' < 0$.
\end{rmk}

\begin{rmk}
  \label{rmk:-h}
By changing $\lambda$ to $-\lambda$, we have $\vartheta_{\af, -h} = -\vartheta_{\af, h}$ for any $\af$ and $h \in A_\af$. 
\end{rmk}

For each fractional ideal $\bfrak \in I_\df$, recall the equality $A_{\Nm^-(\bfrak)} = A_\Delta$ in \eqref{eq:identify}.
From the discussion there, we know that this only depends on the class $\Bf$ of $\bfrak$ in $\Cl^+(F)$.
Therefore, we can define
\begin{equation}
  \label{eq:varthetaAf}
  \vartheta_{\Bf} := \vartheta_{\Nm^-(\bfrak)}, ~   \vartheta^-_{\Bf} := \vartheta^-_{\Nm^-(\bfrak)}
\end{equation}
It is then easy to check that
 \begin{equation}
   \label{eq:rmk3}
 \vartheta_{\Bf [\df]} = \sigma_{\Delta}( \vartheta_{\Bf}) = -\vartheta_{\Bf}, ~\vartheta^-_{\Bf [\df]} = - \vartheta^-_{\Bf},
 \end{equation}
where $\sigma_\Delta$ is the automorphism defined in \eqref{eq:sigmad}.
Let $\chi$ be a genus character of $F$. By summing $\vartheta_{\Bf}^-$ over the classes $\Bf \in \Cl^+(F)$ with the twist by $\chi$, we can now define a cusp form that only depends on $\chi$
\begin{equation}
  \label{eq:vartheta12}
  \vartheta_{\chi}(\tau) :=  \sum_{\Bf \in \Cl^+(F)} \chi(\Bf) \vartheta_{\Bf}^-(\tau) = \sum_{h \in A_\Delta} \ef_h \sum_{m \in \Qb, \, m > 0} c_\chi(m, h) q^m \in S_{1, \overline{\rho_\Delta}}.
\end{equation}
By \ref{rmk:-h}, we know that the coefficient $c_\chi(m, h)$ satisfies 
\begin{equation}
  \label{eq:conj}
  c_\chi(m, h) = - c_\chi(m, -h) = - c_\chi(m, h')
\end{equation}
for all $m \in \Qb$ and $h \in A_\Delta$.
The form $\vartheta_\chi(\tau)$ can be produced by integrating the theta kernel
\begin{equation}
  \label{eq:Theta12}
  \Theta_{\chi}(\tau, w; S_F) := \sum_{\bfrak \in S_F} \chi(\bfrak) \Theta_{\Nm^-(\bfrak)}^-(\tau, w) \in \Ac_{1, \overline{\rho_\Delta}},
\end{equation}
with respect to $dw$ over $\Gamma_\Delta \backslash \Rb$, where $S_F \subset I_\df$ is any set of representatives of classes in $\Cl^+(F)$.
Even though $\vartheta_{\chi}(\tau)$ does not depend on the choice of $S_F$, both $\Theta_\chi(\tau, w; S_F)$ and the following theta function do
\begin{equation}
  \label{eq:theta}
  \theta_{\chi}(\tau; S_F) :=  \Theta_{\chi}(\tau, 0; S_F) \in \Ac_{1, \overline{\rho_\Delta}}.
\end{equation}
The Fourier coefficient $c_\chi(m, h)$ can now be rewritten as follows.

\begin{prop}
  \label{prop:rewrite}
Let $\chi$ be a genus character. For $n \in  \Nb$ and $h \in A_\Delta$, the Fourier coefficient $c_\chi(m, h)$ of $\vartheta_\chi$ can be written as
\begin{equation}
  \label{eq:cchi1}
  c_\chi \lp \frac{n}{\Delta}, h \rp = 
  \begin{cases}
2 \sr_h \displaystyle\sum_{\begin{subarray}{c} \af \subset \Oc_F \text{},~ \Nm(\af) = n\end{subarray}} {\chi}(\rt(\af, h)),    & \text{if } n \in \Nm(F)_+ \text{ and } \chi \text{ is odd}, \\
0, & \text{otherwise.}
  \end{cases}
\end{equation}
where $\Nm(F)_+$, $\sr_h$ and $\rt(\af, h)$ are defined in \eqref{eq:NmFpm}, \eqref{eq:srh} and \eqref{eq:sqrt} respectively.
\end{prop}

\begin{proof}
  We can use the definition of $\vartheta_\chi$ to write
  \begin{align*}
    c_\chi\lp \frac{n}{\Delta}, h \rp &= \sum_{[\bfrak] \in \Cl^+(F)} \chi(\bfrak) \sum_{\begin{subarray}{c} \lambda \in \Gamma_\Delta \backslash \Nm^-(\bfrak) + h \\ \Nm(\lambda) = -n/\Delta \end{subarray}} \sgn(\lambda)
= \sr_h \sum_{[\bfrak] \in \Cl^+(F)} \chi(\bfrak) \sum_{\begin{subarray}{c}(\mu) \subset \Nm^{-}(\bfrak) \\ \frac{\mu}{\sqrt{\Delta}} = h \in A_\Delta \\ \Nm(\mu) = n\end{subarray}} \sgn(\mu).
  \end{align*}
By the definition of $\Nm(F)_+$ in \eqref{eq:NmFpm}, it is clear that the sum above is empty if $n \not\in \Nm(F)_+$. Therefore, we restrict to the case $n \in \Nm(F)_+$, where $\Nm(\mu) = n$ implies $\Nm((\mu)) = n$. 
Using this and $\chi(\bfrak) = \chi(\bfrak')$, we can continue to write
  \begin{align*}
    c_\chi\lp \frac{n}{\Delta}, h \rp 
&= \sr_h \sum_{[\bfrak] \in \Cl^+(F)} \chi(\bfrak')
\lp 
 \sum_{\begin{subarray}{c}(\mu) \subset \Nm^{-}(\bfrak),~\Nm((\mu)) = n \\ \frac{\mu}{\sqrt{\Delta}} = h \in A_\Delta,~ \mu > 0\end{subarray}} 1 - 
 \sum_{\begin{subarray}{c}(\mu) \subset \Nm^{-}(\bfrak),~\Nm((\mu)) = n \\ \frac{\mu}{\sqrt{\Delta}} = h \in A_\Delta,~ \mu < 0\end{subarray}} 1
\rp\\
&= \sr_h \sum_{[\bfrak] \in \Cl^+(F)} \chi(\bfrak')
\lp 
 \sum_{\begin{subarray}{c}\af \subset \Oc_F,~\Nm(\af) = n \\ 
\af = (\mu)\Nm^-(\bfrak') \text{ with }\\ \frac{\mu}{\sqrt{\Delta}} = h \in A_\Delta\text{and } \mu > 0\end{subarray}} 1 
- 
 \sum_{\begin{subarray}{c}\af \subset \Oc_F,~\Nm(\af) = n \\ 
\af = (\mu)\Nm^-(\bfrak') \text{ with }\\ \frac{\mu}{\sqrt{\Delta}} = -h \in A_\Delta \text{and } \mu > 0\end{subarray}} 1 
\rp\\
&= \sr_h 
 \sum_{\af \subset \Oc_F,~\Nm(\af) = n }
\lp 
 \sum_{\begin{subarray}{c}
[\bfrak] \in \Cl^+(F)\\
\af = (\mu)\Nm^-(\bfrak') \text{ with }\\\frac{\mu}{\sqrt{\Delta}} = h \in A_\Delta\text{and } \mu > 0\end{subarray}}  \chi(\bfrak')
- 
 \sum_{\begin{subarray}{c}
[\bfrak] \in \Cl^+(F)\\
\af = (\mu)\Nm^-(\bfrak') \text{ with }\\\frac{\mu}{\sqrt{\Delta}} = -h \in A_\Delta \text{and } \mu > 0\end{subarray}}  \chi(\bfrak')
\rp\\
&= \sr_h 
\lp
 \sum_{\af \subset \Oc_F,~\Nm(\af) = n} {\chi}(\rt(\af, h)  )
- 
\sum_{\af \subset \Oc_F,~\Nm(\af) = n} {\chi}(\rt(\af, -h) )
\rp
\end{align*}
If $\chi$ is even, i.e.\ $\chi([\df]) = 1$, then $c_\chi(m, h)$ vanishes identically by \eqref{eq:permute} as $-h = \sigma_{\Delta_0}(h)$ and $[\df_{\Delta_0}] = [\df]$. 
This finishes the proof.
\end{proof}
From this proposition, we can deduce some the following properties of the coefficient $c_\chi(m, h)$. 
\begin{cor}
  \label{cor:vanishell}
Let $p \mid \Delta$ be a prime such that $\chi([\df_p]) = -1$ with $\df_p$ the unique prime ideal in $\Oc_F$ above $p$. 
Then $c_\chi(m, h) = 0$ for all $m \in \Qb$ and $h \in A_\Delta$ satisfying $p \mid \mathrm{d}(h)$, where $\mathrm{d}$ is the function defined in \eqref{eq:drm}.
\end{cor}

\begin{proof}
This easily follows from the proposition above and \eqref{eq:permute}.
\end{proof}
\begin{prop}
  \label{prop:cchiprop}
Let $\ell \in \Nb$ be a prime such that there is only one prime $\lf$ in $\Oc_F$ above it. 
Then
\begin{equation}
  \label{eq:cchiprop}
c_\chi \lp  \frac{\ell^2 n}{\Delta}, \ell h \rp = 
  \begin{cases}
    c_\chi \lp \frac{n}{\Delta},  h \rp , & \text{ if } \ell \nmid \Delta \text{ or }\ell \mid \gcd(\Delta, n), \\
    (1 + \chi([\lf])) c_\chi \lp \frac{n}{\Delta},  h \rp, & \text{ if } \ell\mid \Delta \text{ and } \ell \nmid n,
  \end{cases}
\end{equation}
for all $n \in \Zb$ and $h \in A_\Delta$.
\end{prop}

\begin{proof}
By Prop.\ \ref{prop:rewrite}, we can suppose that $n \in \Nm(F)_+$ and $\chi$ is odd, otherwise both sides are identically zero.
Applying Prop.\ \ref{prop:rewrite} gives us
\begin{align*}
c_\chi  \lp \frac{\ell^2 n}{\Delta}, \ell h \rp &= 
2 \sr_{\ell h}
\sum_{\begin{subarray}{c} \ell \af \subset \Oc_F,~ \Nm(\ell \af) =  \ell^2 n \end{subarray}} {\chi}(\rt(\ell \af, \ell h))
\end{align*}
since there is only one prime in $\Oc_F$ above $\ell$.
If $\ell$ is inert, i.e.\ $\ell \nmid \Delta$, then $\sr_{\ell h} = \sr_h$ and $\rt(\ell \af, \ell h) = \rt(\af, h)$ by Lemma \ref{lemma:countmul}, which completes the proof. 
Suppose $\ell$ is ramified, i.e.\ $\ell \mid \Delta$, then Lemma \ref{lemma:countmul} implies that 
\begin{align*}
c_\chi  \lp \frac{\ell^2 n}{\Delta}, \ell h \rp &= 
\frac{2 \sr_{\ell h}}{\#(G_{\Delta, \ell} \cap \mathrm{G}_{\Delta, \overline{h}})}
\sum_{\begin{subarray}{c} {\af} \subset \Oc_F\\ \Nm({\af}) = n \end{subarray}}
\sum_{\sigma \in G_{\Delta, \ell}} \chi(\rt(\af, \sigma(h))).
\end{align*}
Let $d_0$ and $\Delta_0$ be as in Lemma \ref{lemma:shval}.
We have now several cases
\begin{description}
\item[Case 1] $\ell \mid \mathrm{d}(h)$. 
\item[Case 2] $\ell \nmid \mathrm{d}(h)$ and $\chi([\lf]) = -1$.
\item[Case 3] $\ell \nmid \mathrm{d}(h)$, $\chi([\lf]) = 1$ and $\Delta_0 \mid \mathrm{d}(\ell h)$. 
\item[Case 4(a)] $\ell \nmid \mathrm{d}(h)$, $\chi([\lf]) = 1$, $\Delta_0 \nmid \mathrm{d}(\ell h)$ and $d_0 \mid \mathrm{d}(h)$.
\item[Case 4(b)] $\ell, d_0 \nmid \mathrm{d}(h)$, $\chi([\lf]) = 1$, $\Delta_0 \nmid \mathrm{d}(\ell h)$ and $d_0 \mid \mathrm{d}(\ell h)$.
\item[Case 5] Otherwise, i.e.\ $\ell \nmid \mathrm{d}(h)$, $\chi([\lf]) = 1$ and $d_0, \Delta_0 \nmid \mathrm{d}(\ell h)$.
\end{description}
For the elements $\rt(\af, h), \rt(\af, \sigma_\ell(h))$ to be non-zero, we can suppose that $\ell \mid n$ in Case 1 and $\ell \nmid n$ for the other cases.
In Case 1, the proof of the inert case carries through since $\sr_{\ell h} = \sr_{h}$ and $G_{\Delta, \ell} \cap \mathrm{G}_{\Delta, \overline{h}} = G_{\Delta, \ell}$.
For Case 2, the left hand side vanishes by Corollary \ref{cor:vanishell}, and equals the right hand side.
For Case 3, since $\chi$ is odd, i.e.\ $\chi([\df]) = -1$, there is a prime $p \mid \Delta_0$ such that $\chi([\df_p]) = -1$, which implies that $p \neq \ell$.
Therefore, $p \mid \mathrm{d}(h)$ and both sides vanish again by Corollary \ref{cor:vanishell}.

For the other cases, it suffices to establish the equality
$$
\frac{2 \sr_{\ell h}}{\#(G_{\Delta, \ell} \cap \mathrm{G}_{\Delta, \overline{h}})} 
= 2 \mathrm{s}_h
$$
and the rest follows as in the previous cases.
In Case 4(a), we have $\mathrm{s}_{\ell h} = \mathrm{s}_h = 2$ from Lemma \ref{lemma:shval}. The conditions $\ell \nmid \mathrm{d}(h)$ and $d_0 \mid \mathrm{d}(h) \mid \mathrm{d}(\ell h)$ together imply that $\sigma_\ell(h) \neq h = \sigma_{d_0}(h)$, i.e.\ $G_{\Delta, \ell} \cap \mathrm{G}_{\Delta, \overline{h}}$ is trivial. 
In Case 4(b), Lemma \ref{lemma:shval} again implies that $\mathrm{s}_{\ell h} = 2 \mathrm{s}_h = 2$.
On the other hand, we have $\ell h = \ell \sigma_{d_0}(h) = \ell \sigma_\ell(h)$, which means $h, \sigma_{d_0}(h)$ and $\sigma_\ell(h)$ all have the same $p$-component for $p \neq \ell$. Since $\ell \nmid \mathrm{d}(h)$ implies that $h \neq \sigma_\ell(h)$, we must have $\sigma_{d_0}(h) = h$ or $\sigma_\ell(h)$ by \eqref{eq:equivalence}. The former cannot happen precisely since $d_0 \nmid \mathrm{d}(h)$. Therefore $\sigma_{d_0}(h) = \sigma_\ell(h)$ and $\sigma_\ell \in \mathrm{G}_{\Delta, \overline{h}}$ by definition.
In the last case, we again have $\mathrm{s}_{\ell h} = \mathrm{s}_h$ and $G_{\Delta, \ell} \cap \mathrm{G}_{\Delta, \overline{h}}$ being trivial. 
This finishes the proof.
\end{proof}

\subsection{Deformed Integral}
Let $(L, Q) = (\af, \Nm/A)$ as in the previous section.
In \cite{CL20}, we considered the following deformed integral
\begin{equation}
  \label{eq:hvartheta}
  \hat{\vartheta}_{\af}(\tau) := \int_{0}^{\log \varep_\Delta}  \Theta_{\af} (\tau, w) w dw  \in \Ac_{1, \rho_{\af}}.
\end{equation}
Note that this function is real-analytic on $\Hb$ with exponential decay near the cusp.
Furthermore, it depends on the choice of the fundamental domain.
In Prop.\ 5.5 in \cite{CL20}, we have calculated its Fourier expansion, which we recall here.
For each $\lambda \in F^\times$, denote 
\begin{equation}
\label{eq:r}
r(\lambda) := \left| \frac{\lambda}{\lambda'} \right|.
\end{equation}
For each $\Gamma_\Delta$-orbit $\Lambda \subset F^\times$, choose a representative $\lambda_0 \in \Lambda$ such that $1 \le r(\lambda_0) < \varep_\Delta^2$. 
Now define
\begin{equation}
  \label{eq:a}
  a(\Lambda) := 
  \begin{cases}
    \sgn(\lambda_0) \log r(\lambda_0), & r(\lambda_0) \neq 1, \\
-\log \varepsilon_\Delta, & r(\lambda_0) = 1.
  \end{cases}
\end{equation}
These are the ``holomorphic'' part coefficients of $\hat{\vartheta}_{\af, h}$ with $h \in \af\df^{-1}/\af$. For the non-holomorphic part, we first define incomplete Gamma function 
\begin{equation}
  \label{eq:gammas}
  \Gamma(s, x) := \int^\infty_x e^{-t} t^s \frac{dt}{t}.
\end{equation}
for $s \in \Cb$ and $x > 0$. Then we can define the following non-holomorphic function on $\Hb$
\begin{equation}
  \label{eq:nonhol}
  \begin{split}
    \tilde{\vartheta}_{\af, h}^*(\tau) &:= \sum_{\Lambda \in \Gamma_\Delta \backslash (\af + h), Q(\Lambda) < 0} \sgn(\Lambda) \Gamma(0, -4\pi Q(\Lambda) v) q^{Q(\Lambda)},  \\
\tilde{\Theta}^*_{\af, h}(\tau) &:= \sum_{\lambda \in \af + h} \frac{\sgn(\lambda - \lambda')}{\sqrt{\pi}} \Gamma\lp \half,  \frac{\pi v | \lambda - \lambda'|^2}{A}  \rp q^{ Q(\lambda)}.
  \end{split}
\end{equation}
Furthermore, under the lowering operator $L_{\tau}$, these functions become modular and satisfy
\begin{equation}
  \label{eq:L_nonhol}
  L_\tau \tilde{\vartheta}_{\af, h}^*(\tau) = - v \overline{\vartheta^-_{\af, h}(\tau)}, \;
L_\tau \tilde{\Theta}^*_{\af, h}(\tau) = - v^{} \overline{\Theta^-_{\af, h}(\tau, 0)}.
\end{equation}
They turn out to form the ``non-holomorphic'' part of $\hat{\vartheta}_{\af, h}(\tau)$.
\begin{prop}[Prop.\ 5.1 in \cite{CL20}]
  \label{prop:5.1}
The function $\hat{\vartheta}_{\af}(\tau) = \sum_{h \in \af\df^{-1}/\af} \ef_h \hat{\vartheta}_{\af, h}(\tau)$ has the Fourier expansion
\begin{equation}
  \label{eq:hvarfe}
  2 \hat{\vartheta}_{\af, h}(\tau) =
  \sum_{m \in \Qb_{> 0}} \tilde{c}_\af(m, h) q^m
  - \tilde{\vartheta}^*_{\af, h}(\tau) - \log \varep_\Delta \cdot \tilde{\Theta}^*_{\af, h}(\tau),
\end{equation}
where $\tilde{c}_\af(m, h) := \sum_{\Lambda \in \Gamma_\Delta \backslash L(m, h)} a(\Lambda)$.
\end{prop}
Let $\chi$ be a genus character of $F$ and $S_F$ as in \eqref{eq:Theta12}.
As before, we can sum $\chi(\bfrak) \hat{\vartheta}_{\Nm^-(\bfrak)}$ over $\bfrak \in S_F$ to define 
\begin{equation}
  \label{eq:hvartheta}
  \begin{split}
      \hat{\vartheta}_{\chi}(\tau; S_F) &:= \sum_{\bfrak \in S_F} \chi(\bfrak) \hat{\vartheta}_{\Nm^-(\bfrak)}(\tau) = \sum_{h \in A_\Delta } \ef_h \sum_{m \in \Qb^\times} \hat{c}_\chi(m,h,  v)\ebf(mu)  \in \Ac_{1, \rho_{\Delta}}, \\
\tilde{c}_\chi(m, h) &:= \lim_{v \to \infty} \hat{c}_\chi(m, h, v) e^{2\pi m v}=  \frac{1}{2} \sum_{\bfrak \in S_F}\chi(\bfrak) \tilde{c}_{\Nm^-(\bfrak)}(m, h) \in \Rb.
  \end{split}
\end{equation}
The lemma below gives a bound of the Fourier coefficients $\hat{c}_\chi(m, h, v)$.
\begin{lemma}
\label{lemma:bound}
For every $h \in A_\Delta$ and $\epsilon > 0$, we have the asymptotic 
  \begin{equation}
  \label{eq:bound}
\hat{c}_\chi(m, h, v)e^{2 \pi |m| v}  =  \tilde{c}_\chi(m, h)  + O_{F, S_F, \epsilon}  \lp |m|^\epsilon \frac{v + 1}{v^{3/2}} \rp,
\end{equation}
where $S_F$ is the set of ideals as in \eqref{eq:Theta12} and $\tilde{c}_\chi(m, h) = 0$ for $m \le 0$.
\end{lemma}
\begin{rmk}
  \label{rmk:asymptotic}
When $m > 0$, it is easy to check that $\tilde{c}_\chi(m, h) = O_{F, \epsilon}(m^\epsilon)$ for any $\epsilon > 0$.
\end{rmk}

\begin{proof}
When $m < 0$, we see from the definition that $\tilde{c}_\chi(m, h) = 0$ and
\begin{equation}
\label{eq:bound1}
|\hat{c}_\chi(m, h, v)| \ll_F
 \sum_{\begin{subarray}{c} \bfrak \in S_F \\ (\lambda) \subset \Nm^-(\bfrak)\df^{-1} \\ \Nm(\lambda) = m \end{subarray}} \lp \Gamma ( 0, -{4\pi m v} )
+ \sum_{\begin{subarray}{c} n \in \Zb\\ \lambda_n \neq \lambda_n' \end{subarray}} \Gamma \lp \half , \pi v {(\lambda_n - \lambda_n')^2}{} \rp \rp
 e^{-{2\pi m v}},
\end{equation}
where we have denote $\lambda_n := \lambda \varep^n$ and chosen the index $n$ such that $r(\lambda_n) \ge 1$ if and only if $n \ge 0$.
Using the bounds $\Gamma(s, x) \ll_s  x^{s-1}e^{-x}$, we can write
$$
\Gamma(0, -4\pi m v) e^{-2\pi m v} \ll \frac{e^{-2\pi |m| v}}{|m|v}, ~
\Gamma \lp \half , \pi v {(\lambda_n - \lambda_n')^2}{} \rp e^{-2\pi m v} \ll \frac{e^{-{\pi |m| v (r(\lambda_n) + r(\lambda_n)^{-1})}}}{\sqrt{v} |\lambda_n - \lambda_n'|}.
$$
Since $\lambda_n - \lambda_n' \neq 0$ and $\bfrak' \df (\lambda_n) \subset \Oc_F$, we have $ |\lambda_n - \lambda_n'|^{-1} \ll_{ S_F} 1$.
Since the function $x + x^{-1}$ is monotonically increasing for $x \ge 1$ and 
, we have $r(\lambda_n) + r(\lambda_n)^{-1} - 2 > C \cdot  |n|$ for all $n \in \Zb$ with some constant $C = C(S_F) > 0$ depending on $S_F$.
Summing over $n \in \Zb$ gives us
\begin{align*}
 \sum_{n \in \Zb} \Gamma \lp \half , \pi v {(\lambda_n - \lambda_n')^2}{} \rp 
 e^{-{2\pi |m| v}} 
&\ll_{S_F} \frac{e^{-2 \pi |m| v}}{\sqrt{v}} \sum_{n \ge 0} e^{-\pi |m| v C n}
 \ll_{S_F} \frac{v + 1}{ v^{3/2}} e^{-2\pi |m| v},
\end{align*}
since $\frac{1}{1 - e^{-x}} \le \frac{ x + 1}{x}$ for $x \ge 0$. 
The sum on the right hand side of \eqref{eq:bound1} counts the number of ideals in $\Oc_F$ with norm $\Delta m$, which is bounded above by $C_\epsilon \cdot (\Delta m)^\epsilon$ for any $\epsilon > 0$ and a constant $C_\epsilon > 0$.
Putting everything together, we obtain the bound \eqref{eq:bound} for $m < 0$.
When $m > 0$, we have
$$
|\hat{c}_\chi(m, h, v) - \tilde{c}_\chi(m, h) e^{-2\pi m v}| \ll_F  
 \sum_{\begin{subarray}{c} \bfrak \in S_F \\ (\lambda) \subset \Nm^-(\bfrak)\df^{-1} \\ \Nm(\lambda) = m \end{subarray}} 
 \sum_{n \in \Zb} \Gamma \lp \half , \pi v {(\lambda_n - \lambda_n')^2} \rp 
 e^{-{2\pi m v}}.
$$
Proceeding as before finishes the proof.
\end{proof}
The function $\hat{\vartheta}_\chi(\tau, S_F)$ will be the input of the theta lift that constructs the preimage of an Eisenstein series under the lowering operator since
\begin{equation}
  \label{eq:basic}
2  L_\tau  \hat{\vartheta}_\chi(\tau; S_F) = 
\vartheta^c_{\chi}(\tau) + \log \varep_\Delta \cdot \theta^c_\chi(\tau; S_F) \in \Ac_{-1, \rho_\Delta},
\end{equation}
where $\theta_\chi(\tau; S_F)$ is defined in \eqref{eq:theta}.

\subsection{Siegel-Weil Formula}
\label{subsec:SW}
In this section, we will consider the theta function $\Theta_L(z, w(\underline{z}))$ in Example \ref{ex:M2Z} for $\underline{z}$ at CM points.
Let $F = \Qb(\sqrt{\Delta})$ be a real quadratic field, $\chi = \chi_{\Delta_1, \Delta_2}$ an odd genus character and $K/F$ the corresponding unramified, quadratic extension. Since $\chi$ is odd, $K$ is a CM field and contains the imaginary quadratic fields $K_j := \Qb(\sqrt{\Delta_j})$ for $j = 1, 2$.
Let $Z(\Delta_j)$ be the set of all CM points in $\Gamma\backslash \Hb$ of discriminants $\Delta_j$, $\Cl_j$, $U_j$, $h_j$ and $w_j$ the class group, unit group in $K_j$ and its size respectively.
Then
\begin{equation}
  \label{eq:Zchi}
  Z_\chi := \frac{4}{w_1w_2} \sum_{\mathbf{z} \in Z(\Delta_1) \times Z(\Delta_2)} \mathbf{z}
\end{equation}
is a divisors on $(\Gamma \backslash \Hb)^2$. 

For an integral ideal $\af \subset \Oc_F$, define the character sum $\sigma_\chi$ by
\begin{equation}
  \label{eq:sigma}
  \sigma_\chi(\af) := \sum_{\af \subset \bfrak } \chi(\bfrak) = \rho_{K/F}(\af),
\end{equation}
where $\rho_{K/F}(\af)$ counts the number of ideals in the ring of integers $\Oc_K \subset K$ with $\af$ as their relative norm.
Since $K/\Qb$ is abelian, the function $\rho$ satisfies $\rho_{K/F}(\af) = \rho_{K/F}(\af')$.
We can now define the following function
\begin{equation}
  \label{eq:Eisen_diagonal}
  \Ec_\chi(z) :=  \frac{2h_1h_2}{w_1w_2} + \sum_{\begin{subarray}{c} \lambda_0 \in \df^{-1}\\ \lambda_0 > 0 > \lambda'_0 \end{subarray}} \sigma_\chi((\lambda_0) \df) \ebf(\lambda_0 z + \lambda'_0 \overline{z}).
\end{equation}
One can show that $y\Ec_\chi(z) \in \Ac_0$ is the image of the diagonal restriction of the incoherent Eisenstein series in \cite{GZ85} under the lowering operator $L_{\tau, 2}$. 
Note that $\overline{\Ec_\chi(z)} = \Ec_\chi(z)$.
As an application of the Siegel-Weil formula, we have the following result (see \cite[Prop.\ 4.5]{BKY12}).
\begin{prop}
  \label{prop:SW}
Let $\Theta_L(z, w(\underline{z}))$ be the theta function in \eqref{eq:Theta}. Then 
\begin{equation}
  \label{eq:SW}
 \Theta_L(z, Z_\chi) = {2} \Ec_\chi(z)
\end{equation}
for all $z \in \Hb$.
\end{prop}
\subsection{Counting Result I}
\label{subsec:countI}
In the notations of sections \ref{subsec:F}, \ref{subsec:Hecke} and \ref{subsec:SW}, we can define a quantity
\begin{equation}
  \label{eq:C}
C_\chi(\mu_0) =
 \sum_{t \mid (\mu_0 ), \, t \in \Nb}
 c_{\chi} \lp \frac{\Nm(\mu_0/t)}{\Delta}, \frac{\mu_0/t}{\sqrt{\Delta}} \rp
\end{equation}
for any $\mu_0 \in \Oc_F$.
This turns out to be related to the ideal counting function $\rho_{K/F}((\mu_0))$ defined in Eq.\ \eqref{eq:sigma}.

\begin{prop}
  \label{prop:count}
Let $\chi$ be an odd genus character of $F$ and $K/F$ the corresponding CM extension.
For any totally positive $\mu_0 \in \Oc_F$,
\begin{equation}
  \label{eq:count}
  C_\chi(\mu_0) = 2 \sigma_\chi((\mu_0)) =  2 \rho_{K/F}((\mu_0)),
\end{equation}
where $\rho_{K/F}(\af)$ is the ideal-counting function defined in Eq.\ \eqref{eq:sigma}.
\end{prop}

\begin{proof}
  We will proceed by considering the prime factorization of the ideal $(\mu_0)$. 
Let $\ell$ be a rational prime and $\ord_\ell( \Nm(\mu_0)) = a$. 
If $\ell$ is inert in $\Oc_F$, then $2 \mid a$, $\chi([\ell]) = 1$ and 
$$
C_\chi(\mu_0) = (a + 1) C_\chi(\mu_0/\ell^{a/2}) = \rho_{K/F}(\ell^a) C_\chi(\mu_0/\ell^{a/2})
$$ 
by Prop.\ \ref{prop:cchiprop}.
If $\ell$ is ramified in $\Oc_F$, i.e.\ $\ell \Oc_F = \lf^2$, then there are two cases depending on the parity of $a$. When $a$ is even, Prop.\ \ref{prop:cchiprop} again implies that 
\begin{align*}
  C_\chi(\mu_0) &= 
 \sum_{t \mid (\mu_0/\ell^{a/2} )} \sum_{r = 0}^{a/2}
 c_{\chi} \lp \frac{\Nm(\mu_0/(t \ell^r))}{\Delta}, \frac{\mu_0/(t \ell^r)}{\sqrt{\Delta}} \rp\\
&=
\lp \frac{a}{2} ( 1 + \chi([\lf])) + 1 \rp
  \sum_{t \mid \widetilde{\mu_0} }
 c_{\chi} \lp \frac{\Nm(\widetilde{\mu_0}/t )}{\Delta}, \frac{\widetilde{\mu_0}/t}{\sqrt{\Delta}} \rp 
= \rho_{K/F}(\lf^a) C_\chi(\widetilde{\mu_0}),
\end{align*}
where $\widetilde{\mu_0} := \mu_0/\ell^{a/2}$.
When $a$ is odd and $\chi([\lf]) = -1$, Corollary \ref{cor:vanishell} implies that $C_\chi(\mu_0) = 0 = 2 \rho_{K/F}((\mu_0))$.
Suppose that $a$ is odd and $\chi([\lf]) = 1$. Applying Prop.\ \ref{prop:cchiprop} as before yields
\begin{equation}
\label{eq:step2}
  C_\chi(\mu_0) = 
 \sum_{t \mid (\widetilde{\mu_0} )} \sum_{r = 0}^{(a-1)/2}
 c_{\chi} \lp \frac{\ell^r \Nm(\widetilde{\mu_0}/t )}{\Delta}, \frac{\ell^r\widetilde{\mu_0}/t }{\sqrt{\Delta}} \rp
=
 \frac{a+1}{2}
  \sum_{t \mid \widetilde{\mu_0} }
 c_{\chi} \lp \frac{\Nm(\widetilde{\mu_0}/t )}{\Delta}, \frac{\widetilde{\mu_0}/t}{\sqrt{\Delta}} \rp,
\end{equation}
where $\widetilde{\mu_0} := \mu_0/\ell^{(a-1)/2}$ and $\ord_\lf (\widetilde{\mu_0}) = 1$.
Now we can apply the Chebotarev density theorem to choose a prime $\tilde{\lf}$ relatively prime to $\mu_0$ in the class $[\lf] \in \Cl^+(F)$.
That means $\lf \tilde{\lf} = (\alpha)$ with $\alpha \in \Oc_F$ totally positive.
Mimicking the proof of Prop.\ \ref{prop:cchiprop}, we have
\begin{align*}
  c_\chi  \lp \frac{n}{\Delta}, h\rp &= 
\sum_{\begin{subarray}{c} \af \subset \Oc_F\\ \Nm(\af) = n\end{subarray}} 
\chi (\rt( \af, h))
=
\frac{1}{2} \sum_{\begin{subarray}{c} \af \subset \Oc_F\\ \Nm(\af) = n\end{subarray}} 
\chi (\rt( \af(\alpha), \alpha h)) + \chi(\rt( \af(\alpha'), \alpha' h))\\
&= 
\frac{1}{2} \sum_{\begin{subarray}{c} \tilde{\af} \subset \Oc_F\\ \Nm(\tilde{\af}) = n\Nm(\alpha)\end{subarray}} 
\chi (\rt( \tilde{\af}, \alpha h) )
=
\frac{1}{2}
  c_\chi  \lp \frac{n\Nm(\alpha)}{\Delta}, \alpha h\rp
\end{align*}
for any $n \in \Zb$ and $h \in A_\Delta$ with $\sigma_\ell(h) = h$. Note that the second step follows from Lemma \ref{lemma:countmul} and $\alpha h = \alpha' h \in A_\Delta$ since $\sigma_\ell(h) = h$. 
Substituting this into \eqref{eq:step2} gives us
\begin{align*}
  C_\chi(\mu_0) &= 
 \frac{a+1}{4}
  \sum_{t \mid \widetilde{\mu_0} }
 c_{\chi} \lp \frac{\Nm(\widetilde{\mu_0}\alpha/t )}{\Delta}, \frac{\widetilde{\mu_0}\alpha/t}{\sqrt{\Delta}} \rp =
 \frac{a+1}{2}
  \sum_{t \mid \widetilde{\mu_0} \alpha /\ell }
 c_{\chi} \lp \frac{\Nm(\widetilde{\mu_0}\alpha/t )/\ell^2}{\Delta}, \frac{\widetilde{\mu_0}\alpha/(t\ell)}{\sqrt{\Delta}} \rp\\
&= \frac{\rho_{K/F}(\lf^a)}{\rho_{K/F}(\tilde{\lf})} C_\chi(\widetilde{\mu_0}\alpha/\ell).
\end{align*}
Notice that $(\widetilde{\mu_0}\alpha/\ell)$ is not divisible by $\lf$.

From this, we can suppose that $\mu_0$ is only divisible by split primes in $\Oc_F$, i.e.\
\begin{equation}
\label{eq:mu0}
(\mu_0) = \prod  \underline{\lf}^{\underline{a}} (\underline{\lf'})^{\underline{b}}
\end{equation}
where $\underline{\lf} = (\lf_j)_{1 \le j \le J}$ with $\Nm(\underline{\lf}) = \underline{\ell} = (\ell_j)_{1 \le j \le J}$ a set of distinct rational primes and $\underline{a}, \underline{b} \in \Nb^{J}$. 
Here, the operations on vectors are carried out componentwisely and $\prod$ sends a vector to the product of its components.
WLOG, we take $\underline{a} \le \underline{b}$. 
Then the divisors of $\mu_0$ in $\Nb$ are of the form 
$$t(\underline{s}) := \prod \underline{\ell}^{\underline{s}},~ \underline{s} \le \underline{a}.$$
For each such $t(\underline{s})$, the ideals with norm $\Nm(\mu_0/t(\underline{s}))$ are exactly given by
\begin{equation}
  \label{eq:aur}
  \af(\underline{r}) := \frac{(\mu_0)}{(t(\underline{s}))} \prod \Nm^-(\underline{\lf}^{\ur}) = \prod \underline{\lf}^{\ua - \us + \ur} (\underline{\lf}')^{\ub - \us - \ur} 
,~ \underline{s} - \underline{a} \le \underline{r} = (r_j)_{1 \le j \le J} \le \underline{b} - \underline{s},
\end{equation}
and 
\begin{equation*}
  \chi \lp\rt\lp \af(\underline{r}),  \frac{\mu_0/t(\underline{s})}{\sqrt{\Delta}} \rp \rp =
\prod \chi(\underline{\lf})^{\ur}.
\end{equation*}
Therefore, we can write
\begin{align*}
  C_\chi(\mu_0) &= 
2 \sum_{\begin{subarray}{c} \underline{s} \in \Nb^J\\ \underline{s} \le \underline{a}\end{subarray}}
\sum_{\begin{subarray}{c} \underline{r} \in \Zb^J\\ \underline{s} - \underline{a} \le \underline{r} \le \underline{b} - \underline{s}\end{subarray}}
\chi \lp\rt\lp \af(\underline{r}),  \frac{\mu_0/t(\underline{s})}{\sqrt{\Delta}} \rp \rp
= 
2\prod_{1 \le j \le J}
\lp
\sum_{s_j = 0}^{a_j} \sum_{r_j = s_j - a_j}^{b_j - s_j} \chi(\lf_j)^{r_j}
\rp.
\end{align*}
Since $\chi$ is a genus character, $\chi(\lf_j) = \chi(\lf_j') = \pm 1$. It is straightforward to verify that 
$$
\sum_{s  = 0}^{a } \sum_{r  = s  - a }^{b  - s } \epsilon^r
= \rho_{K/F}(\lf^a (\lf')^b)
$$
for any $a, b \in \Nb$ and $\lf$ with $\chi(\lf) = \epsilon$. Therefore, we have $C_\chi(\mu_0) = 2 \rho_{K/F}((\mu_0))$ and finished the proof.
\end{proof}

\subsection{Counting Result II}
\label{subsec:countII}
Recall that $S_F \subset I_\df$ is a set of ideal representing classes in $\Cl^+(F)$ as in \eqref{eq:Theta12}.
By replacing the coefficients $c_\chi$ of $\vartheta_\chi$ in $C_\chi$ from the previous section with the coefficients $\tilde{c}_\chi$ of $\hat{\vartheta}_\chi$ in \eqref{eq:hvartheta}, we can define
\begin{equation}
  \label{eq:Ctilde}
\tilde{C}_\chi(\lambda_0; S_F) =
 \sum_{t \mid \sqrt{\Delta} \lambda_0, \, t \in \Nb}
 \tilde{c}_{\chi} \lp \Nm(\lambda_0/t), \lambda_0/t \rp
\end{equation}
for any $\lambda_0 \in \df^{-1}$. 
This quantity appears later in the Fourier expansion of a theta lift (see Prop.\ \ref{prop:FE1}) and the special value of the higher Green's function (see Theorem \ref{thm:main1}).
From definition, it is clear that there exists $\alpha(\lambda_0, S_F) \in F^\times \subset \Rb^\times$ unique up to sign such that
\begin{equation}
  \label{eq:alpha0}
   \tilde{C}_\chi(\lambda_0) = \log|\alpha(\lambda_0, S_F)|.
\end{equation}
A good understanding of the factorization of $\alpha(\lambda_0, S_F)$ directly leads to Theorem \ref{thm:main}.
If $\Nm(\lambda_0) \le 0$, then Lemma \ref{lemma:bound} implies that $\tilde{C}_\chi(\lambda_0) = 0$, and hence $\alpha(\lambda_0, S_F) = \pm 1$. 
Therefore we will focus the case $\Nm(\lambda_0) > 0$.
In this section, we will factor the algebraic number $\alpha(\lambda_0, S_F)$ into two parts. The factorization of the first part is nice and independent of $S_F$. The second part eventually vanishes when substituted into the expression \eqref{eq:fundeq} for the special value of the higher Green's function.
The result is as follows.

\begin{prop}
  \label{prop:fac}
For $\lambda_0 \in \df^{-1}$ with positive norm, there exists $c_\bfrak \in \Rb$ independent of $\lambda_0$ and $\gamma(\lambda_0; S_F) \in \Oc_F$ such that
\begin{equation}
  \label{eq:facCtilde}
  \tilde{C}_\chi(\lambda_0) = \frac{1}{2h_F} \log \left| \frac{\gamma(\lambda_0; S_F)}{\gamma(\lambda_0; S_F)'}\right| + \sum_{\bfrak \in S_F} c_\bfrak \sum_{t \mid \sqrt{\Delta} \lambda_0} c_{\Nm^-(\bfrak)} \lp \Nm(\lambda_0/t), \lambda_0/t \rp,
\end{equation}
where $h_F$ is the class number of $F$, and for any prime $\lf$ of $\Oc_F$
\begin{equation}
  \label{eq:facgamma}
\frac{1}{2h_F}  \ord_{\lf}(  \gamma(\lambda_0; S_F) ) = 
  \begin{cases}
\rho_{K/F}((\lambda_0) \df \lf)(a+1),    & \chi(\lf) = -1, \lf \neq \lf', \ord_{\lf}((\lambda_0) \df) = a,\\
0, & \text{ otherwise,}
  \end{cases}
\end{equation}
where $\rho_{K/F}$ is the ideal counting function in \eqref{eq:sigma}.
\end{prop}

In order to prove this proposition, we need the following analogues of Prop.\ \ref{prop:rewrite} and \ref{prop:cchiprop} for $\tilde{c}_\chi(m, h)$. 
To state them, it is convenient to fix generators $\mu_\af$ of the principal ideal $\af^{h_F}$ for every ideal $\af \subset \Oc_F$ such that 
\begin{equation}
  \label{eq:mult}
  \mu_{\af\bfrak} = \mu_{\af} \mu_{\bfrak},\quad \mu_{\af'} = \mu_\af' 
\end{equation}
for every $\af, \bfrak \subset \Oc_F$.
\begin{prop}
  \label{prop:rewrite2}
Let $\chi$ be an odd genus character, $ n\in \Nb$ and $h \in A_\Delta$.
Then there exist constants $c_\bfrak \in \Rb$ independent of $n$ or $h$ for each $\bfrak \in S_F$, and $r \in \frac{1}{h_F} \Zb \subset \Qb$ such that
\begin{equation}
  \label{eq:tcchi}
  \tilde{c}_\chi \lp \frac{n}{\Delta}, h\rp = 
r \log \varep_F +
\tilde{\crm}_\chi \lp \frac{n}{\Delta}, h \rp
+  \sum_{\bfrak \in S_F} c_\bfrak c_{\Nm^-(\bfrak)} \lp \frac{n}{\Delta}, h \rp
\end{equation}
where
\begin{equation}
  \label{eq:tcrmchi}
  \tilde{\crm}_\chi \lp \frac{n}{\Delta}, h \rp := 
\frac{\sr_h}{h_F} \sum_{\af \subset \Oc_F,~ \Nm(\af) = n} \chi(\rt(\af, h)) \log \left|
\frac{\mu_{\af}}{\mu'_{\af}} \right|.
\end{equation}
\end{prop}
\begin{proof}
As in the proof of Prop.\ \ref{prop:rewrite}, we can write
$$
\tilde{c}_\chi \lp \frac{n}{\Delta}, h \rp
= \frac{\sr_h}{2} \sum_{\bfrak \in S_F} \chi(\bfrak) \sum_{\begin{subarray}{c}(\mu) \subset \Nm^{-}(\bfrak) \\ \frac{\mu}{\sqrt{\Delta}} = h \in A_\Delta \\ \Nm((\mu)) = n\end{subarray}} \sgn(\mu) \log \left|  \frac{\mu}{\mu'}\right|,
$$  
where the generator $\mu$ is chosen appropriately as in the definition of $a(\Lambda)$ in \eqref{eq:a}. Continue as in the proof of Prop.\ \ref{prop:rewrite}, we have
$$
\tilde{c}_\chi \lp \frac{n}{\Delta}, h \rp
= \frac{\sr_h}{2} \sum_{\bfrak \in S_F} \chi(\bfrak')
\lp 
 \sum_{\begin{subarray}{c}\af \subset \Oc_F,~\Nm(\af) = n \\ 
\af = (\mu)\Nm^-(\bfrak') \text{ with }\\ \frac{\mu}{\sqrt{\Delta}} = h \in A_\Delta\text{and } \mu > 0\end{subarray}} \log  \left|  \frac{\mu}{\mu'}\right|
- 
 \sum_{\begin{subarray}{c}\af \subset \Oc_F,~\Nm(\af) = n \\ 
\af = (\mu)\Nm^-(\bfrak') \text{ with }\\ \frac{\mu}{\sqrt{\Delta}} = -h \in A_\Delta \text{and } \mu > 0\end{subarray}} \log  \left|  \frac{\mu}{\mu'}\right|
\rp
$$
Given an ideal $\af \subset \Oc_F$ satisfying the conditions in the summand, the principal ideal $\af^{h_F}$ is generated by $\mu^{h_F} \tfrac{\mu_{\bfrak}'}{\mu_\bfrak}$, which then differs from the fixed generator $\mu_\af$ by a unit in $\Oc_F^\times$. Therefore 
\begin{align*}
   \sum_{\begin{subarray}{c}\af \subset \Oc_F,~\Nm(\af) = n \\ 
\af = (\mu)\Nm^-(\bfrak') \text{ with }\\ \frac{\mu}{\sqrt{\Delta}} = h \in A_\Delta\text{and } \mu > 0\end{subarray}} \log  \left|  \frac{\mu}{\mu'}\right|
&= 
2 r_{\bfrak, n, h} \log |\varep_F| + 
\frac{1}{h_F}
   \sum_{\begin{subarray}{c}\af \subset \Oc_F,~\Nm(\af) = n \\ 
\af = (\mu)\Nm^-(\bfrak') \text{ with }\\ \frac{\mu}{\sqrt{\Delta}} = h \in A_\Delta\text{and } \mu > 0\end{subarray}} \log  \left|  \frac{\mu_\af}{\mu'_\af}\right|\\
& \quad + 
\frac{2}{h_F}
 \log  \left|  \frac{\mu_\bfrak}{\mu'_\bfrak}\right|
\sum_{\begin{subarray}{c}(\mu) \subset \Nm^{-}(\bfrak) \\ \frac{\mu}{\sqrt{\Delta}} = h \in A_\Delta \\ \Nm((\mu)) = n,~ \mu > 0\end{subarray}} 1
\end{align*}
for $r_{\bfrak, n, h} \in \frac{1}{h_F} \Zb \subset \Qb$.
Summing over $\bfrak \in S_F$, we obtain \eqref{eq:tcchi} with
 $r:= \sr_h \sum_{\bfrak \in S_F} \chi(\bfrak)( r_{\bfrak, n, h} - r_{\bfrak, n, -h})$ and $c_\bfrak := \frac{\chi(\bfrak)}{h_F}  \log  \left|  \frac{\mu_\bfrak}{\mu'_\bfrak}\right|$.
\end{proof}

Using the proposition above, we can deduce the following analogue of Prop.\ \ref{prop:cchiprop}.
\begin{cor}
  \label{cor:cchiprop}
  Let $\ell \in \Nb$ be a prime such that there is only one prime $\lf$ in $\Oc_F$ above it. 
Then
\begin{equation}
  \label{eq:cchiprop}
\tilde{\crm}_\chi \lp  \frac{\ell^2 n}{\Delta}, \ell h \rp = 
  \begin{cases}
    \tilde{\crm}_\chi \lp \frac{n}{\Delta},  h \rp , & \text{ if } \ell \nmid \Delta \text{ or }\ell \mid \gcd(\Delta, n), \\
    (1 + \chi([\lf])) \tilde{\crm}_\chi \lp \frac{n}{\Delta},  h \rp, & \text{ if } \ell\mid \Delta \text{ and } \ell \nmid n,
  \end{cases}
\end{equation}
for all $n \in \Zb$ and $h \in A_\Delta$.
\end{cor}
\begin{proof}
We have $\mu_{\af \lf} = \mu_{\af} \mu_{\lf}$ by \eqref{eq:mult} and are done by proceeding as in the proof of Prop.\ \ref{prop:cchiprop}.
\end{proof}
\begin{proof}[Proof of Prop.\ \ref{prop:fac}]
 Prop.\ \ref{prop:rewrite2} implies that 
 \begin{equation*}
     \tilde{C}_\chi(\lambda_0; S_F) = \tilde{r} \log \varep_F + 
\tilde{\Cr}_{\chi}(\lambda_0) 
 + \sum_{\bfrak \in S_F} c_\bfrak \sum_{t \mid \sqrt{\Delta} \lambda_0} c_{\Nm^-(\bfrak)} \lp \Nm(\lambda_0/t), \lambda_0/t \rp
 \end{equation*}
for some $r \in \frac{1}{h_F} \Zb$, where
\begin{equation}
  \label{eq:tildeCr}
  \tilde{\Cr}_{\chi}(\lambda_0)  := 
\sum_{t \mid \sqrt{\Delta} \lambda_0} \tilde{\crm}_\chi(\Nm(\lambda_0/t), \lambda_0/t) .
\end{equation}
Therefore,
$$
  \tilde{C}_\chi(\lambda_0; S_F) = 
\frac{1}{h_F} \log |\alpha| + \sum_{\bfrak \in S_F} c_\bfrak \sum_{t \mid \sqrt{\Delta} \lambda_0} c_{\Nm^-(\bfrak)} \lp \Nm(\lambda_0/t), \lambda_0/t \rp
$$
for some $\alpha \in F$ satisfying $\alpha' = 1/\alpha$. 
Even though $\alpha$ depends on $S_F$, the factorization of the fractional ideal $(\alpha)$ only depends $\tilde{\Cr}_\chi(\lambda_0)$, which is independent of the set $S_F$. To finish proving the proposition, we need to analyze this part. 

Suppose the principal ideal generated by $\mu_0 := \sqrt{\Delta}\lambda_0$ has the factorization $\af \prod \underline{\lf}^\ua (\underline{\lf}')^{\ub}$ with $\af$ consisting of exactly the inert and ramified primes in $\mu_0$.
Here $\underline{\lf} = (\lf_j)_{1 \le j \le J}$ are the split primes in $\mu_0$, and we have adopted the notations in the proof of Prop.\ \ref{prop:count}, along which we obtain
\begin{equation}
  \label{eq:tCrchi2}
  \tilde{\Cr}_\chi(\lambda_0) = \frac{\rho_{K/F}(\af)}{h_F} \sum_{\underline{s} \le \underline{a}} \sum_{\us - \ua \le \ur \le \ub - \us}  \chi \lp \prod \underline{\lf}^\ur \rp \log \left|
\frac{\mu_{\af(\ur)}}{\mu'_{\af(\ur)}}
\right|.
\end{equation}
Using the definition of the ideal $\af(\ur)$ in \eqref{eq:aur}, we can rewrite the equation above as
$$
  \tilde{\Cr}_\chi(\lambda_0) = \sum_{1 \le j \le J} \frac{\rho_{K/F}((\mu_0)/(\lf_j^{a_j} (\lf'_j)^{b_j}))}{h_F} 
\log \left|\frac{\mu_{\lf_j}}{\mu'_{\lf_j}}\right|
\sum_{\begin{subarray}{c} s_j \le a_j\\ s_j - a_j \le r_j \le b_j - s_j \end{subarray}} \chi(\lf_j)^{r_j} (a_j - b_j + 2r_j).
$$
By the lemma below, the last sum vanishes unless $\chi(\lf_j) = -1$ and $2\nmid(a-b)$. 
Therefore, we can further simply $\tilde{\Cr}_\chi(\lambda_0)$ to
\begin{align*}
    \tilde{\Cr}_\chi(\lambda_0) &= 
\sum_{\begin{subarray}{c} 1 \le j \le J,~ \chi(\lf_j) = -1\\ 2 \mid \gcd(a_j + 1, b_j)\end{subarray}} \frac{\rho_{K/F}((\mu_0)/(\lf_j^{a_j}(\lf'_j)^{b_j}))(a_j + 1)}{h_F} 
\log \left|\frac{\mu_{\lf_j}}{\mu'_{\lf_j}}\right|  \\
&\quad +\sum_{\begin{subarray}{c} 1 \le j \le J,~ \chi(\lf'_j) = -1\\ 2 \mid \gcd(a_j, b_j +1)\end{subarray}} \frac{\rho_{K/F}((\mu_0)/(\lf_j^{a_j} (\lf_j')^{b_j}))(b_j + 1)}{h_F} 
\log \left|\frac{\mu'_{\lf_j}}{\mu_{\lf_j}}\right|\\
&= 
\sum_{\begin{subarray}{c} 1 \le j \le J\\ \chi(\lf_j) = -1\end{subarray}} 
\frac{\rho_{K/F}((\mu_0)\lf_j)(a_j + 1)}{h_F} 
\log \left|\frac{\mu_{\lf_j}}{\mu'_{\lf_j}}\right|  +
\frac{\rho_{K/F}((\mu_0)\lf_j')(b_j + 1)}{h_F} 
\log \left|\frac{\mu'_{\lf_j}}{\mu_{\lf_j}}\right|.
\end{align*}
By setting
$$
\gamma(\lambda; S_F) := \varep_F^{h_F \tilde{r}} \prod_{1 \le j \le J,~ \chi(\lf_j) = -1}  \mu_{\lf_j}^{\rho_{K/F}((\mu_0)\lf_j)(a_j+1)}
(\mu'_{\lf_j})^{\rho_{K/F}((\mu_0)\lf'_j)(b_j+1)},
$$
we obtain the claim in \eqref{eq:facgamma}.
\end{proof}

\begin{lemma}
  \label{lemma:identity}
For $\epsilon = \pm 1$ and integers $ 0 \le a \le b$, we have
\begin{equation}
  \label{eq:identity}
\sum_{\begin{subarray}{c} s  \le a \\ s  - a  \le r  \le b  - s  \end{subarray}} \epsilon^{r } (a  - b  + 2r) =
\begin{cases}
a + 1, & \text{if } \epsilon = - 1 \text{ and }2 \mid \gcd(a + 1, b), \\
-(b + 1), & \text{if } \epsilon = -1 \text{ and }2 \mid \gcd(a , b+1), \\
  0, & \text{otherwise.}
\end{cases}
\end{equation}
\end{lemma}
\section{Higher Green's Function}
\label{sec:hGf}
In this section, we will state and prove the main result (Theorem \ref{thm:main}). The basic idea is to express the higher Green's function as a theta integral. At the cycle $Z_\chi$, the theta kernel becomes the Eisenstein series $y\Ec_\chi(z)$ by the Siegel-Weil formula (see Prop.\ \ref{prop:SW}). The key idea then is to use theta lift to construct a suitable real-analytic modular form, whose image under the lowering operator is related to the image of $y\Ec_\chi(z)$ under the raising operator.
This construction, along with necessary calculations, are postponed until section \ref{sec:thetalifts}.

\subsection{Differential Operators}
For any $k \in \half \Zb$ and $\tau = u + iv \in \Hb$, define the following differential operators on real-analytic $\Cb$-valued function $f$ on $\Hb$
\begin{equation}
\label{eq:diffops}
\begin{split}
  R_{\tau, k} &= 2i \partial_\tau + \frac{k}{v}, \;
L_{\tau, k} = -2iv^2 \partial_{\overline{\tau}}, \\
\Delta_{\tau, k} &= - R_{\tau, k - 2} L_{\tau, k} = - L_{\tau, k + 2} R_{\tau, k} - k = -v^2 \lp \partial_u^2 + \partial_v^2 \rp + ikv (\partial_u + i \partial_v).
\end{split}
\end{equation}
We omit the subscript $\tau$ when it is clear from the context. 
For any $n \in \Nb$, we also know that 
\begin{equation}
\label{eq:raise}
R^n_{ k} f := R_{k + 2n-2} \circ R_{  k + 2n-4} \circ \dots \circ R_{  k} f
=
\sum_{r = 0}^{n} (2 i)^r \binom{n}{r} \frac{(k+r)_{n-r}}{v^{n - r}} (\partial_\tau^r f)(\tau),
\end{equation}
from equation (56) in \cite{BvdGZ123}.
Suppose that $f \in \Ac_{k, \rho}(\Gamma)$ for any $\Gamma \subset \Mp_2(\Zb)$ of finite index and representation $\rho: \Gamma \to \GL(V)$. Then applying the differential operators $R_{  k}$ and $L_{  k}$ componentwisely to $f$ preserves its modularity by raising and lowering the weight by 2 respectively, i.e.\ $R_{  k}(f) \in \Ac_{k+2 , \rho}(\Gamma), L_{  k}(f) \in \Ac_{k-2, \rho}(\Gamma)$.
Furthermore, it is a simple consequence of the commutation of differential operators that
\begin{equation}
  \label{eq:eigen}
  \Delta_{  k + 2n} R^{n}_{  k} f = (k + n - 1)n  R^{n}_{  k} f
\end{equation}
if $\Delta_{  k} f = 0$. 
For $N,n \in \Nb$, recall the unary theta function $\theta_N(\tau, n) \in \Ac_{n + 1/2}(\rho_N)$ from \eqref{eq:thetan}. It is easy to check that 
\begin{equation}
  \label{eq:Runary}
  R_{n + 1/2} \theta_N(\tau; n) = -2 \pi \theta_N(\tau; n+2).
\end{equation}
From Stokes' theorem, we can deduce the following consequence.
\begin{lemma}
\label{lemma:RLswitch}
Let $f, g$ be smooth functions on $\Hb$ and $S \subset \Hb$ a compact domain.
Then
\begin{equation}
\label{eq:RLswitch}
\begin{split}
  \int_{S} f (R_{-2-k} g) d\mu(\tau) + 
\int_{S} (R_k f)  g d\mu(\tau) &=
- \oint_{\partial S} \frac{f g}{v^2} d \overline{\tau}, \\
  \int_{S} f (L_{k_1} g) d\mu(\tau) + 
\int_{S} (L_{k_2} f)  g d\mu(\tau) &=
- \oint_{\partial S} {f g}{} d {\tau},
\end{split}
\end{equation}
with counterclockwise orientation on $\partial S$ for any $k, k_1, k_2 \in \half \Zb$.
\end{lemma}
\begin{proof}
  Note that $2i \frac{\partial (fg /v^2)}{\partial \tau} = \frac{R_{-2}(fg)}{v^2} = v^{-2}((R_{k}f) g + f (R_{-k-2} g)) $. Integrate this against ${du\wedge dv}$ and using $2i du \wedge dv =  d \overline{\tau}\wedge d\tau $ along with Stokes' theorem proves the first equality. The second one follows similarly by considering $\int_S d(f g d\tau)$.
\end{proof}

In order to maintain both holomorphicity and modularity of derivatives of modular forms, one can take linear combinations of products of repeatedly raised modular forms. This is the idea behind the following operation (see e.g.\ section 5.2 of part 1 of \cite{BvdGZ123}).
For $j = 1, 2$, let $\Gamma_j \subset \Mp_2(\Zb)$ be finite index subgroups and $\rho_j$ representations of $\Gamma_j$ on $\Cb[A_j]$.
For $k, \ell \in \half \Zb $ and $r \in \Nb$, we can define the Rankin-Cohen bracket on two real-analytic modular forms $f \in \Ac_{k, \rho_1}(\Gamma_1), g \in \Ac_{\ell, \rho_2}(\Gamma_2)$ by
\begin{equation}
  \label{eq:RC}
  \begin{split}
      [f, g]_r &:= \sum^r_{s = 0} (-1)^s \binom{k + r - 1}{s} \binom{\ell + r - 1}{r - s} \sum_{h_1 \in A_1, h_2 \in A_2} f_{h_1}^{(r-s)} g_{h_2}^{(s)}, \\
&= (-4\pi)^{-r} \sum^r_{s = 0} (-1)^s \binom{k + r - 1}{s} \binom{\ell + r - 1}{r - s} \sum_{h_1 \in A_1, h_2 \in A_2} R^{r-s}_{k} f_{h_1} R^s_{\ell} g_{h_2},
  \end{split}
\end{equation}
where $G^{(j)} := \frac{1}{(2\pi i )^j} \partial_\tau^j G$ for any real-analytic function $G$ on $\Hb$ and $\binom{m}{n} := \frac{m(m-1)(m-2)\dots(m-n + 1)}{n!}$ is the binomial coefficient.
Then one can check that 
$$
[f, g]_r \in \Ac_{k + \ell + 2r, \rho_1 \otimes \rho_2}(\Gamma)
$$
with $\Gamma = \Gamma_1 \cap \Gamma_2$.
Furthermore, if $f$ and $g$ are holomorphic with rational Fourier coefficients, so is $[f, g]_r$.

\subsection{Higher Green's Function}
For $\Re(s) > 1$, let 
\begin{equation}
  \label{eq:Qs}
Q_{s-1}(t) := \int^\infty_0 (t + \sqrt{t^2 - 1} \cosh v)^{-s} dv
\end{equation}
 be the Legendre function of the second kind, which satisfies the ordinary differential equation 
$$
(1 - t^2) F''(t) - 2t F'(t) + s(s-1) F(t) = 0.
$$
Define a function $g_s$ on $\Hb^2$ by
\begin{equation}
  \label{eq:gs}
  g_s(\underline{z}) := -2Q_{s-1} (\cosh d(\underline{z})) = -2 Q_{s - 1} \lp 1 + \frac{|z_1 - z_2|^2}{2y_1y_2} \rp
\end{equation}
for $\underline{z} = (z_1, z_2)  \in \Hb^2$ with $d(\underline{z})$ the hyperbolic distance between $z_1$ and $z_2$.
By averaging over the $\Gamma$-translates of the second variable, we obtain a function 
\begin{equation}
  \label{eq:Gs}
  G_s(\underline{z}) := \sum_{\gamma \in \mathrm{PSL}_2(\Zb)} g_s(z_1, \gamma z_2)
\end{equation}
on $(\mathrm{PSL}_2(\Zb) \backslash \Hb)^2$.
Since linear fractional transformation is an isometry with respect to the hyperbolic distance, the function $G_s(\underline{z})$ is symmetric in $z_1$ and $z_2$.
Furthermore, it is an eigenfunction of the hyperbolic Laplacian $\Delta_{z_j, 0}$ of eigenvalue $s(1-s)$ for $j = 1, 2$.

Fix a positive integer $k \ge 2$ and let $f(\tau) = \sum_{m \ge -\infty} c_f(m) q^m \in M^!_{2-2k}$ be a weakly holomorphic modular form.
By the residue theorem, we have 
$$
\sum_{m \ge 1} c_f(-m) a_g(m) = 0
$$
for all $g(\tau) = \sum_{m \ge 1} a_g(m) q^m \in S_{2k}$. 
Define the higher Green's function with respect to $f$ by
\begin{equation}
\label{eq:Gkf}
G_{k, f}(\underline{z}) :=  \sum_{m \ge 1} c_f(-m) m^{k-1} G_k(\underline{z}) \mid T_m
\end{equation}
with $T_m$ the $m$\tth Hecke operator acting on $z_2$ (see \textsection IV.4 of \cite{GZ86}).
The singularity of $G_{k, f}$ is 
\begin{equation}
  \label{eq:Tf}
  T_f := \bigcup_{\begin{subarray}{c} m \ge 1,~ c_f(-m) \neq 0\\ z \in \Hb \end{subarray}} (z, T_m z) \subset \Hb^2.
\end{equation}
It was shown in \cite{Viazovska11} that 
\begin{equation}
\label{eq:GkfPhi}
2 G_{k, f}(\underline{z})
= 
(-4\pi)^{1 - k} \Phi_{M_2(\Zb)}(w(\underline{z}), 1, R^{k-1}_{2-2k}f),
\end{equation}
where $w(\underline{z})$ is the identification in Example \ref{ex:M2Z}.
This is essentially a direct consequence of Theorem 6.2 in \cite{Borcherds98}, which implies that the right hand side has the same singularity on $\Hb^2$ as the left hand side.
There is an extra factor of 2 here since $\SL_2(\Zb)$ was used in \cite{Viazovska11} in defining the higher Green's function. 
By Lemma \ref{lemma:RLswitch}, we can obtain the following simple consequence.
\begin{prop}
  \label{prop:Gkf}
Let $L = M_2(\Zb)$ and $\Theta_L(\tau, w(\underline{z}))$ the theta function in \eqref{eq:Theta}. Then we have
\begin{equation}
  \label{eq:Gk_Theta}
  G_{k, f}(\underline{z}) {=} (4\pi)^{1-k}  \int^\mathrm{reg}_{\Gamma \backslash \Hb}  f(\tau) R^{k-1}_{\tau, 0}(  \overline{\Theta_L(\tau, w(\underline{z}))}) d\mu(\tau).
\end{equation}
for all $\underline{z} \in \Hb^2$.
\end{prop}

\begin{proof}
  Suppose $z_1, z_2 \in \Hb$. Then $|az_1z_2 + bz_1 + cz_2 + d|^2 + |az_1\zbar_2 + bz_1 + c\zbar_2 + d|^2 \ge 1$ for all $a, b,c , d \in \Zb$ not all zero. Therefore, one can check inductively that 
$$\lim_{T \to \infty} \int_{\partial \Fc_T} (R^{k-1-j}_{2-2k}f)(z) R^{j}_{z, 0} ( v \overline{\Theta_L(z, w(\underline{z}))}) \frac{d z}{y^2}  = 0$$
 for all $0 \le j \le k-1$. Then applying Lemma \ref{lemma:RLswitch} $(k-1)$ times finishes the proof.
\end{proof}

\subsection{Main Theorems and Proofs}
Let $\Delta > 0$ be a fundamental discriminant, $\varep_\Delta \in \Oc_F^\times$ the generator of the discriminant kernel $\Gamma_\Delta$ as in \eqref{eq:varepDelta}, $\chi = \chi_{\Delta_1, \Delta_2}$ an odd genus character, and $Z_\chi$ the associated 0-cycle of CM points on $(\Gamma \backslash \Hb)^2$ as in \eqref{eq:Zchi}. 
Suppose $k \in 2\Nb$ and $f(z) = \sum_{m \ge -m_0} c_f(m) \ebf(mz) \in M^!_{2-2k}$, we will now state and prove the two theorems concerning $G_{k, f}(Z_\chi)$.
\begin{thm}
  \label{thm:main1}
Let $S_F \subset I_\df$ be as in \eqref{eq:Theta12} and $c(f, S_F) \in \Cb$ the constant defined in \eqref{eq:constant}.
Then we have
\begin{equation}
  \label{eq:fundeq}
  \begin{split}
        {G_{k, f}(Z_\chi)}{} &- c(f, S_F) \log \varepsilon_\Delta = 
- \frac{\sqrt{\Delta}^{1-k}}{2} \sum_{1 \le m \le m_0} 
 \sum_{\mu_0 \in \mathrm{S}_m} 
c_f(-m) C_k(\mu_0, m) 
 \tilde{C}_\chi\lp \frac{\mu_0}{\sqrt{\Delta}} \rp,
  \end{split}
\end{equation}
where $C_{k}(\mu_0, m) \in \Zb$ is the integer defined in \eqref{eq:Ck} and $\tilde{C}_\chi$ is defined in \eqref{eq:Ctilde}.
\end{thm}

\begin{proof}
  By the Siegel-Weil formula in Prop.\ \ref{prop:SW} and the integral representation of the higher Green's function in Prop.\ \ref{prop:Gkf}, we can rewrite 
$$
{G_{k, f}(Z_\chi)}{} = (4\pi)^{1-k} \lim_{T \to \infty} \int_{\Fc_T} f(z) R^{k-1}_{z, 0} (y \Ec_\chi(z)) d\mu(z).
$$
In Prop.\ \ref{prop:preimage}, we use theta lift to construct essentially a preimage of $R^{k-1}_{z, 0}(\overline{y \Ec_\chi(z)})$ under the lowering operator. Substituting this into the equation above yields
\begin{align*}
  {G_{k, f}(Z_\chi)}{} &= \frac{1}{4\sqrt{2}} \lim_{T \to \infty} \int_{\Fc_T} f(z) 
\lp L_z ((y^{k} \Phi_{}(z; \tp_{k-1}, {\hat{\vartheta}^c_\chi}))^c)
+ (y^{k-1} \Phi_{}(z; \hp_{k-1}, {\theta_\chi}))^c \log \varep_\Delta \rp
d\mu(z)\\
&= 
c(f, S_F) \log \varep_\Delta + 
\frac{1}{4\sqrt{2}} \lim_{T \to \infty} \int_{\Fc_T} f(z) 
 L_z ((y^{k} \Phi_{}(z; \tp_{k-1}, {\hat{\vartheta}^c_\chi}))^c)
d\mu(z) \\
&= 
c(f, S_F) \log \varep_\Delta + 
\frac{1}{4\sqrt{2}} \lim_{T \to \infty} \int_{0}^1 f(x + iT) 
 (T^{k} \Phi_{}(x + iT; \tp_{k-1}, {\hat{\vartheta}^c_\chi}))^c
dx,
\end{align*}
where the second step follows from the definition of $c(f, S_F)$ in \eqref{eq:constant} and the last step is a consequence of Lemma \ref{lemma:RLswitch}.
Now Prop.\ \ref{prop:FE1} tells us the $m$\tth Fourier coefficient $\hat{a}_\chi(m, y)$ of $(y^k\Phi(z; \tp_{k-1}, \hat{\vartheta}_\chi^c))^c$. It decays as $y^{-1/2}{(1 + |m|)^{k+3}e^{-2 \pi |m| y}}$ when $m \le 0$ and has a main term $\tilde{a}_\chi(m) e^{-2\pi m y}$ when $m > 0$.
Putting these together, we see that 
\begin{align*}
  {G_{k, f}(Z_\chi)}{} &- c(f, S_F) \log \varepsilon_\Delta = \frac{1}{4\sqrt{2}} \sum_{1 \le m \le m_0} c_f(-m) \tilde{a}_\chi(m).
\end{align*}
Substituting the definition for $\tilde{a}_\chi(m)$ in \eqref{eq:tildea} and we are done.
\end{proof}

We can now combine formula \eqref{eq:fundeq} with Prop.\ \ref{prop:fac} to prove Theorem \ref{thma:main1}. 
\begin{thm}\label{thm:main}
If $f(\tau)  = \sum_{m \ge -m_0} c_f(m) q^m  \in M^!_{2-2k}$ has rational Fourier coefficients, then there exists $\kappa = \kappa_{f, F} \in \Nb$ and $\gamma_f \in F^\times$ such that 
\begin{equation}
\label{eq:rationality}
\kappa G_{k, f}(Z_\chi) = - \Delta^{\frac{1-k}{2}} \cdot \log \left|\frac{\gamma_f}{\gamma_f'}\right|.
\end{equation}
Furthermore, for every prime ideal $\lf$ of $\Oc_F$ satisfying $\chi(\lf) = -1, \lf \neq \lf'$,
\begin{equation}
  \label{eq:factorization}
\frac{  \ord_{\lf}  (\gamma_f) }{\kappa} =
  \sum_{1 \le m \le m_0} 
c_f(-m) 
\sum_{\mu_0 \in \mathrm{S}_m}
\frac{C_k(\mu_0, m)}{2}
\rho_{K/F}((\mu_0) \lf)( 1 + \ord_\lf((\mu_0))).
\end{equation}
and $\ord_\lf(\gamma_f) = 0$ for every other prime ideal $\lf$ of $\Oc_F$, where $C_k(\mu_0, m)$ is the quantity defined in \eqref{eq:Ck}.
In particular, the fractional ideal generated by $\frac{\gamma_f'}{\gamma_f}$ has the factorization
$$
\prod_{m \ge 1} \prod_{\mu_0 \in \mathrm{S}_m} I^-_\chi(\mu_0)^{c_f(-m) C_k(\mu_0, m)},
$$
where $I^-_\chi(\af)$ was defined in \eqref{eq:I-}.
\end{thm}

\begin{rmk}
The result above still holds if $\gamma_f$ above is replaced by $n \gamma_f $ for any $n \in \Nb$. 
\end{rmk}
\begin{proof}
  First, we can substitute \eqref{eq:facCtilde} into \eqref{eq:fundeq} to obtain
  \begin{align*}
    {G_{k, f}(Z_\chi)}{} &= c(f, S_F) \log \varepsilon_\Delta -
\sqrt{\Delta}^{1-k}                           
\sum_{1 \le m \le m_0}
 \sum_{\mu_0 \in \mathrm{S}_m}
  \frac{c_f(-m) C_k(\mu_0, m) }{4h_F}
 \log \left| \frac{\gamma(\mu_0/\sqrt{\Delta}; S_F)}{\gamma(\mu_0/\sqrt{\Delta}; S_F)'} \right| \\
    &\quad  - \frac{1}{2} \sum_{\bfrak \in S_F} c_\bfrak
\sum_{1 \le m \le m_0}  c_f(-m) c_{\Phi, \bfrak, k}(m),
  \end{align*}
where $c_{\Phi, \bfrak, k}(m)$ is defined by
$$
c_{\Phi, \bfrak, k}(m) := \sqrt{\Delta}^{1-k} 
\sum_{\mu_0 \in \mathrm{S}_m} C_k(\mu_0, m)
 \sum_{t \mid \mu_0} c_{\Nm^-(\bfrak)} \lp \Nm(\mu_0/(\sqrt{\Delta}t)), \lambda_0/(\sqrt{\Delta}t) \rp
$$
and is the $m$\tth Fourier coefficient of the holomorphic cusp form $\tfrac{-1}{2\sqrt{2}}(y^k \Phi(z; \tp_{k-1}, \vartheta_{\Nm^-(\bfrak)}^c))^c$ in $S_{2k}$ by Prop.\ \ref{prop:classical}.
The sum $\sum_{1 \le m \le m_0} c_f(-m) c_{\Phi, \bfrak, k}(m)$ is then the constant term of the $\tfrac{-f}{2\sqrt{2}}(y^k \Phi(z; \tp_{k-1}, \vartheta_{\Nm^-(\bfrak)}^c))^c \in M^!_2$, which is zero.
Now Prop.\ \ref{prop:rational} tells us that $\sqrt{\Delta} c(f, S_F) \in \Qb$.
Since $k \in 2\Nb$, the Legendre polynomial $P_{k-1}$ is odd and $\sqrt{\Delta}^{k-1} P_{k-1}((\lambda_0 - \lambda'_0)/m) \in \Qb$ for $\lambda_0 \in F$. 
We can now choose $\kappa \in \Nb$ such that 
$$
\kappa \Delta^{(k-1)/2} c(f, S_F),~ 
\kappa \frac{c_f(-m) C_k(\mu_0, m)}{4h_F} 
 \in \Zb
$$
for all $\mu_0 \in \mathrm{S}_m$ and $1 \le m \le m_0$.
We then obtain equation \eqref{eq:factorization} by the factorization of $\gamma(\lambda_0; S_F)$ in Prop.\ \ref{prop:fac}.
To see the factorization of the fractional ideal generated by $\frac{\gamma_f'}{\gamma_f}$, notice that
\begin{equation}
  \label{eq:rho'}
  \prod_{\lf \mid \mu_0 \text{ prime},~ \chi(\lf) = -1} \lf^{\rho_{K/F}((\mu_0)\lf)( 1 + \ord_{\lf}((\mu_0)))} = \prod_{\af \mid (\mu_0)} \af^{-2\chi(\af)}.
\end{equation}
This finishes the proof.
\end{proof}

\section{Three Theta Lifts}
\label{sec:thetalifts}
This section contains the construction of the preimage of the kernel $R^{k-1}_0 (y \Ec_\chi(z))$ under the lowering operator $L_{2k}$. It consists of several theta lifts. We will first present the basic setup and some technical calculations, before moving on to the construction in \ref{subsec:lift2}.

\subsection{Basic Setup}
\label{subsec:setup22}
Let $F = \Qb(\sqrt{\Delta})\subset \Rb$ be a real quadratic field with fundamental discriminant $\Delta > 0$.
Consider the lattice $M_{} \subset M_2(F)$ defined by
\begin{equation}
  \label{eq:M}
  M_{}:= \left\{
\pmat{a}{\lambda}{\lambda'}{ b} \in M_2(F): a, b \in \Zb, \lambda \in \Oc_F
\right\}
\end{equation}
with determinant as the quadratic form.
The dual lattice $M^\vee $ is given by
$$
M^\vee := 
\left\{
\pmat{a}{\lambda}{\lambda'}{b} \in M_2(F): a, b \in \Zb, \lambda \in \df^{-1}
\right\}.
$$
Then the lattice $L^- = (\Oc_F, -\Nm)$ embeds isometrically into $M$ via $\lambda \mapsto  \pmat{0}{\lambda}{\lambda'}{0} $, through which $A_M \cong A_{L^-}$ as $\Gamma$-modules with respect to the Weil representation $\rho_{M} \cong \overline{\rho_\Delta}$ in section \ref{subsec:Hecke}.
We also denote $Q(\lambda) := Q \lp \smat{0}{\lambda}{\lambda'}{0} \rp = -\Nm(\lambda)$ for $\lambda \in F$.

Via the two embeddings of $F$ into $\Rb$, the real quadratic space $V:= M \otimes_\Zb \Rb$ is isomorphic to $(M_2(\Rb), \det)$.
As in section \ref{sec:Fock}, we identify $M_2(\Rb)$ with $\Rb^{2, 2}$ and the symmetric space $\Dc_V$ with $\Hb^2$.
For our purpose, it suffices here to consider the diagonal $\Hb \subset \Hb\times \Hb$. 
For $a, b, c \in \Nb$, consider the polynomial $p_{a, b, c}$ on $\Rb^{2, 2}$ defined in \eqref{eq:pabc}. The function $y^{-c} \Theta_M(\tau, z; p_{a, b, c})$ is in $\Ac_{a + b - c, \rho_M}$ and $\Ac_{2c}$ as a function in $\tau$ and $z$ respectively.
Suppose 
$$
f(\tau) = \sum_{h \in A_M} \ef_h \sum_{m \in \Qb^\times} c(m, h, v) \ebf(mu) \in \Ac_{a + b - c, \rho_M}
$$ is regular on $\Hb$ and decays sufficiently fast near the cusp such that the following limit exists
\begin{equation}
  \label{eq:Phiz}
  \Phi(z; p_{a, b, c}, f) := \lim_{T \to \infty} \int_{\Fc_T } \langle f(\tau),   \Theta_M(\tau, z; p_{a, b, c} ) \rangle  v^{a + b - c}  d\mu(\tau).
\end{equation}
We would like to express its Fourier expansion in terms of $c(m, h, v)$ by specializing Theorem 7.1 in \cite{Borcherds98}.
The result is as follows.

\begin{prop}
  \label{prop:FE}
In the notations above, the function $y^{-1}\Phi(z; p_{a, b, c}, f)$ has the expansion
\begin{equation}
  \label{eq:PhiFE}
  \begin{split}
y^{-1} \Phi(z; p_{a, b, c}, f) &=    c(f; p_{a, b, c}) + 
  \sum_{\begin{subarray}{c} 0 \le  h \le  h^- \le c \\ 0 \le r \le j \le (b + c - h^-)/2 \\ d \ge 1, \lambda \in \df^{-1} \backslash \{0\}\end{subarray}} 
\frac{ \binom{c}{h^-} \binom{a}{ h} \binom{h^-}{h} h! (2j)! \binom{j}{r} \binom{b}{2r} \binom{c-h^-}{2j-2r} (-1)^{a + b + c  + j}}{ j! \binom{2j}{2r} 2^{(a+b+c)/2 + 2j + h - 1} \pi^{h + j}}\\ 
&\hspace{1in}  \times \ebf(-d(\lambda + \lambda')x)  (dy)^{a + b - h - j - 1/2 }(\lambda - \lambda')^{b-2r}(\lambda +\lambda')^{c-h^- -2j + 2r}\\
&\quad  \times \int^\infty_0 c(-\lambda\lambda', \lambda, dyv) \exp \lp - \pi d y \lp \frac{1}{v} + (\lambda^2 + (\lambda')^2 )v\rp \rp v^{b + h - h^- - j - 1/2} \frac{dv}{v},
  \end{split}
\end{equation}
where $c(f; p_{a, b, c})$ is a constant.
\end{prop}

Later, we will evaluate the integral in $v$, which becomes the $K$-Bessel function for holomorphic input $f$. The following standard estimate is then useful.

\begin{lemma}
  \label{lemma:Besselestimate}
  For $\alpha \in \Cb$ and $r \in \Rb$, let $K_\alpha(r)$ be the $K$-Bessel function defined by
  \begin{equation}
    \label{eq:KBessel}
    K_\alpha(r) := \frac{1}{2} \lp \frac{r}{2} \rp^{-\alpha} \int^\infty_0 \exp\lp -t - \frac{r^2}{4t} \rp t^{\alpha} \frac{dt}{t}.
  \end{equation}
It has the following asymptotic expansion
\begin{equation}
  \label{eq:KBesselExp}
  K_\alpha(r) = \sqrt{\frac{\pi}{2r}} e^{-r} + O_\alpha(r^{-3/2}e^{-r})
\end{equation}
for $r$ large.
\end{lemma}

\subsection{Lift 1}
\label{subsec:lift1}
Now, we are ready to apply the above result to $(a, b, c) = (0, 1, 0)$ and $f = \vartheta_\chi$ for $\chi = \chi_{\Delta_1, \Delta_2}$ an odd genus character of $F$.
The theta lift turns out to be the function $y\Ec_\chi(z)$, which is a finite sum of the theta kernels.
\begin{prop}
  \label{prop:lift1}
Let $\chi$ be an odd genus character and $\vartheta_{\chi} \in S_{1, \overline{\rho_\Delta}}$ be the cusp form defined in \eqref{eq:vartheta12}. Then
\begin{equation}
  \label{eq:lift1}
 \Phi_{}(z; p_{0, 1, 0}, \vartheta_{\chi}) =
\overline{ \Phi_{}(z; p_{0, 1, 0}, \vartheta_{\chi})} =  - {4\sqrt{2}}{y \Ec_\chi(z)}
\end{equation}
for all $z \in \Hb$.
\end{prop}

\begin{proof}
We will compare the Fourier expansions of both sides. 
The parameters $h, h^-, j, r$ in the summation of \eqref{eq:PhiFE} are all zero, and Prop.\ \ref{prop:FE} implies that
\begin{equation*}
 \Phi_{}(z; p_{0, 1, 0}, \vartheta_\chi) = - \sqrt{2} {y} \left\{
\begin{split}
2 c_0 &+ \sum_{\begin{subarray}{c} \lambda \in \df^{-1} \\ \lambda > 0 > \lambda' \end{subarray}}c_\chi(-\lambda\lambda',  \lambda) \sum_{n \ge 1} \ebf(n (-\lambda \zbar - \lambda' z))\\
&-\sum_{\begin{subarray}{c} \lambda \in \df^{-1} \\ \lambda' > 0 > \lambda \end{subarray}}c_\chi(-\lambda \lambda',  \lambda) \sum_{n \ge 1} \ebf(n (-\lambda z - \lambda' \zbar))
\end{split}
\right\},
\end{equation*}
with $c_0 \in \Cb$ and $c_\chi(m, h)$ the Fourier coefficients of $\vartheta_\chi$. 
By \eqref{eq:conj}, we know that $c_\chi(m, \lambda) = -c_\chi(m, -\lambda') = - c_\chi(m, \lambda')$. Therefore, the equation above becomes
$$
\Phi_{}(z; p_{0,1,0}, \vartheta_\chi) = -{2\sqrt{2}}y \lp
c_0 + \sum_{\begin{subarray}{c} \lambda_0 \in \df^{-1} \\ \lambda_0 > 0 > \lambda_0'  \end{subarray}} 
C_\chi (\lambda_0 \sqrt{\Delta})
\ebf(\lambda_0 z + \lambda'_0 \zbar)
 \rp,
$$
where we have substituted $\lambda_0 = -n \lambda'$ and $C_\chi$ is defined in \eqref{eq:C}. 
Prop.\ \ref{prop:count} now directly implies that $\Phi(z; p_{0, 1, 0}, \vartheta_{\chi}) + 4\sqrt{2} y \Ec_\chi(z) = y c_1$ for some constant $c_1 \in \Cb$. Since this difference is in $\Ac_0$, we must have $c_1 = 0$, therefore proving \eqref{eq:lift1}.
\end{proof}

\subsection{Lift 2}
\label{subsec:lift2}
In the first theta lift, we have already expressed $y\Ec_\chi(z)$ as a theta lift. 
Since the value of the higher Green's function $G_{k, f}(z_1, z_2)$ at $Z_\chi$ is related to $R^{k-1}_{z, 0} y \Ec_\chi(z)$ for $k \in \Nb$, we want to realize this function also as a suitable theta lift. 
Let $M$ be the lattice in the previous section and $p_{a, b, c}$ the polynomial on $\Rb^{2, 2}$ defined in \eqref{eq:pabc}.
Our first result is as follows.

\begin{prop}
  \label{prop:lift2}
Let $b \in \Nb$ be an odd, positive integers with $b < k$ and $\chi$ an odd genus character of $F$. Then we have 
\begin{equation}
  \label{eq:lift2}
 (4\pi)^{k-1}
(y^{k-1} \Phi_{}(z; p_{k-b, b, k-1}, \vartheta_{\chi}))^c = -4\sqrt{2} R^{k-1}_{z, 0} y \Ec_\chi(z)
\end{equation}
for all $z \in \Hb$.
\end{prop}

\begin{rmk}
  \label{rmk:vanish}
The same proof shows that $\Phi(z; p_{1, 0, 0}, \vartheta_\chi)$ vanishes identically.
\end{rmk}

\begin{proof}
First, we denote $a:= k-b$ and show that $\Phi(z; p_{a, b, k-1}, f)$ only depends on $b$ modulo 2 when $f  \in S_{1, \rho_M}$ is holomorphic. 
In this case, it is easy to first check that $R_{\tau, -1} {f^c(\tau)} = 0$.
By Lemma \ref{lemma:RLswitch} and the first equation in Corollary \ref{cor:diffswitch}, we see that
\begin{align*}
  \overline{\Phi(z; p_{a, b+2,k-1}, f) - \Phi(z; p_{a+2, b, k-1}, f)} &= (2\pi)^{-1}\lim_{T \to \infty} \int_{\Fc_T} \langle {f^c(\tau)}, \overline{ R_{\tau, -1} \Theta_{M}(\tau, z; p_{a, b, k-1}) } \rangle d \mu(\tau)\\
&= - (2\pi)^{-1} \lim_{T \to \infty} \int_{\partial \Fc_T}  \langle f^c(\tau), \overline{\Theta_{M}(\tau, z; p_{a, b, k-1})} \rangle \frac{d\overline{\tau}}{v^2} = 0.
\end{align*}
Now to prove \eqref{eq:lift2}, we can use Prop.\ \ref{prop:lift1} and apply the raising operator $R_{z, 0}^{k-1}$ to the theta integral $\overline{\Phi(z; p_{0, 1, 0}, \vartheta_\chi)}$ to obtain
  \begin{align*}
    -4\sqrt{2} R_{z, 0}^{k-1} y \Ec_\chi(z) & = 
R_{z, 0}^{k-1} \overline{\Phi(z; p_{0, 1, 0}, \vartheta_\chi)}
= \int_{\Gamma \backslash \Hb} \langle \overline{\vartheta_{\chi}(\tau)}, \overline{R_{z, 0}^{k-1} \Theta_{M}(\tau, z; p_{0, 1, 0})} \rangle d\mu(\tau) \\
&= (4\pi/ y)^{k-1} \overline{\Phi(z; p_{k-1, 1, k-1}, \vartheta_\chi)}
= (4\pi y)^{k-1} (\Phi(z; p_{k-1, 1, k-1}, \vartheta_\chi))^c,
  \end{align*}
where we used $R^{k-1}_{z, 0} \Theta_M(\tau, z; p_{0, 1, 0}) = (4\pi/y)^{k-1} \Theta_M(\tau, z; p_{k-1, 1, k-1})$ from Corollary \ref{cor:diffswitch}.
\end{proof}

As an immediate consequence of this proposition and Eq.\ \eqref{eq:basic}, we can use the polynomial $\hp_{k-1}$ to produce a theta kernel that could reproduce $R^{k-1}_{z, 0} y \Ec_\chi(z)$.
\begin{thm}
  \label{thm:hpEc}
For a positive, even integer $k \in \Nb$, let $\hp_{k-1}$ be the polynomial on $\Rb^{2, 2}$ defined in \eqref{eq:keypols}. Then 
\begin{equation}
  \label{eq:hpEc}
-4\sqrt{2}  R^{k-1}_{z, 0} (y \Ec_\chi(z)) = (4\pi)^{k-1} (y^{k-1} \Phi(z; \hp_{k-1}, \vartheta_\chi))^c.
\end{equation}
\end{thm}

\begin{rmk}
  Similar statement holds for odd $k \in \Nb$ if $y\Ec_\chi(z)$ is replaced with $\Phi(z; p_{1, 0, 0}, \vartheta_\chi)$. But this vanishes identically by Remark \ref{rmk:vanish}. 
\end{rmk}

\begin{proof}
  By its definition in \eqref{eq:keypols}, we know that $\hp_{k-1} = \sum_{0 \le b \le k-1} c_{k-1, b} p_{k-b, b, k-1}$. Since $k$ is even, $c_{k-1, b}$ is non-zero precisely when $b$ is odd and between $0$ and $k-1$. 
The theorem then follows from Prop.\ \ref{prop:lift2} and the identity $\sum_{0 \le b \le k-1} c_{k-1, b} = P_{k-1}(1) = 1$. 
\end{proof}

By exploiting the relationship between $\Theta_M(\tau, z; \hp_{k-1})$ and $\Theta_M(\tau, z; \tp_{k-1})$ in Corollary \ref{cor:diffswitch}, we can now use theta lift to construct a function that \textit{almost} maps to $R^{k-1}_{z, 0}y \Ec_\chi(z)$ under $L_z$. 
\begin{prop}
  \label{prop:preimage}
Choose a set of ideals $S_F$ as in \eqref{eq:Theta12}, and let $\theta_\chi(\tau; S_F) \in \Ac_{1, {\rho_{M}}}, \hat{\vartheta}_\chi(\tau; S_F) \in \Ac_{1, \rho_{\Delta}}$ be the functions defined in \eqref{eq:theta} and \eqref{eq:hvartheta} respectively.
In the notations above, we have
\begin{equation}
  \label{eq:preimage}
L_z ((y^{k} \Phi_{}(z; \tp_{k-1}, {\hat{\vartheta}^c_\chi}))^c)
 = 
4\sqrt{2} (4\pi)^{1-k} R^{k-1}_{z, 0}( y \Ec_\chi(z) )
- (y^{k-1} \Phi_{}(z; \hp_{k-1}, {\theta_\chi}))^c \log \varep_\Delta
\end{equation}
for all $z \in \Hb$ and even $k \in \Nb$.
\end{prop}

\begin{proof}
We can substitute the third equation in \eqref{eq:diffswitchTheta} into the left hand side of \eqref{eq:preimage}, apply Lemma \ref{lemma:RLswitch} and \eqref{eq:basic} to obtain
\begin{align*}
  L_z 
((y^{k} \Phi_{}(z; \tp_{k-1}, \hat{\vartheta}^c_\chi))^c)   &= 
\int_{\Gamma\backslash \Hb} \langle {\hat{\vartheta}_{\chi}(\tau)}, \overline{ L_z y^{-k}\Theta_M(\tau, z; \tp_{k-1})} \rangle d\mu(\tau)\\
&=
2y^{1-k}\int_{\Gamma\backslash \Hb}  \langle \hat{\vartheta}_{\chi}(\tau), \overline{L_\tau \Theta_M(\tau, z; \hp_{k-1})} \rangle  d\mu(\tau)\\
&= 
- 2 y^{1-k}\int_{\Gamma\backslash \Hb}  \langle L_\tau  \hat{\vartheta}_{\chi}(\tau),  \overline{\Theta_M(\tau, z; \hp_{k-1})}   \rangle d\mu(\tau)\\
&=  -(y^{k-1} \Phi_{}(z; \hp_{k-1}, {\vartheta_{\chi}}))^c - 
(y^{k-1} \Phi_{}(z; \hp_{k-1}, {\theta_\chi}))^c \log \varep_\Delta \in \Ac_{2k-2}.
\end{align*}
Theorem \ref{thm:hpEc} then finishes the proof.
\end{proof}

Now, we can again apply Prop.\ \ref{prop:FE} to describe the Fourier expansion of $\Phi(z; \tp_{k-1}, \hat{\vartheta}_\chi^c )$. 
Since the input is real-analytic and the theta kernel has a polynomial, the precise Fourier coefficient, which is a function of $y$, is complicated. Fortunately, we only need the main term that survives when $y$ is large. 
Therefore, the sum in \eqref{eq:PhiFE} greatly simplifies and yields a nice answer.

\begin{prop}
  \label{prop:FE1}
Suppose that the modular form $(y^{k}{\Phi_{}(z; \tp_{k-1},  \hat{\vartheta}^c_\chi)})^c \in \Ac_{2k}$ has the Fourier expansion $\sum_{m \in \Zb} \hat{a}_\chi(m, y) \ebf(mx)$, then 
\begin{equation}
\label{eq:tildea}
  \begin{split}
      \hat{a}_\chi(m, y) &= \lp \tilde{a}_\chi(m) + O_{F, S_F, k}((1+|m|)^{k+3} y^{-1/2}) \rp e^{-2\pi |m|y}, \\
\tilde{a}_\chi(m) &:= -2\sqrt{2} \sqrt{\Delta}^{1-k}
\sum_{\mu_0 \in \mathrm{S}_m}
C_k(\mu_0, m)
\tilde{C}_\chi(\mu_0/\sqrt{\Delta}),
  \end{split}
\end{equation}
where $\mathrm{S}_m$ is the set defined in \eqref{eq:Sm}, $C_k(\mu_0, m)$ is the constant defined in \eqref{eq:Ck} and $\tilde{C}_\chi$ is defined in \eqref{eq:Ctilde}
\end{prop}

\begin{rmk}
  Notice that $\tilde{a}_\chi(m) = 0$ for $m \le 0$ since the sum is empty.
\end{rmk}
\begin{proof}
First, we apply Prop.\ \ref{prop:FE} to $(a, b, c) = (k-b-1, b, k)$ and 
$$f(\tau) = \hat{\vartheta}^c_\chi(\tau) = v \sum_{h \in \df^{-1}/\Oc_F} \ef_h \sum_{m \in \Qb^\times} \hat{c}_\chi(m, h, v) \ebf(- m u)$$ to obtain
\begin{align*}
y^{-k}& \overline{\Phi(z; p_{k-b-1, b, k}, f)} =    y^{1-k} c(f; p_{k-b-1, b, k}) + \\
&\quad  y^{-h-j+1/2}  \sum_{\begin{subarray}{c} h, h^-, j, r \ge 0 \\ d \ge 1, \lambda \in \df^{-1} \backslash \{0\}\end{subarray}} 
c_{k, b, h, h^-, j, r}
 d^{k - h - j - 1/2 }
  (\lambda - \lambda')^{k-1-2\ell-2r}(\lambda +\lambda')^{k-h^- -2j + 2r}\\
&\quad  \times \ebf(d(\lambda + \lambda')x) \int^\infty_0 \hat{c}_\chi(\lambda\lambda', \lambda, dyv) \exp \lp - \pi d y \lp \frac{1}{v} + (\lambda^2 + (\lambda')^2 )v\rp \rp v^{\alpha} \frac{dv}{v},
\end{align*}
where $\alpha = \alpha(b, h, h^-, j) := b + h - h^- - j + 1/2$ and
$$
c_{k, b, h, h^-, j, r} := \frac{ \binom{k}{h^-} \binom{k-b-1}{ h} \binom{h^-}{h} h! (2j)! \binom{j}{r} \binom{b}{2r} \binom{k-h^-}{2j-2r} (-1)^{1 + j}}{ j! \binom{2j}{2r} 2^{k + 2j + h - 3/2} \pi^{h + j}}.
$$
By Lemmas \ref{lemma:bound}, \ref{lemma:Besselestimate} and Remark \ref{rmk:asymptotic}, we have the estimate
\begin{align*}
   \int^\infty_0 &\hat{c}_\chi(\lambda\lambda', \lambda, dyv) \exp \lp - \pi d y \lp \frac{1}{v} + (\lambda^2 + (\lambda')^2 )v\rp \rp v^{\alpha} \frac{dv}{v} = \\
& \frac{ e^{-2\pi dy(|\lambda| + |\lambda'|)}}{\sqrt{dy}(|\lambda| + |\lambda'|)^{\alpha + 1/2}}
\lp \tilde{c}_\chi(\lambda \lambda', \lambda)
 +  O_{F, S_F, k}\lp \frac{(d(|\lambda| + |\lambda'|))^2}{\sqrt{y}}
\lp 1 + y^{-1} \rp \rp \rp,
\end{align*}
where we have used $|\lambda \lambda'| \ll  (|\lambda| + |\lambda'|)^2$ and  $1 \le d$. 
Now, we consider the coefficient of $\ebf(mx)$ in $y^{-k}\overline{\Phi(z; p_{k-b-1, b, k}, f)}$ for a fixed $m \in \Qb$.
When $m = 0$, the claimed asymptotics is clearly true.
Suppose $m \neq 0$. 
This imposes the condition $m = d(\lambda + \lambda')$.
The sum of the terms with $h, j, r > 0$, or $\lambda \lambda' < 0$ gives an asymptotic term of $O_{F, S_F, k}(|m|^{k+3}y^{-1/2})$.
The other terms, namely those satisfying $h = j = r = 0$ and $\lambda \lambda' > 0$, will contribute the following main term
\begin{align*}
  &  y^{1/2}  \sum_{\begin{subarray}{c} h^- \ge 0, \; d \ge 1\\ \lambda \in \df^{-1}, \; \lambda \lambda' > 0\\ d(\lambda + \lambda') = m\end{subarray}} 
{- 2^{3/2 - k} \binom{k}{h^-}}
 d^{k - 1/2 }
  (\lambda - \lambda')^{b}(\lambda +\lambda')^{k-h^-}
\frac{\tilde{c}_\chi(\lambda\lambda', \lambda)   \ebf(mx)e^{-2\pi |m| y}}{\sqrt{dy}(|\lambda| + |\lambda'|)^{b - h^- + 1}}\\
  &=  - 2^{3/2 - b} m^{k-b-1} \sum_{\begin{subarray}{c} d \ge 1\\ \lambda \in \df^{-1}, \; \lambda \lambda' > 0\\ d(\lambda + \lambda') = m\end{subarray}} 
  (d(\lambda - \lambda'))^{b}
\ebf(mx)e^{-2\pi |m| y}
\tilde{c}_\chi(\lambda\lambda', \lambda)   
\sum_{h^- = 0}^{k} \binom{k}{h^-} 
\sgn(\lambda)^{k-h^-}
\\
&= - 2\sqrt{2} m^{k-1} \ebf(mz)
  \sum_{\begin{subarray}{c}  \lambda_0 \in \df^{-1}, \; \lambda_0 \gg 0\\ \tr(\lambda_0) = m\end{subarray}} 
  \lp \frac{\lambda_0 - \lambda'_0}{m} \rp^{b}
\sum_{d \ge 1, \; d \mid \lambda_0} \tilde{c}_\chi(\Nm(\lambda_0/d), \lambda_0/d)   
 \end{align*}
plus an error of $O_{F, S_F, k}(|m|^{k+3} y^{-1/2})$.
Notice that the main term is present only when $m > 0$. 
Adding this together over $0 \le b \le k-1$ with the factor $c_{k-1, b}$ from \eqref{eq:Pnexplicit} then finishes the proof.
\end{proof}

The theta lift above becomes a classical one if $\hat{\vartheta}_\chi \in \Ac_{1, \rho_\Delta}$ is replaced by any holomorphic cusp form $f \in S_{1, \rho_\Delta}$.

\begin{prop}
  \label{prop:classical}
For any cusp form $f(\tau) = \sum_{h \in A_\Delta} \ef_h \sum_{m \in \Qb_{>0}} c_f(n, h) q^n \in S_{1, \rho_\Delta}$ and even $k \in \Nb$, the theta integral $(y^k \Phi_{}(z; \tp_{k-1}, f^c))^c = \sum_{m \ge 1} a(m) \ebf(mz)$ is a holomorphic cusp form in $S_{2k}$ with 
$$
a(m) :=  -2\sqrt{2}\sqrt{\Delta}^{1-k}
\sum_{\mu_0 \in \mathrm{S}_m}
C_k(\mu_0, m)
 \sum_{(d) \mid (\mu_0)} {c}_f (\Nm(\mu_0/(d\sqrt{\Delta}), \mu_0/(d\sqrt{\Delta})).
$$
\end{prop}

\begin{proof}
 The calculation from the proof of Prop.\ \ref{prop:preimage} shows that $L_z$ annihilates $y^{k} \Phi_{}(z; \tp_{k-1}, f^c))^c$, i.e.\ it is holomorphic. The Fourier coefficients are calculated in the same way as in the proof of Prop.\ \ref{prop:FE1}.
\end{proof}
\subsection{Lift 3}
Let $f(z) = \sum_{m \ge -m_0} c_f(m) \ebf(mz) \in M^!_{2-2k}$ be a weakly holomorphic modular form with $k \in 2 \Nb$.
In this section, we will analyze a constant defined by the following integral
\begin{equation}
  \label{eq:constant}
c(f, S_F) :=  \frac{1}{4\sqrt{2}} \lim_{T \to \infty} \int_{\Fc_T} f(z) y^{1-k}\overline{ \Phi(z; \hp_{k-1}, \theta_\chi) }d\mu(z) \in \Cb.
\end{equation}
The main result is as follows.
\begin{prop}
  \label{prop:rational}
Suppose that the principal part coefficients of $f$ are rational. Then 
\begin{equation}
\sqrt{\Delta} c(f, S_F) \in \frac{1}{\kappa(f, S_F)} \Zb \subset \Qb
\end{equation}
 with $\kappa(f, S_F) \in \Nb$ a constant depending only on $f$ and $S_F$.
\end{prop}

We break the proof of this proposition into several steps. First, we will rewrite $c(f, S_F)$ as a double integral and change the order of integration. Then, the inner integral becomes a Millson theta lift \cite{AS18} and produces a weakly holomorphic modular form of weight $\frac{3}{2} - k$.
By a well-known result of Shimura \cite{Sh75} (see also \cite{MP10, BO13}), this form will have rational Fourier coefficients if $f$ does.
Finally, the outer integral becomes a Borcherds' type regularized theta lift, which we can evaluate using the result of \cite{BS17} and Stokes' theorem.

We start by analyzing the lattice involved in $\Phi(z; \hp_{k-1}, \theta_\chi)$. Recall that 
$$
\theta_\chi(\tau; S_F) = 
 \sum_{\bfrak \in S_F} 
{\chi(\bfrak)}{}   \Theta_{\Nm^-(\bfrak)}^-(\tau, 0) \in \Ac_{1, \overline{\rho_\Delta}}
$$
from \eqref{eq:Theta12} and \eqref{eq:theta}. 
Given such $\bfrak \in S_F$, define lattice
\begin{equation}
  \label{eq:Mb2}
  M_{\bfrak} := \left\{
( \Lambda , \lambda ) \in M^\vee \times \df^{-1} \Nm^-(\bfrak) :  \Lambda = \lambda \in A_\Delta
\right\},
\end{equation}
with the quadratic form $Q((\Lambda, \lambda)) := \det(\Lambda) + {\Nm(\lambda)}$.
This turns $M_\bfrak$ into a unimodular lattice of signature $(3, 3)$, which contains $M \oplus L_{\Nm^-(\bfrak)}$ as a sublattice.
Let $\Dc$ be the Grassmannian associated to $M_\bfrak \otimes \Rb$. 
For $z \in \Hb$, we have the following identification
\begin{equation}
  \begin{split}
\nu_z:      M_\bfrak \otimes \Rb &= (M \otimes \Rb) \oplus (L_{\Nm^-(\bfrak)} \otimes \Rb) \cong \Rb^{3, 3} = \Rb^{2, 2} \oplus \Rb^{1, 1}\\
(\Lambda, \lambda) &\mapsto \lp \Lambda_z, \Br (\Lambda, Y), \frac{\lambda + \lambda'}{\sqrt{2} }, 
\Re(\Lambda_{z^\perp}), \Im(\Lambda_{z^\perp}), \frac{\lambda - \lambda'}{\sqrt{2} } \rp,
  \end{split}
\label{eq:nuz}
\end{equation}
where $\Lambda_z$ and $\Lambda_{z^\perp}$ are defined in \eqref{eq:Lambdaz}.
This gives rise to a point in $\Dc$, which we also denote by $z$. 
Then we can rewrite
\begin{equation}
  \label{eq:Phi2}
  \begin{split}
\sqrt{2}{      \Phi(z; \hp_{k-1}, \theta_\chi)} &= 
\sqrt{2} \sum_{\bfrak \in S_F}  \chi(\bfrak)
\int_{\Gamma\backslash \Hb} \sum_{h \in A_\Delta}
\overline{\Theta_{M, h}(\tau, z; \hp_{k-1})} {\Theta^-_{\Nm^-(\bfrak), h}(\tau, 0)}
v^{} d\mu(\tau)
\\
&= 2 \sum_{\bfrak \in S_F}  {\chi(\bfrak)} \Phi_{M_\bfrak}(z; p^\dagger_{k-1}, \mathds{1}), \\
  \end{split}
\end{equation}
where $\mathds{1}$ denotes the constant function on $\Gamma \backslash \Hb$ and
$$
p^\dagger_{k-1}(\underline{x}) := (ix_1)^{k} P_{k-1} \lp \frac{x_2}{ix_1} \rp (i x_4 + x_5)^{k-1}  x_6
$$
 for $k \in \Nb$ and $\underline{x} = (x_1, \dots, x_6) \in \Rb^{3, 3}$.
Therefore, we can write
$$
c(f, S_F) = \frac{1}{4} \sum_{\bfrak \in S_F} \chi(\bfrak) c(f, \bfrak).
$$
with
\begin{equation}
  \label{eq:cfb}
c(f, \bfrak) := \lim_{T \to \infty} \int_{\Fc_T} f(z) y^{1-k} \overline{\Phi_{M_\bfrak}(z; p^\dagger_{k-1}, \mathds{1})} d\mu(z) \end{equation}
and it suffices to show that $c(f, \bfrak) \in \frac{\sqrt{\Delta}}{\kappa(f, \bfrak)} \Zb$ for some $\kappa(f, \bfrak) \in \Nb$.

To execute the first step, we need to find a sublattice in $M_\bfrak$ that behaves well with respect to the polarization in \eqref{eq:nuz}.
Let $L_1 :=  \left\{\smat{a}{c}{c}{d} \in M_2(\Zb)\right\} $ be a signature $(1, 2)$ sublattice of $M_2(\Zb)$, $B:= \Nm(\bfrak) = \Nm(\bfrak')$ and 
\begin{equation}
  \label{eq:L2}
L_2 := P_\Delta \oplus P_{ B^2} \oplus P^-_{\Delta B^2}  ,
\end{equation}
where $P_N$ is the positive definition lattice in Example \ref{ex:P} for $N \in \Nb$. Then $L_1 \oplus L_2 $ embeds isometrically into $M \oplus L_{\Nm^-(\bfrak)} \subset M_\bfrak$ by
\begin{align*}
\iota:  L_1 \oplus L_2  &\hookrightarrow M_\bfrak \\
(\Lambda, r_1, r_2, r_3) &\mapsto \lp \Lambda + r_1 \sqrt{\Delta} \pmat{0}{1}{-1}{0}, B(r_2 + r_3 \sqrt{\Delta}) \rp.
\end{align*}
We use $\wtM$ to denote the image of this embedding and use it to identify $\wtM^\vee/\wtM$ with $L_1^\vee/L_1 \oplus L^\vee_2/L_2$.
By linearity $\iota$ can be extended to $\wtM\otimes\Rb$ and composing with it $\nu_z$ from \eqref{eq:nuz} gives us
\begin{align*}
  \nu_z \circ \iota :   L^\vee_1 \oplus L^\vee_2  &\hookrightarrow \Rb^{3, 3} \\
\lp\Lambda, r_1, r_2, r_3\rp  &\mapsto  
\lp \Lambda_z, -r_1\sqrt{2 \Delta}, \sqrt{2}Br_2, \Re(\Lambda_{z^\perp}), \Im(\Lambda_{z^\perp}), Br_3 \sqrt{2\Delta}\rp.
\end{align*}
For $\underline{x} \in \Rb^{1, 2}$ and $n, b, N \in \Nb$, let $p^*_{n, b}(\underline{x}) := (ix_1)^{n-b+1} (i x_2 + x_3)^n$ be a polynomial on $\Rb^{1, 2}$ and $\theta_N(\tau; n)$ the unary theta functions in Example \ref{ex:P}.
Define
\begin{equation}
  \label{eq:thetab}
  \thetab_{L_2}(\tau; b, B) := (-1)^b \theta_\Delta(\tau; b) \otimes \theta_{B^2}(\tau; 0) \otimes \theta^c_{\Delta B^2}(\tau; 1).
\end{equation}
If we denote $\psi: \Cb[\wtM^\vee/\wtM] \to \Cb[M^\vee_\bfrak/M_\bfrak]$ and $c_{k-1, b}$ as in \eqref{eq:psi} and \eqref{eq:Pnexplicit} respectively, then 
\begin{align*}
  \theta_{M_\bfrak}(\tau, z; p^\dagger_{k-1}) &= \psi \lp \Theta_{\wtM}(\tau, z; p^\dagger_{k-1}) \rp
=
 \psi \lp 
\sum_{b = 0}^{k-1} c_{k-1, b} \Theta_{L_1}(\tau, z; p^*_{k-1, b}) \otimes \thetab_{L_2}(\tau; b, B)
\rp.
\end{align*}

As a function of the variable $z \in \Hb$, $y^{1-k} \Theta_{L_1}(\tau, z; p^*_{k-1, b})$ is a modular form in $\Ac_{2k-2}$.
When $b = k-1$, this is one of the theta kernels studied by Kudla-Millson \cite{KM90}, and is also known as the Millson kernel (see \cite{AS18}).
The theta kernel for the other odd $b$ can be obtained from $b = k-1$ via the raising operator.
One could switch to the Fock model as in section \ref{sec:Fock} (or just do a straightforward computation) to verify the following result.

\begin{lemma}
  \label{lemma:raise12}
For all $\tau, z \in \Hb$ and $b, k \in \Nb$, we have
\begin{equation}
  \label{eq:raise12}
R_{\tau, -b + 1/2} \Theta_{L_1}(\tau, z; p^*_{k-1, b}) = 2\pi \Theta_{L_1}(\tau, z; p^*_{k-1, b + 2})
\end{equation}
where $R_{\tau, -b+1/2}$ is the raising operator in \eqref{eq:diffops}.
\end{lemma}

Since $f \in M^!_{2-2k}$, we can use $\Theta_{L_1}$ to define the following regularized integral
\begin{equation}
  \label{eq:If}
  I(\tau, f; b) := \lim_{T \to \infty} \int_{\Fc_T} f(z) \Theta_{L_1}(\tau, z; p^*_{k-1, b})y^{1-k} d\mu(z),
\end{equation}
which is a modular form in $\Ac_{-b + 1/2, \rho_{L_1}}$.
To evaluate $c(f, \bfrak)$, it turns out that we could change the order of integration in $\tau$ and $z$.
\begin{lemma}
  \label{lemma:interchange}
Let $c(f, \bfrak)$ be the limit defined in \eqref{eq:cfb}. Then
\begin{equation}
  \label{eq:cfb1}
  c(f, \bfrak) = 
\sum_{b = 0}^{k-1} c_{k-1, b}
\lim_{T' \to \infty} \int_{\Fc_{T'}} \psi
\lp 
 I(\tau, f; b)
\otimes { \thetab_{L_2}(\tau; b, B)} \rp  d\mu(\tau).
\end{equation}
\end{lemma}

\begin{proof}
Using the definition of $\Theta_{M_\bfrak}$, we can rewrite \eqref{eq:cfb} as
\begin{align*}
c(f, \bfrak) &= \lim_{T \to \infty} \lim_{T' \to \infty} \int_{\Fc_T} \int_{\Fc_{T'}} 
\psi\lp
\sum_{b = 0}^{k-1} c_{k-1,b}  f(z)\Theta_{L_1}(\tau, z; p^*_{k-1, b}) \otimes {\thetab_{L_2}}(\tau; b, B) \rp y^{1-k} \frac{dudv}{v^2} \frac{dxdy}{y^2}.
\end{align*}
The region $\Fc_T \times \Fc_{T'}$ is compact and there is no problem with changing the order of integration. 
The sum over $b$ is also a finite sum.
Therefore, we need to argue that the limit in $T$ can be moved inside. To achieve this, it suffices to consider each component and prove this with $\Fc_T \times \Fc_{T'}$ replaced by $\mathcal{R}_T \times \mathcal{R}_{T'} = [-1/2, 1/2]^2 \times [1, T] \times [1, T']$, where $\mathcal{R}_T := \Fc_T \backslash \Fc_1$. 
We can first integrate in $u$ to obtain
\begin{align*}
\int^{1/2}_{-1/2} \Theta_{L_1, h_1}&(\tau, z; p^*_{k-1, b} )
{\thetab}_{L_2, h_2}(\tau; b, B)  du = 
i^{2k-b-1} v^{k+3/2} \sum_{d \in \Zb} \mathrm{S}_d(v, x, y; b),
\end{align*}
for $h_1 \in L^\vee_1/L_1$ and $h_2 \in L^\vee_2/L_2$, where
\begin{align*}
\mathrm{S}_d(v, x, y; b) :=
\sum_{\begin{subarray}{c} \Lambda = \smat{a}{c}{c}{d} \in L_1 + h_1\\ r = (r_1, r_2, r_3) \in L_2 + h_2\\\det(\Lambda) + \Delta r_1^2  + B^2 r_2^2 = \Delta B^2 r_3^2 \end{subarray}}
&\Hc_v(x_1^{k-b})\mid_{x_1 = \Lambda_z} (\overline{\Lambda_{z^\perp}})^{k-1} c(r, h_2, v; b)\\
& e^{-{\pi v} \lp  \Lambda_z^2 + |\Lambda_{z^\perp}|^2 + 2\Delta r_1^2 + 2 B^2 r_2^2 + 2 \Delta B^2 r_3^2 \rp}
\end{align*}
and $c(r, h_2, v; b)$ is the $r$\tth Fourier coefficient of $\thetab_{L_2, h_2}(\tau; b, B)$, i.e.
$$
\thetab_{L_2, h_2}(\tau; b, B) = \sum_{ r = (r_1, r_2, r_3) \in L_2 + h_2} c(r, h_2, v; b) \ebf \lp (\Delta r_1^2 + B^2 r_2^2)\tau  - \Delta B^2 r_3^2 \overline{\tau}\rp.
$$
From the definition, it is clear that $|c(r, h_2, v; b)| < p(r) $ for $v \in [1, \infty)$ with $p$ a polynomial in $r_1, r_2, r_3$. 
Furthermore, $c((r_1, r_2, 0), h_2, v; b) = 0$ for all $r_1, r_2 \in \Qb$ and $b \in \Nb$.
Now, it suffices to consider 
\begin{equation}
\label{eq:int1}
\mathrm{I}_d(m, T, T'; b) := \int_{1}^T 
\int_{1}^{T'}
\int^{1/2}_{-1/2} \ebf(mz) \mathrm{S}_d(v, x, y; b) v^{k-1/2} y^{-k-1} dx dv dy,
\end{equation}
since we want to bound $\lim_{T \to \infty} \lim_{T' \to \infty} \sum_{m \ge -m_0} \sum_{d \in \Zb} c_f(m) \mathrm{I}_d(m, T, T'; b)$.
For $\Lambda = \smat{a}{c}{c}{d} \in L_1 \otimes \Rb$, it is easy to check that
$$
\Lambda_z^2 + \left|\Lambda_{z^\perp}\right|^2 = d^2y^2 + 2(dx - c)^2 + \frac{(dx^2 - 2cx + a)^2}{y^2} \ge d^2 y^2.
$$
When $d \neq 0$, the integrand in \eqref{eq:int1} is bounded for any $m \in \Qb$ and $x, v, y$. Furthermore, it decays exponentially as $d, v, y$ approaches infinity. 
The same holds for $d = 0$ and $m > 0$.
The sum over $d \in \Zb$ and $m > 0$ weighted by $c_f(m)$ then converges uniformly and absolutely. 
Therefore, we have
\begin{align*}
  \lim_{T \to \infty} \lim_{T' \to \infty} \sum_{m > 0 \text{ or } d \neq 0} c_f(m) \mathrm{I}_d(m, T, T'; b) &=
\sum_{m > 0 \text{ or } d \neq 0} \lim_{T \to \infty} \lim_{T' \to \infty}  c_f(m) \mathrm{I}_d(m, T, T'; b) =\\
\sum_{m > 0 \text{ or } d \neq 0} \lim_{T' \to \infty} \lim_{T \to \infty}  c_f(m) \mathrm{I}_d(m, T, T'; b) &=
\lim_{T' \to \infty} \lim_{T \to \infty} \sum_{m > 0 \text{ or } d \neq 0} c_f(m) \mathrm{I}_d(m, T, T'; b).
\end{align*}
So it suffices to consider $\mathrm{I}_d(m, T, T'; b)$ with $d = 0$ and $m \le 0$. In this case, $\mathrm{S}_d$ simplifies to
\begin{align*}
\mathrm{S}_0(v, x, y; b) =
\sum_{\begin{subarray}{c} c\in \Zb + h_1\\ r = (r_1, r_2, r_3) \in L_2 + h_2\\ \Delta r_1^2  + B^2 r_2^2 = c^2 +  \Delta B^2 r_3^2 \end{subarray}}
& c(r, h_2, v; b) 
e^{-2 \pi v \lp \Delta r_1^2 + B^2 r_2^2 +  \Delta B^2 r_3^2 \rp} \times
\\
&\sum_{a \in \Zb}
\Hc_v(x_1^{k-b}) \mid_{x_1 = \tfrac{a-2cx}{\sqrt{2}y}} \lp \frac{a - 2c\overline{z}}{\sqrt{2} y}\rp^{k-1}
 e^{- \pi v \lp \frac{( a - 2cx)^2}{y^2} + 2c^2 \rp}.
\end{align*}
By Poisson summation, the inner sum becomes
$$
\sum_{a \in \Zb}
\Hc_v(x_1^{k-b}) \mid_{x_1 = \tfrac{a-2cx}{\sqrt{2}y}} \lp \frac{a - 2c\overline{z}}{\sqrt{2} y}\rp^{k-1}
 e^{- \pi v \lp \frac{( a - 2cx)^2}{y^2} \rp}
=
\sum_{n \in \Zb}
\ebf(2ncx) \frac{y\mathrm{p}_{k, b}(ny, c, v)}{v^{2k-b-1/2}} e^{-\frac{\pi y^2 n^2}{v }}
$$
for a polynomial $\mathrm{p}_{k, b}$, which can be evaluated using Corollary 3.3 of \cite{Borcherds98}.
When $m = 0$, we have
\begin{align*}
\mathrm{I}_0(0, T, T'; b) = \int_{1}^T 
\int_{1}^{T'}
\sum_{\begin{subarray}{c} c\in \Zb + h_1\\ r = (r_1, r_2, r_3) \in L_2 + h_2\\ \Delta r_1^2  + B^2 r_2^2 = c^2 +  \Delta B^2 r_3^2 \end{subarray}}
&
 {\mathrm{p}_{k, b}(0, c, v)} 
 c(r, h_2, v; b) 
e^{-2 \pi v \lp \Delta r_1^2 + B^2 r_2^2 +  \Delta B^2 r_3^2 + c^2 \rp} 
 \frac{dv dy}{v^{k - b} y^k}.  
\end{align*}
The sum converges absolutely and uniformly in $v, y$ since $\mathrm{p}_k(0, c, v)c(r, h_2, v; b)$ is bounded by a polynomial in $c, v, r_1, r_2, r_3$ and $c((0, 0, 0), h_2, v; b) = 0$.
Since $k \ge 2$, the limit in $T$ also exists and we can interchange the limits in $T$ and $T'$.

When $m < 0$, we have
\begin{align*}
\mathrm{I}_0(m, T, T'; b) = \int_{1}^T 
\int_{1}^{T'}
\sum_{\begin{subarray}{c} c\in \Zb + h_1, \; 2c \mid m\\ r = (r_1, r_2, r_3) \in L_2 + h_2\\ \Delta r_1^2  + B^2 r_2^2 = c^2 +  \Delta B^2 r_3^2 \end{subarray}}
& c(r, h_2, v; b) 
e^{-2 \pi v \lp \Delta r_1^2 + B^2 r_2^2 +  \Delta B^2 r_3^2 \rp} \times
\\
\sum_{n \in \Zb, 2nc = m} & {\mathrm{p}_{k, b}(ny, c, v)} e^{- \frac{\pi y^2 n^2}{v} - 2\pi m y - 2\pi c^2 v} 
 \frac{dv}{v^{k-b}} \frac{dy}{y^k}.  
\end{align*}
We can suppose that $r_3 \neq 0$ in the sum above since $c((r_1, r_2, 0), h_2, v; b)$ vanishes identically.
Using the equalities $\Delta r_1^2 + B^2 r_2^2 = c^2 + \Delta B^2 r_3^2$ and $m = 2nc$, we can bound the exponent as
\begin{align*}
  - \frac{\pi y^2 n^2}{v} - 2\pi m y - 2\pi c^2 v 
-2 \pi v \lp \Delta r_1^2 + B^2 r_2^2 +  \Delta B^2 r_3^2 \rp 
&=
- \frac{\pi }{v} (yn + 2cv)^2 - 4\pi \Delta B^2 r_3^2 v \\
< - \alpha \frac{y^2}{v} - \beta v
& < - \epsilon_1 v - \epsilon_2 y
\end{align*}
for all $v, y \ge 1$ with $\alpha, \beta, \epsilon_1, \epsilon_2 > 0$ absolute constants depending only on $B, \Delta$ and $m_0$.
Therefore the integrand decays exponentially in $v, y$, and
$$
\lim_{T\to \infty}\lim_{T'\to \infty} \sum_{-m_0 \le m \le 0} c_f(m) \mathrm{I}_0(m, T, T'; b) = 
\lim_{T' \to \infty}\lim_{T \to \infty} \sum_{-m_0 \le m \le 0} c_f(m) \mathrm{I}_0(m, T, T'; b) .
$$
This then finishes the proof.
\end{proof}

We could further simplify the formula for $c(f, \bfrak)$ in \eqref{eq:cfb1} and express it using the Rankin-Cohen bracket in \eqref{eq:RC}.

\begin{lemma}
  \label{lemma:cfb2}
In the notations above, we have
\begin{equation}
  \label{eq:cfb2}
c(f, \bfrak) = - 2^{k/2 - 1} \lim_{T' \to \infty} \int_{\Fc_{T'}} \psi \lp 
\mathrm{RC}(\tau; f, \Delta) \otimes \theta_{B^2}(\tau; 0) \otimes \theta^c_{\Delta B^2}(\tau; 1)
\rp d \mu(\tau),
\end{equation}
where $\mathrm{RC}(\tau; f, \Delta) := [I(\tau, f; k-1), \theta_\Delta(\tau;1)]_{k/2 - 1}$ is a modular form of weight 1 with $[\cdot, \cdot]_r$ the Rankin-Cohen bracket defined in \eqref{eq:RC}.
\end{lemma}

\begin{proof}
  To deduce \eqref{eq:cfb2} from \eqref{eq:cfb1}, it suffices to prove that 
$$
\sum_{b = 0}^{k-1} c_{k-1, b} I(\tau, f; b) \otimes (-1)^b \theta_\Delta(\tau; b) = 
- 2^{k/2 - 1} [I(\tau, f; k-1), \theta_\Delta(\tau; 1)]_{k/2-1}.
$$
Since $k$ is even, $c_{k-1, b}$ vanishes whenever $b$ is even.
By \eqref{eq:Runary} and Lemma \ref{lemma:raise12}, we can rewrite the left hand side as
$$
-(2\pi)^{-r} \sum_{s = 0}^{r} (-1)^{s} c_{k-1, 2s + 1} R_{3/2-k}^{r - s}I(\tau, f; k-1) \otimes R^{s}_{3/2} \theta_\Delta(\tau; 1)
$$
after substituting $s := (b-1)/2$ and $r:= k/2 - 1$.
From definition, one could verify that 
$$
c_{k-1, 2s + 1} = (-1)^r \binom{3/2 - k  + r -1}{s} \binom{3/2 + r - 1}{k/2-1-s}
$$
for all $0 \le s \le r = k/2-1$. Using the definition \eqref{eq:RC}, we see that both sides agree.
\end{proof}

The last lemma before the proof of Prop.\ \ref{prop:rational} concerns with the rationality of the lift $I(\tau, f)$.
\begin{lemma}
  \label{lemma:rationality}
Suppose $f \in M^!_{2-2k}$ has rational Fourier coefficients. Then the function $I(\tau, f; k-1)$ defined in \eqref{eq:If} is a weakly holomorphic modular form in $M^!_{3/2 - k, \rho_{L_1}}$ with rational Fourier coefficients.
\end{lemma}

\begin{proof}
The Fourier expansion of $I(\tau, f; k-1)$ has been calculated in \cite{AS18} (see Theorem 5.1 loc.\ cit.). For $m \in \Qb_{> 0}$ and $h_1 \in L^\vee_1/L_1$, the Fourier coefficient of $I_{h_1}(\tau, f)$ is given by
$$
c(m, h_1) := \frac{1}{2{m}^{k/2}} (4\pi)^{1-k} \sum_{X \in \Gamma \backslash (L_{1} + h_1), \det(X) = m} \frac{1}{|\overline{\Gamma}_X|} (R^{k-1}_{2-2k} f)(z_X),
$$
where $\Gamma = \SL_2(\Zb)$ and $z_X \in \Hb$ is the CM point associated to the positive vector $X$.
Since $f$ is weakly holomorphic, the function $R^{k - 1}_{2-2k}f$ is a nearly holomorphic function. 
It is a well-known result of Shimura that the values $(4\pi)^{1-k} R^{k-1}_{2-2k}f$ at CM points are in the ring class field of an imaginary quadratic field \cite{Sh75}.
In fact, the set $$
\{(4\pi)^{1-k} (R^{k-1}_{2-2k} f)(z_X): X \in \Gamma \backslash L_1 + h_1, \det(X) = m\}$$
 is the union of Galois orbits (see \cite[Theorem 1.1]{BO13} and \cite[Prop.\ 3.1]{MP10}).
Therefore, $c(m, h_1) \in \Qb$ for all $m \in \Qb_{>0}$ and $h_1 \in L^\vee_1/L_1$, which implies that $I(\tau, f; k-1)$ has all rational Fourier coefficients. 
\end{proof}

\begin{proof}[Proof of Prop.\ \ref{prop:rational}]
  By Theorem 4.7 of \cite{BS17}, we know that there exists a harmonic Maass form $\widetilde{\theta}_{\Delta B^2}(\tau) \in \Ac_{1/2, \rho_{\Delta B^2}}$ such that $L_2 \widetilde{\theta}_{\Delta B^2}(\tau) = \theta^c_{\Delta B^2}({\tau}, 1)/\sqrt{2}$ and the holomorphic part $\widetilde{\theta}_{\Delta B^2}^+$ has rational Fourier coefficients. Substituting this into \eqref{eq:cfb2} and applying \eqref{eq:RLswitch} yields
  \begin{align*}
\sqrt{\Delta}\cdot      c(f, \bfrak) &=  \lim_{T' \to \infty} \int_{\Fc_{T'}}
  \psi( 
g(\tau)
\otimes
L_2 \widetilde{\theta}_{\Delta B^2}(\tau))
  d\mu(\tau) \\
&= \mathrm{Const} ( \psi(g(\tau) \otimes \widetilde{\theta}^+_{\Delta B^2} (\tau))),
  \end{align*}
where $g(\tau) :=  \sqrt{2\Delta} \cdot [I(\tau, f; k-1), \theta_{\Delta}(\tau; 1)]_{k/2-1} \otimes \theta_{B^2}(\tau, 0) \in M^!_{3/2, \rho_{L_1} \otimes \rho_\Delta \otimes \rho_{B^2}}$. Since $f$ has rational Fourier coefficients, so does $g(\tau)$ by Example \ref{ex:P} and Lemma \ref{lemma:rationality}. Therefore, $\sqrt{\Delta} \cdot c(f, \bfrak) \in \Qb$ and the denominator only depends on $f$ and $\bfrak$. This finishes the proof of Prop.\ \ref{prop:rational}.
\end{proof}

\appendix
  
\section{Calculations in Fock Model}
\label{sec:Fock}
In this section, we will give some differential equations satisfied by theta kernels. 
Even though these follows from straightforward calculations, the steps are long and tedious. Instead, we follow \cite{KM90} (see also \cite[Appendix]{FM06}) and switch to the Fock model of the Weil representation, where the actions of differential operators can be described elegantly using elements in the Lie algebra $\slf_2(\Cb)$.

\subsection{Fock Model}
For our purpose, we restrict to the case $(V, Q) = (M_2(\Rb), \det)$. 
We identify it with $\Rb^{2, 2}$ with respect to the orthogonal basis
\begin{equation}
  \label{eq:vs}
  \begin{split}
      v_1 &:= \Re Z(i) = \frac{1}{\sqrt{2}} \pmat{1}{0}{0}{1}, \; 
  v_2 := \Im Z(i) = \frac{1}{\sqrt{2}} \pmat{0}{-1}{1}{0}, \\
  v_3 &:= \Re Z^\perp(i) = \frac{1}{\sqrt{2}} \pmat{-1}{0}{0}{1}, \; 
  v_4 := \Im Z^\perp(i) = \frac{1}{\sqrt{2}} \pmat{0}{1}{1}{0},
  \end{split}
\end{equation}
where $Z(z)$ and $Z^\perp(z)$ are defined as in \eqref{eq:Z}.
We define $V_+ := \Rb v_1 \oplus \Rb v_2, V_- := \Rb v_3 \oplus \Rb v_4$ and can identify $\bigwedge^2 V$ with the Lie algebra $\mathfrak{o}(V)$ via the isomorphism $\rho: \bigwedge^2 V \to \mathfrak{o}(V)$ given by
$$
\rho(v' \wedge v'')(v') := (v, v'') v' - (v, v') v'',
$$
where $(,)$ is the associate bilinear form.
For $1 \le i < j \le 4$, let $X_{ij}$ denote the image of $v_i \wedge v_j$ under $\rho$.

As in section \ref{subsec:setup22}, where we considered the diagonally embedded symmetric space of $\SL_2$ in that of $\SO(V)$, we can embed the group $\SL_2$ into $\SO(V)$ in the following compatible way
\begin{equation}
  \label{eq:hgamma}
  \begin{split}
    \iota_\Delta:   \SL_2(\Rb) &\to \SO(V) \\
\gamma & \mapsto h(\gamma): A \mapsto \gamma \cdot A\cdot  \transpose{\gamma}.
  \end{split}
\end{equation}
The pushforward of $\iota_\Delta$ identifies the Lie algebra $\slf_2(\Rb)$ as a subalgebra of $\mathfrak{o}(V)$.
On the generating elements $E := \smat{0}{1}{0}{0}, F:= \smat{0}{0}{1}{0}$ and $H := \smat{1}{0}{0}{-1}$, it is easily checked
$$
\iota_{\Delta *}(E) = X_{14} + X_{34}, \;
\iota_{\Delta *}(F) = X_{14} - X_{34}, \;
\iota_{\Delta *}(H) = -2 X_{13}.
$$
The elements $R:= \half(H + iE + iF), L:= \half(H - i E - iF)$ in $\slf_2(\Cb)$ are then sent to $-X_{13} + i X_{14}$ and $-X_{13} - iX_{14}$ respectively.

On the symplectic side, we take $W = \Rb e + \Rb f$ to be the real vector space with the skew-symmetric pairing $\langle, \rangle$ satisfying $\langle e, f\rangle = 1$, and the positive definite complex structure $J := \smat{0}{-1}{1}{0}$ with respect to the basis $\{e, f\}$.
Then $W\otimes \Cb = W' + W''$ with $W' := \Cb w',\, W'' := \Cb w''$ and
$$
w' := e - if, w'' := e + if
$$
eigenvectors of $J$ with eigenvalues $i$ and $-i$ respectively.
For $a, b \in W$, we denote $a \circ b$ the element $a \otimes b + b \otimes a$ in the symmetric algebra $\mathrm{Sym}^2(W)$, and identify $\mathrm{Sym}^2(W)$ with $\spf(W)$ via the $\Rb$-linear map $\varrho: \mathrm{Sym}^2(W) \to \spf(W)$ given by
\begin{align*}
\varrho(a \circ b)( c) := \langle a, c \rangle b + \langle b, c \rangle a.
\end{align*}
Under the isomorphism $\SL_2(\Rb) \cong \Sp(W)$ induced by left multiplication, the Lie algebras $\slf_2(\Cb)$ and $\spf(W\otimes \Cb)$ are isomorphic and $L, R$ are identified with $-\frac{i}{4} w' \circ w', \frac{i}{4} w'' \circ w''$ respectively.

Let $\Wb := V \otimes W$ be the symplectic space with the skew-symmetric form $(,)\otimes \langle, \rangle$ and $\omega$ the oscillator representation of $\Sp(\Wb)$, which contains $\Sp(W) \times \mathrm{O}(V)$.
Different polarizations of $\Wb$ gives rise to different models $\omega$. We recover the Schr\"{o}dinger model by taking the polarization $\Wb = V \otimes\Rb e + V \otimes \Rb f$, where $\omega$ acts on $\Ss(V)$ by
$$
(\omega(h)(\varphi))(x) := \varphi(h^{-1} x), \;
(\omega(n(b)) (\varphi))(x) := \ebf(b Q(x)) \varphi( x), \;
(\omega(m(a)) (\varphi))(x) := |a|^{2} \varphi( xa)
$$
for $n(b) := \smat{1}{b}{0}{1}, m(a) := \smat{a}{0}{0}{a^{-1}} \in \SL_2(\Rb)$, $h \in \mathrm{O}(V)$, $\varphi \in \Ss(V)$ and $x \in V$.

The Lie algebra $\spf(\Wb \otimes \Cb)$ now also acts on vectors in $\Ss(V)$ through the infinitesimal action $d \omega$ induced by $\omega$, which satisfies
$$
\domega(A) \omega(g) \varphi := \partial_t \omega(g e^{tA}) \varphi \mid_{t = 0} = \omega(g) \domega(A) \varphi.
$$
for any $g \in \Sp(\Wb)$ and $A \in \spf(\Wb)$.
For $1 \le r \le 4$, define the following operators on $\Ss(V)$
\begin{equation}
  \label{eq:Dj}
D_r := \partial_{x_r} - 2\pi x_r.  
\end{equation}
One can interpret $\spf(\Wb)$ as quotients of graded pieces of the Weyl algebra $\mathscr{W}_{2\pi i}$ of $\Wb$, and view $D_j$ as the actions by generators of $\mathscr{W}_{2\pi i}$ (see \cite[section 6]{KM90}). The subspace $\mathbb{S}(V) \subset \Ss(V)$ spanned by the Gaussian 
$$
\varphi_0(x_1, x_2, x_3, x_4) := e^{-\pi(x_1^2 + x_2^2 + x_3^2 + x_4^2)}
$$
and functions of the form $\prod_{1 \le j \le 4} D_j^{r_j} \varphi_0$ for $r_j \in \Nb$ is called the \textit{polynomial Fock space}.

It is easier to describe the infinitesimal action of $\spf(\Wb)$ if we choose the polarization $\Wb \otimes \Cb = \Wb' + \Wb''$ with 
$$
\Wb' := V_+ \otimes W' + V_- \otimes W'', \,
\Wb'' := V_- \otimes W' + V_+ \otimes W''
$$
and switch to the Fock model of the Weil representation.
The underlying vector space that $\Sp(\Wb)$ acts on now is $\mathrm{Sym}^\bullet(\Wb'')$ and we also use $\omega$ to represent this action.
We introduce the linear functionals $\{\zf_1, \zf_2, \zf_3, \zf_4\}$ on $\mathrm{Sym}^\bullet(\Wb'')$ given by
\begin{equation}
  \label{eq:zr}
\zf_r(v \otimes w) :=
  \begin{cases}
     \langle v \otimes w , v_r \otimes w'' \rangle, \, &r = 1, 2, \\
     \langle v \otimes w, v_r \otimes w' \rangle, \, &r = 3, 4,
  \end{cases}
\end{equation}
 which identifies $\mathrm{Sym}^\bullet(\Wb'')$ with $\mathscr{P}(\Cb^4)$, the algebra of polynomials in $\zf_1, \zf_2, \zf_3, \zf_4$. For convenience, we also denote
 \begin{equation}
   \label{eq:wf}
\wf := \zf_3 - i\zf_4 
 \end{equation}
in $\mathscr{P}(\Cb^4)$.
Furthermore, there is a unique intertwining operator $\iota$ from $\mathbb{S}(V)$ to $\mathscr{P}(\Cb^4)$ that sends $\varphi_0$ to 1. The following lemma describes $d \omega$ in the Fock model and the effect of $\iota$.

\begin{lemma}[Lemma A.1-A.3 \cite{FM06}]
\label{lemma:KM}
In the notations above, the following elements in $\slf_2(\Cb) \cong \spf(W) \subset \spf(\Wb)$ acts on $\mathscr{P}(\Cb^4)$ as
\begin{equation}
  \label{eq:RLtau}
  \domega(L) = - 2\pi (\partial_{\zf_1}^2 + \partial_{\zf_2}^2)  + \frac{1}{8\pi} \wf \overline{\wf}, \;
  \domega(R) = - 8\pi \partial_{\wf} \partial_{\overline{\wf}} + \frac{1}{8\pi} (\zf_1^2 + \zf_2^2).
\end{equation}
Recall that $\iota_\Delta: \SL_2(\Rb) \hookrightarrow \SO(V)$ is the map in \eqref{eq:hgamma}. Then 
\begin{equation}
  \label{eq:RLz}
  \begin{split}
      \domega(\iota_{\Delta *}(L)) &= \domega(-X_{13} - iX_{14}) = 8\pi \partial_{\zf_1} \partial_{{\wf}} - \frac{1}{4\pi} \zf_1 \overline{\wf}, \\
  \domega(\iota_{\Delta *}(R)) &= \domega(-X_{13} + iX_{14}) = 8\pi \partial_{\zf_1} \partial_{\overline{\wf}} - \frac{1}{4\pi} \zf_1 {\wf}.
  \end{split}
  \end{equation}
Under the intertwining map $\iota: \mathbb{S}(V) \to \mathscr{P}(\Cb^4)$, the operator $D_r$ acts as follows
\begin{equation}
  \label{eq:Djiota}
      \iota D_r \iota^{-1} =
  \begin{cases}
 i z_r, \, &r = 1, 2, \\
-i z_r, \, &r = 3, 4.
  \end{cases}
\end{equation}
\end{lemma}

For $\tau = u + iv \in \Hb$, let $g_\tau \in \SL_2(\Rb)$ be any element to takes $i$ to $\tau$. In particular, any $g_\tau$ differs from $n(u)m(\sqrt{v})$ by multiplying an element in the maximal compact $K:= \SO_2(\Rb) \subset \SL_2(\Rb)$ on the right.
For $a, b, c \in \Nb$, we define a family of polynomials $p_{a, b, c}$ on $V = \Rb^{2, 2}$ by
\begin{equation}
  \label{eq:pabc}
  p_{a, b, c}(\underline{x}) := (ix_1)^a x_2^b (ix_3 + x_4)^c.
\end{equation}
It is homogeneous of degrees $a + b$ and $c$ in the first two and last two variables respectively. From this, we can construct a Schwartz function 
$$
\varphi_{a, b, c}(\underline{x}) := \varphi(\underline{x}; p_{a, b, c}) \in \mathbb{S}(V) \subset \Ss(V).
$$
By \eqref{eq:hermite} and Lemma \ref{lemma:KM}, we have
\begin{equation}
  \label{eq:iotaabn}
  \iota( \varphi_{a, b, c}) = (-4\pi)^{-a-b-c} i^{2a + b} \zf_1^a \zf_2^b {\wf}^c.
\end{equation}
Under the action of $\omega(r_\theta)$ with $r_\theta := \smat{\cos \theta}{\sin \theta}{-\sin \theta}{\cos \theta} \in K$, the function $\varphi_{a, b, c}$ is an eigenvector with eigenvalue $e^{i k \theta}$ and $k := a + b - c$.
For $z \in \Hb$, the isometry $\nu_z: V \to \Rb^{2, 2}$ in \eqref{eq:nu} is given by applying the map $h(g_z) \in \SO(V)$ in \eqref{eq:hgamma}.
Therefore, the Schwartz function $\varphi(\lambda; \tau, z; p_{a, b, c})$ in \eqref{eq:varphi2} can be written as
\begin{equation}
  \label{eq:varphiweil}
\varphi(\lambda; \tau, z; p_{a, b, c}) = v^{-k/2} (\omega(g_\tau) \omega(h(g_z)^{-1}) \varphi_{a, b, c})(\lambda).  
\end{equation}
For $\kappa \in \Zb$, recall that $R_{\tau, \kappa}$ and $L_{\tau, \kappa}$ are the raising and lowering operators. 
Using Lemma \ref{lemma:KM}, we can calculate their effects on $\varphi(\lambda; \tau, z; p_{a, b, c})$ as follows.
\begin{lemma}
  \label{lemma:LRrelation}
For any $a, b, c \in \Nb$, let $k := a + b - c$. Then we have
\begin{equation}
\label{eq:LRrelation}
\begin{split}
\frac{1}{2\pi}  R_{\tau, k} \varphi(\lambda; \tau, z; p_{a, b, c}) &=  \varphi(\lambda; \tau, z; p_{a+2, b, c}) - \varphi(\lambda; \tau, z; p_{a, b+2, c}), \\
R_{z, 2c} y^{-c} \varphi(\lambda; \tau, z; p_{a, b, c}) &= 4\pi y^{-c-1} \varphi(\lambda; \tau, z; p_{a + 1, b, c + 1}), \\
2 L_{\tau, k} y^{-c} \varphi(\lambda; \tau, z; p_{a+1, b, c}) &= 
\varphi(\lambda; \tau, z; (2c + 1 - a) a p_{a-1, b, c} + b(b-1)p_{a+1, b-2, c})\\
& \quad + L_{z, 2c+2} y^{-c-1} \varphi(\lambda; \tau, z; p_{a, b, c+1})
\end{split}
\end{equation}
for all $\lambda \in V$.
\end{lemma}

\begin{proof}
Using the relation between the differential operator $R_{\tau, k}$ and the action of the Lie algebra element $R$, it is enough to calculate the effect of $\domega(R), \domega(L)$ and $\domega(\iota_{\Delta*}(L))$ on $\varphi_{a, b, c}$.
For example, the equations
\begin{align*}
R_{\tau, k} &\varphi(\lambda; \tau, z; p_{a, b, c}) =   
v^{-1-k/2} (\omega(g_\tau) \omega(h(g_z)^{-1}) \domega(R) \varphi_{a, b, c})(\varphi), \\
L_{\tau, k} &\varphi(\lambda; \tau, z; p_{a, b, c}) =   
v^{1-k/2} (\omega(g_\tau) \omega(h(g_z)^{-1}) \domega(L) \varphi_{a, b, c})(\varphi), \\
R_{z, 2c} &y^{-c} \varphi(\lambda; \tau, z; p_{a, b, c}) =   
- y^{-1-c} (\omega(g_\tau) \omega(h(g_z)^{-1}) \domega(\iota_{\Delta *}(R)) \varphi_{a, b, c})(\varphi),\\
L_{z, 2c} &y^{-c} \varphi(\lambda; \tau, z; p_{a, b, c}) =   
- y^{1-c} (\omega(g_\tau) \omega(h(g_z)^{-1}) \domega(\iota_{\Delta *}(L)) \varphi_{a, b, c})(\varphi)
\end{align*}
follows from \eqref{eq:varphiweil}.
The negative sign in the last equation is a result of the inversion action by $h(g_z)$ in \eqref{eq:varphiweil}.
From \eqref{eq:RLtau}, \eqref{eq:RLz} and \eqref{eq:iotaabn}, we can easily calculate that
\begin{align*}
  \domega(R)\varphi_{a, b, c} &= 2\pi (\varphi_{a+2, b, c} - \varphi_{a, b+2, c}), \\
-2 \domega(L) \varphi_{a+1, b, c} &= \domega(\iota_{\Delta*}(L)) \varphi_{a, b, c+1} - (2c + 1 - a)a \varphi_{a-1, b, c} - b(b-1) \varphi_{a+1, b-2, c}.
\end{align*}
These then imply the lemma.
\end{proof}

\subsection{Legendre Polynomials}
For $n \ge 0$, the $n$\tth Legendre polynomial $P_n(x)$ is a polynomial solution to Legendre's differential equation. It can be defined from the recursion
$$
(n+1) P_{n+1}(x) = (2n + 1) x P_n(x) - nP_{n-1}(x), \; P_0(x) = 1, P_1(x) = x,
$$
and satisfies 
\begin{equation}
  \label{eq:Taylor}
  \sum_{n = 0}^\infty P_n(x) t^n = \frac{1}{\sqrt{ 1 - 2xt + t^2}}.
\end{equation}
Furthermore, it has the explicit representation
\begin{equation}
  \label{eq:Pnexplicit}
  P_n(x) =  \sum_{b = 0}^n c_{n, b}\cdot x^b, \quad
c_{n, b} := 2^n \binom{n}{b} \binom{\tfrac{n+b-1}{2}}{n} \in \Zb,
\end{equation}
where $\binom{a}{m} := \frac{a(a-1)\dots (a-m+1)}{m!}$ for any $m \in \Nb$ and $a \in \Rb$.
Since $\tfrac{n+b-1}{2} \in \Nb$ is less than $n$ when $b \equiv n + 1 \bmod{2}$, we have $P_n(-x) = (-1)^n P_n(x)$. 
From \eqref{eq:Taylor}, it is clear that $P_n(1) = 1$ for all $n \ge 0$.
Using the Legendre polynomial, we can now define the following polynomial for each $n \in \Nb$ on $\Rb^{2, 2}$. 
\begin{equation}
  \label{eq:keypols}
  \begin{split}
\tp_n &:= P_n \lp \frac{x_2}{ix_1} \rp (ix_1)^n (ix_3 + x_4)^{n+1}, \quad
\hp_n := P_n \lp \frac{x_2}{ix_1} \rp (ix_1)^{n+1} (ix_3 + x_4)^{n}.
  \end{split}
\end{equation}
Note that we can write $\tp_n$ and $\hp_n$ as a linear combinations of $p_{a, b, c}$ as
\begin{equation}
\label{eq:lincomb}
  \tp_n := 
 \sum_{b = 0}^n c_{n, b} p_{n - b, b, n+1},  \quad \hp_n := 
 \sum_{b = 0}^n c_{n, b} p_{n-b+1, b, n}.
 \end{equation}
Using these two polynomials, we can form a theta function as above. The main result is as follows.
\begin{prop}
  \label{prop:diffswitch}
For $n \in \Nb$, let $\tp_n$ and $\hp_n$ be the polynomials on $\Rb^{2, 2}$ defined in \eqref{eq:keypols}. Then the corresponding Schwartz functions $\varphi(\cdot; \tau, z; \tp_n)$ and $\varphi(\cdot; \tau, z; \hp_n)$ satisfy the following relation
\begin{equation}
  \label{eq:diffswitch}
 L_{z, 2n+2} y^{-n-1} \varphi(\lambda; \tau, z; \tp_n)  =
2 L_{\tau, 1} y^{-n} \varphi(\lambda; \tau, z; \hp_n) 
\end{equation}
for all $\tau, z \in \Hb$ and $\lambda \in V$.
\end{prop}
\begin{proof}
  From the third equation in lemma \ref{lemma:LRrelation} and \eqref{eq:lincomb}, we see that it suffices to prove that
$$
\sum_{b = 0}^{n} 
c_{n, b}
\lp (2n + 1 - (n-b))(n-b)  p_{n-b-1, b, n} + b(b-1)p_{n-b+1, b-2, n}
\rp= 0.
$$
The left hand side can be rewritten as
$$
\sum_{b = 0}^{n-2} 
\lp c_{n, b} (n+1+b)(n-b) +
c_{n, b+2} (b+2)(b+1) \rp
p_{n-b-1, b, n}
$$
which vanishes identically for all $0 \le b \le n-2$ from the definition of $c_{n, b}$ in \eqref{eq:Pnexplicit}.
If $n$ is odd, then there is an extra 
\end{proof}

As an immediate consequence of the definition in \eqref{eq:Theta}, the theta function formed from $p_{a, b, c}$, $\hp_n$ and $\tp_n$ satisfy the same differential equation.

\begin{cor}
  \label{cor:diffswitch}
Let $M \subset V$ be any even, integral lattice of full rank.
Then
\begin{equation}
  \label{eq:diffswitchTheta}
  \begin{split}
    \frac{1}{2\pi}  R_{\tau, a+b-c} \Theta_M(\tau, z; p_{a, b, c}) &=  \Theta_M(\tau, z; p_{a+2, b, c}) - \Theta_M(\tau, z; p_{a, b+2, c}), \\
R_{z, 2c} y^{-c} \Theta_M(\tau, z; p_{a, b, c}) &= 4\pi y^{-c-1} \Theta_M(\tau, z; p_{a + 1, b, c + 1}), \\
 L_{z, 2n+2} y^{-n-1} \Theta_M(\tau, z; \tp_n)  &=
2 L_{\tau, 1} y^{-n} \Theta_M(\tau, z; \hp_n)
  \end{split}
\end{equation}
for any $a, b, c, n\in \Nb$ and $\tau, z\in \Hb$.
\end{cor}

\bibliography{HGreen}{}
\bibliographystyle{amsplain}

\end{document}